\documentclass[11pt]{article}
\usepackage{amsfonts}
\usepackage{amsmath}
\usepackage{amsthm}
\usepackage{amssymb}
\usepackage{mathrsfs}
\usepackage[numbers]{natbib}
\usepackage[fit]{truncate}
\usepackage{fullpage}
\usepackage{mathtools}
\usepackage{makecell}
\usepackage{subfig}
\usepackage{hyperref}
\usepackage{dsfont}

\newcommand{\N}{\mathbb{N}}
\newcommand{\R}{\mathbb{R}}
\renewcommand{\P}{\mathbb{P}}
\newcommand{\E}{\mathbb{E}}
\newcommand{\D}{\mathbb{D}}

\theoremstyle{plain}
\newtheorem{theorem}{Theorem}[section]
\newtheorem{corollary}[theorem]{Corollary}
\newtheorem{lemma}[theorem]{Lemma}

\newtheorem{proposition}[theorem]{Proposition}

\theoremstyle{definition}
\newtheorem{definition}[theorem]{Definition}
\newtheorem{example}[theorem]{Example}

\begin{document}

\title{Asymptotics of generalized Pólya urns with non-linear feedback}
\author{Thomas Gottfried$^*$ and Stefan Grosskinsky\footnote{Stochastics and Its Applications, Institute of Mathematics, University of Augsburg}}
\date{\today}
\maketitle

\begin{abstract}
Generalized Pólya urns with non-linear feedback are an established probabilistic model to describe the dynamics of growth processes with reinforcement, a generic example being competition of agents in evolving markets. It is well known which conditions on the feedback mechanism lead to monopoly where a single agent achieves full market share, and various further results for particular feedback mechanisms have been derived from different perspectives. In this paper we provide a comprehensive account of the possible asymptotic behaviour for a large general class of feedback, and describe in detail how monopolies emerge in a transition from sub-linear to super-linear feedback via hierarchical states close to linearity. We further distinguish super- and sub-exponential feedback, which show conceptually interesting differences to understand the monopoly case, and study robustness of the asymptotics with respect to initial conditions, heterogeneities and small changes of the feedback mechanisms. 
Finally, we derive a scaling limit for the full time evolution of market shares 
in the limit of diverging initial market size, including the description of typical fluctuations and extending previous results in the context of stochastic approximation.
\end{abstract}


\section{Introduction}

In the near future, customers who intend to buy a new car will have the choice between several different technologies like modern cars powered by fossile or synthetic fuels, hydrogen or batteries. Although electric cars seem to be in the pole position in the race for the future car market, it is still open which technology will win or whether there will be a mixture of different technologies. The economist Brian R. Arthur suggests in \cite{Arthur} to model the competition between technologies as a generalized Pólya urn, which was basically introduced by Hill, Lane and Sudderth in \cite{Hill}. In this model the decision which technology to choose depends on three factors. First, it supposes that each technology has an intrinsic deterministic attractiveness or fitness. Second, the decision depends on the choice of earlier customers. For example, if many bought an electric car before, there will be a dense charging infrastructure and thus electric cars get more attractive for future customers. A second argument for this reinforcement is that high revenues in the past provide financial means for a faster technological development as well as cheaper prices because of lower production costs per unit. The resulting overall attractiveness of technology $i$ is now modeled as a hypothetical feedback-function $F_i(X_i)\ge0$ depending on the number $X_i\in\N=\{1,2,3\ldots\}$ of customers, who chose technology $i$ before. High values of $F_i(X_i)$ indicate high attractiveness of technology i. A typical example is $F_i(k)=\alpha_ik^\beta$, where $\alpha_i>0$ models the intrinsic attractiveness and $\beta>0$ the reinforcement effects in the market. The third determinant of customers decision is their personal preference, which is difficult to include in a deterministic model and probabilistic approaches are more appropriate. We assume that customers enter the market sequentially and have full information. Given the current state $(X_1 ,\ldots ,X_A )$ of the market, a customer will opt for technology $i$ with probability 
\begin{equation*}
\frac{F_i(X_i)}{F_1(X_1)+...+F_A(X_A)},
\end{equation*}
where $A\geq 2$ is the number of different technologies. The market size $X_1 +\ldots +X_A$ increases by one in each step. If $F_i(k)=k$, then this corresponds to the original Pólya urn, which was introduced by Pólya and Eggenberger in \cite{Polya}. 
Depending on the feedback function, monopoly 
may occur where one technology achieves full market share, as well as random or deterministic non-zero asymptotic market shares for several technologies. The monopolist is in general random and depends on the behaviour of the young market. 
Analyzing which feedback function leads to which regime provides an understanding of the determinants of the long-time behavior of markets.

Mathematically, this setup corresponds to a discrete-time Markov process, which is called a (generalized) non-linear Pólya urn in the following and introduced in detail in the next Section. 
Apart from the competition of technologies, many other interpretations and applications of generalized Pólya urns are possible. An obvious one is the competition of companies in the same market for new customers or the competition between regions for new companies to settle. 
The dynamics of household wealth is another growth process with reinforcement (see e.g. \cite{forbes} and references therein) that can be modelled with urns. 
\cite{Pemantle} summarizes further applications in psychology or evolutionary biology, and more recently, \cite{Tang,Rocsu} use Pólya urns in the context of cryptocurrencies. In the following we will adapt the more general terminology of agents $\{ 1,\ldots ,A\}$ instead of technologies.

Mathematical properties of non-linear Pólya urns  have been examined before, often focused on polynomial feedback functions \cite{Kearney,Drinea,Khanin,Oliveira3,Hill,Jiang,Chung,Menshikov} or homogeneous models with $F_i \equiv F$ 
\cite{Oliveira,Oliveira2,Mitzenmacher}. In applications, the feedback functions are usually a hypothetical construction that can barely be measured in real systems similar to utility functions in economic situations, thus a general mathematical understanding without restrictive conditions on $F_i$ is important. This paper investigates the long-time behavior of non-linear Pólya urns for a very general class of feedback functions. $F_i$ could even be decreasing or exponentially increasing, which reveals some surprising differences to the usually studied polynomial case. An important restriction is, however, that $F_i$ depends only on $X_i$, which excludes stationary limit cycles as studied e.g. in \cite{Costa}. 

In Section \ref{sec: model} we introduce the model, give a detailed summary of previous related results and highlight the main novelties of the paper. In the monopoly case, we present in Section \ref{sec: monopoly} an asymptotic result for large initial market sizes on the distribution of the winner, extending previous results for particular feedback functions. In the non-monopoly case we present in Section \ref{sec: non-monopoly} a novel approach to compute the deterministic long-time market shares, which do not depend on the initial condition or early dynamics. In Section \ref{sec: specialCase}, we study in detail the transition between both cases for almost linear feedback functions, which are particularly relevant in  various applications including wealth dynamics \cite{forbes}. Moreover, we derive in Section \ref{sec:dynamics} a law of large numbers for the dynamics of the process for large initial market size, which is asymptotically described by an ordinary differential equation and has previously been studied for particular feedback functions in the context of stochastic approximation \cite{Arthur2,Pemantle,Ruszel}. Extending these results, we also establish a functional central limit theorem to describe typical dynamic fluctuations by a system of SDEs in Section \ref{sec: clt}. The question of a Gaussian approximation of the dynamics of Pólya urns has also been addressed in recent research, see \cite{Borovkov} and \cite{Dean}. Predictable behaviour can only be expected for large initial market size, the behavior of very young markets is intrinsically random. While bounds on the probabilities of certain events can be obtained, we focus here mostly on asymptotic results and provide a rather complete account of the possible dynamic and long-time behaviour of generalized non-linear Pólya urns. To our knowledge this paper provides the first complete account for the generalized non-linear P\'olya urn, which is a classical model for reinforcement dynamics.

More generalisations of Pólya's urn (e.g. for infinitely many agents and more complex replacement mechanisms) have been addressed in further recent research \cite{maulik2022feedback,Costa,Bandyopadhyay,Janson,Benaim,Kolesko, Ruszel, Singh, qin}, and in \cite{Aletti1,Aletti2,Aletti3} the authors include a rescaling mechanism to inhibit long-range dynamic dependencies. The study of generalized Pólya urns is also closely related to reinforced random walks, see e.g. \cite{Davis, Toth, Cotar2}

\section{The generalized Pólya urn model}
\label{sec: model}

\subsection{Basic definitions and background}

We now formally introduce the model. All random variables are defined on some large enough probability space $[\Omega, \mathcal{A}, \P]$. Let $A\ge2$ be the number of agents and $F_i\colon\N\to(0, \infty)$ the feedback function of agent $i\in [A]\coloneqq\{ 1,\ldots, A\}$. We define a homogeneous, discrete-time Markov process $(X(n))_{n\in\N_0}=((X_1(n),\ldots, X_A(n))_{n\in\N_0}$ on the state space $\N^A$ with initial condition $X(0)=(X_1(0),\ldots,X_A(0))\in\N^A$ such that $X_i (0)\geq 1$ for all $i\in [A]$, and transition probabilities
\begin{equation}\label{eq:trapro}
\P\left(X(n+1)=X(n)+e^{(i)}\, \big|\, X(n)\right)=\frac{F_i(X_i(n))}{F_1(X_1(n))+...+F_A(X_A(n))},\, i=1,\ldots,A,
\end{equation}
where $e^{(i)}=(\delta_{i,j})_{j=1}^A$ is the $i$-th unit vector. We denote by $N\coloneqq X_1(0)+...+X_A(0)\geq A$ the initial market size. Whenever needed, we set $F_i(0)=0$, and whenever useful, we take continuously differentiable extensions $F_i :(0,\infty )\to (0,\infty )$ to the positive real line, which is supposed to be monotone on intervals of the form $[n, n+1], n\in\N$. 

We interpret $X_i(n)$ as the number of customers of agent $i$ at time $n$ and define the corresponding time-inhomogeneous Markov process $(\chi (n))_{n\in\N_0 }$ of market shares
\begin{equation*}
\chi_i(n)\coloneqq\frac{X_i(n)}{N+n}\in(0, 1),\,i=1,\ldots, A,\, n\in\N_0 \ ,
\end{equation*}
with $\chi(n)=(\chi_1(n),\ldots,\chi_A(n))\in\Delta_{A-1}^o$, where $\Delta_{A-1}^o$ is the interior of the unit simplex $\Delta_{A-1}\coloneqq\{(x_1,\ldots, x_A)\in[0, 1]^A\colon x_1+...+x_A=1\}$. Moreover, we establish the notation
$$
\chi(\infty)\coloneqq\lim_{n\to\infty}\chi(n)$$
for the long time market share whenever it exists. We will see throughout the paper that $\chi(\infty)$ is well defined in all generic situations, but it is possible to construct counterexamples (see Example \ref{ex: noConv}). For later use we introduce the notation
\begin{equation}\label{eq:transpr}
p(k, x)=(p_i(k, x))_{i\in[A]}=\left(\frac{F_i(kx_1)}{F_1(kx_1)+\ldots F_A(kx_A)}\right)_{i\in[A]}
\end{equation}
for the transition probabilities, where $k\in\N$ and $x=(x_1,\ldots,x_A)\in\Delta_{A-1}$. Figure \ref{figure: Simulation} shows three simulations of this process for different feedback functions.

\begin{figure}
  \centering
  \subfloat[][$F_i(k)=k^2,\, i=1, 2, 3$]{\includegraphics[width=0.5\linewidth]{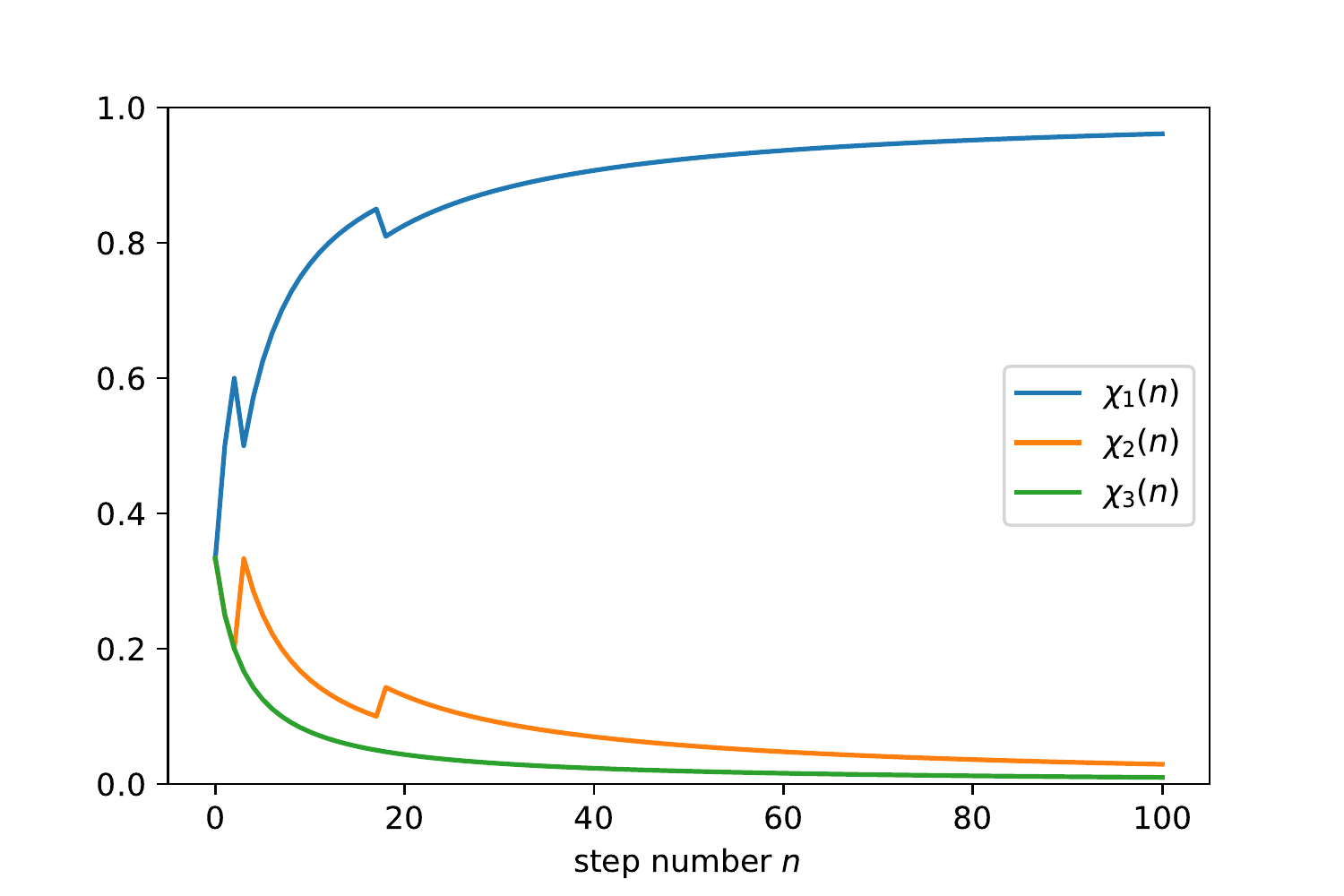}}%
  \subfloat[][$F_i(k)=k,\, i=1, 2, 3$]{\includegraphics[width=0.5\linewidth]{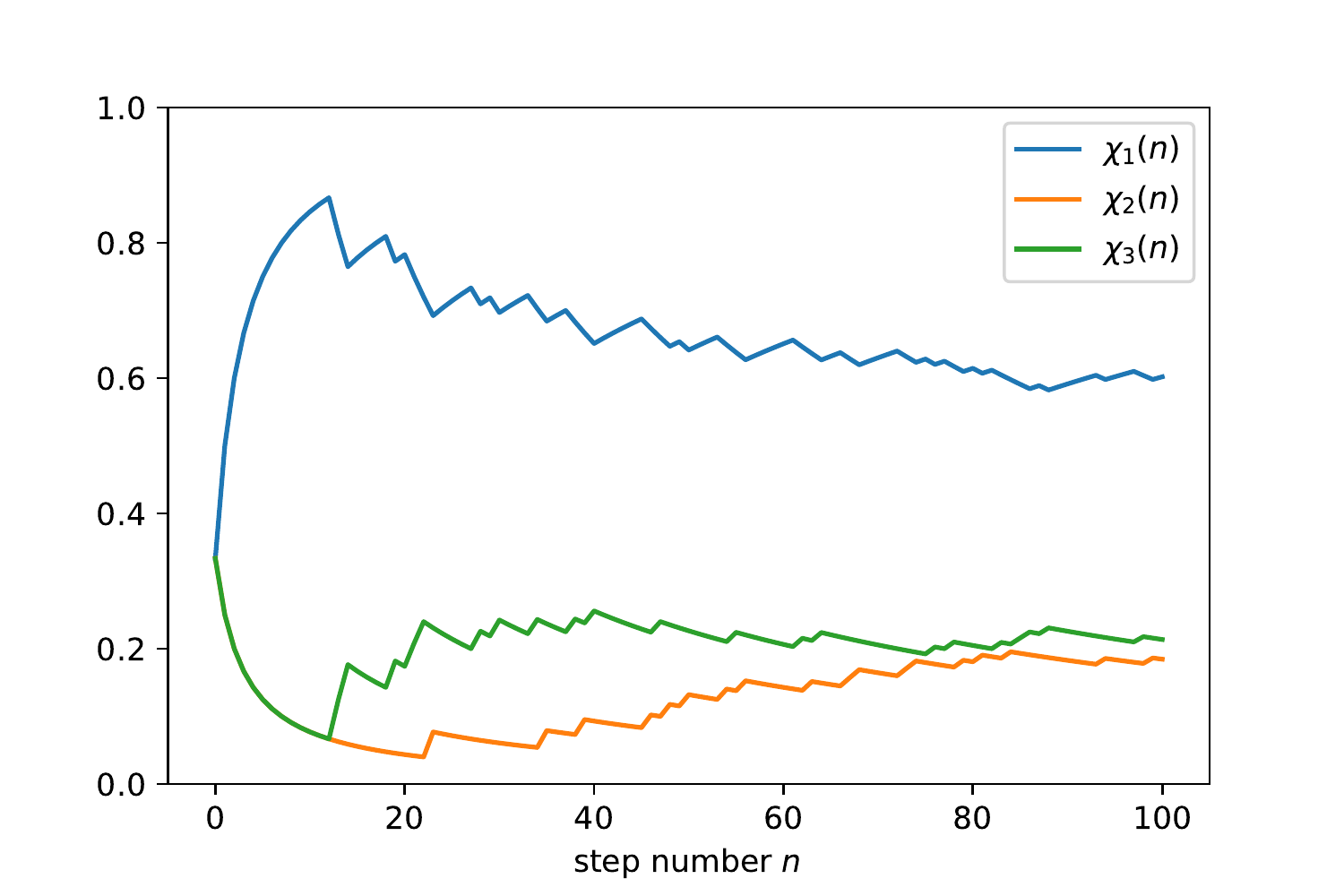}}%
  \quad
  \subfloat[][$F_i(k)=\sqrt{k},\, i=1, 2, 3$]{\includegraphics[width=0.5\linewidth]{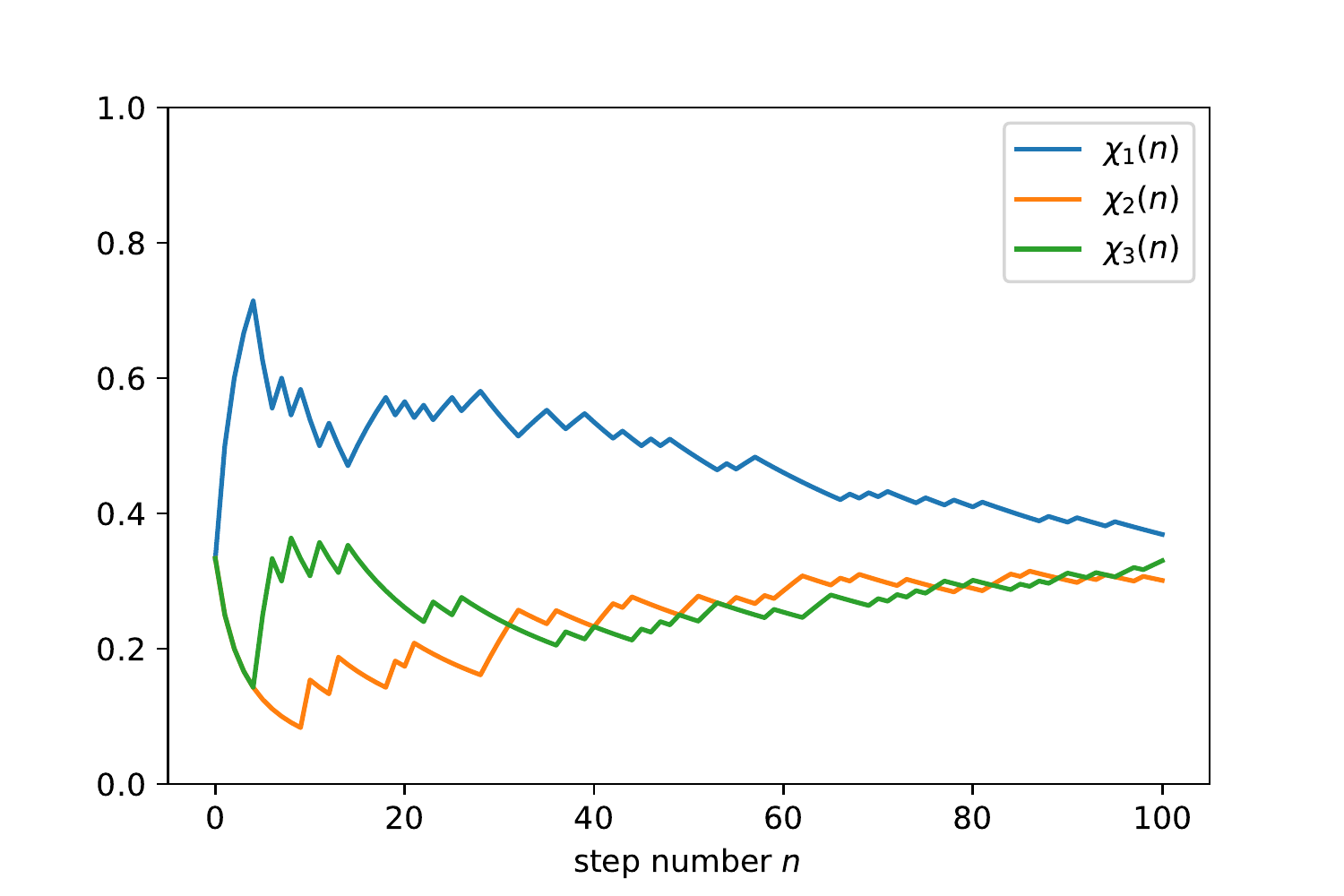}}%
  \subfloat[][$F_i(k)=k\log(k+1),\, i=1, 2, 3$]{\includegraphics[width=0.5\linewidth]{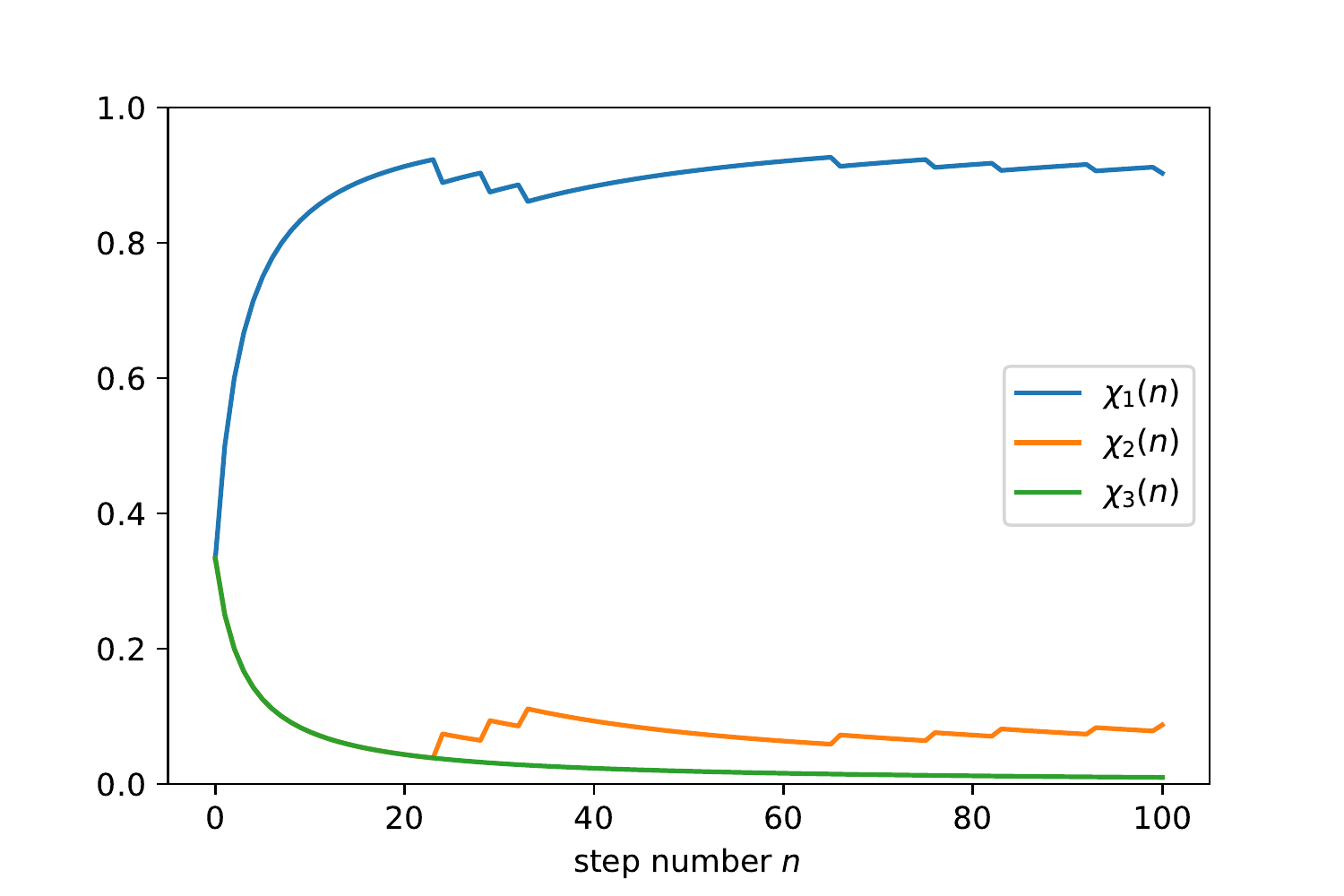}}%
  \caption{Simulated evolution of the market shares for the first 100 steps of a generalized Pólya urn with different feedback functions. Here $A=3$ and $X(0)=(1, 1, 1)$.}
  \label{figure: Simulation}
\end{figure}

A useful alternative construction of the process is provided by the so-called exponential embedding (see e.g. \cite{Oliveira} and references therein). We take independent random variables $\tau_i(k),\, i=1,\ldots, A,\, k\in\N$, where $\tau_i(k)$ is exponentially distributed with rate parameter $F_i(k)$. For each $i$ we define the corresponding continuous-time counting process $\big( \Xi_i(t)\big)_{t\geq 0}$ with
\begin{equation}\label{eq:counting}
\Xi_i(t)=\Xi_i^{(X_i(0))}(t)\coloneqq\max\left\{l\in\N_0: \sum_{k=0}^l\tau_i(X_i(0)+k)\le t\right\}+X_i(0),\, t\ge0.
\end{equation}
These are independent birth processes with $\Xi_i (0)=X_i (0)$, where the time between the $k$-th and $(k+1)$-th event of $\Xi_i$ is given by $\tau_i(X_i(0)+k)$. If $0=t_0<t_1<t_2<...$ is the sequence of jump-times of the process $\Xi(t)=(\Xi_1(t),\ldots,\Xi_A(t))$, i.e.
\begin{equation*}
t_{n+1}=\min\left\{t>t_n: \Xi(t)\ne \Xi(t_n)\right\},
\end{equation*}
then Rubin's theorem (proven in e.g. \cite{Oliveira}) states, that the jump chain $\left(\Xi(t_n)\colon n\in\N_0\right)$ has the same distribution as the process $\left(X(n)\colon n\in\N_0\right)$. Thus we can define:
\begin{equation}
\label{eq: expemb}
X(n)\coloneqq\Xi(t_n)
\end{equation}
In fact, the birth processes $\Xi_i(t)$ can explode as the sum $\sum_{k=X_i(0)}^\infty\tau_i(k)$ might be finite. We therefore define the random \textit{explosion times}
\begin{equation*}
T_i(X_i(0))\coloneqq\sum_{k=X_i(0)}^{\infty}\tau_i(k)\in(0, \infty],\, i=1,\ldots, A.
\end{equation*}

In the following we are especially interested in the occurrence of monopoly, which requires some definitions.

\begin{definition}
For $i\in[A]$ we define the events
\begin{enumerate}
\item \textit{weak monopoly}
\begin{equation*}
	wMon_i(\chi(0), N)\coloneqq\left\{\omega\in\Omega: \lim_{n\to\infty}\chi_i(n)(\omega)=1\right\} =\left\{ \lim_{n\to\infty}\chi_i(n)=1\right\},
\end{equation*}
i.e. the market share of agent $i$ converges to one;
\item \textit{strong monopoly}
\begin{equation*}
sMon_i(\chi(0), N)\coloneqq\Big\{\lim_{n\to\infty} \sum_{j\neq i}X_j (n)<\infty\Big\},
\end{equation*}
i.e. agent $i$ wins in all but finitely many steps;
\item \textit{total monopoly}
\begin{equation*}
tMon_i(\chi(0), N)\coloneqq\Big\{ \forall n\ge 0\,\forall j\in [A]\setminus\{i\}: X_j(n)=X_j(0)\Big\},
\end{equation*}
i.e. agent $i$ wins in all steps.
\end{enumerate}
\end{definition}

Obviously, a total monopoly is also a strong monopoly and a strong monopoly always implies a weak monopoly. Via exponential embedding one can express the event $sMon_i(\chi(0), N)$ by the explosion times through
\begin{equation}
\label{eq: sMonEmbed}
sMon_i(\chi(0), N)=\bigcap_{j\ne i}\Big\{T_i(X_i(0))<T_j(X_j(0))\Big\}
\end{equation}
as equality of finite explosion times has probability zero (see below). With the observation
\begin{equation*}
T_i(X_i(0))<\infty\,\Leftrightarrow\,\E T_i(X_i(0))=\sum_{k=X_i(0)}^{\infty}\frac{1}{F_i(k)}<\infty\,\Leftrightarrow\,\sum_{k=1}^\infty\frac{1}{F_i(k)}<\infty,
\end{equation*}
one can easily derive the following generally known criterion for the occurrence of strong monopoly (see e.g. \cite{Davis, Oliveira}).

\begin{theorem}
\label{thm: theorem2}
Strong monopoly occurs with probability one, i.e. 
\begin{equation*}
\P\left(\bigcup_{i=1}^AsMon_i(\chi(0), N)\right\}=1,
\end{equation*}
if and only if
\begin{equation}
\label{eq: explosion}
\sum_{k=1}^{\infty}\frac{1}{F_i(k)}<\infty\quad\mbox{for at least one }i\ ,\tag{M}
\end{equation}
otherwise the probability is zero.
\end{theorem}

If (\ref{eq: explosion}) holds, the density of the explosion time $T_i(X_i(0))$ (computed in \cite{Zhu}) as a sum of exponential variables has support on the whole positive real line for all choices of $F_i$. So the probability of $sMon_i(\chi(0), N)$ is positive if and only if agent $i$ fulfills (\ref{eq: explosion}) and the monopolist is random among all agents $i\in [A]$ that satisfy (\ref{eq: explosion}). For the polynomial case $F_i(k)=\alpha_ik^\beta,\,\alpha_i>0,\, \beta\in\R,\,i\in[A]$, Theorem \ref{thm: theorem2} implies that strong monopoly occurs if and only if $\beta>1$.\\

On the other hand, when no agent fulfills (\ref{eq: explosion}), $X_i (n)\to\infty$ almost surely for all $i\in [A]$, and we have the following consistency property.

\begin{proposition}\label{prop: partialProcess}
       Assume that none of the $F_i$ satisfies (\ref{eq: explosion}). Define a 'partial' P\'olya urn process $\tilde X(n)$ for a subset $B\subset [A]$ of agents with the same feedback functions $F_i$ and initial condition $\tilde X(0)=(X_i(0) :i\in B)$. Then the process $\big(\tilde X(n)\big)_{n\in\N_0}$ can be identified as a (random) subsequence of $\big( X_i (n):i\in B\big)_{n\in\N_0}$.
\end{proposition}

\begin{proof}
    The independence property of the exponential embedding provides a canonical coupling of the processes $\tilde X$ and $X$. For that, define recursively $s_0=0$ and
    $$s_{n+1}\coloneqq\inf\{s>s_n\colon\exists i\in B\colon \Xi_i(s)\ne\Xi_i(s_n)\}\,.$$
    Note that $s_n<\infty$ is well defined for all $n\geq 0$, since none of the $F_i$ fulfill (\ref{eq: explosion}). Then set $\tilde X_i(n)\coloneqq\Xi_i(s_n)$, which directly implies the claim since $(s_n)$ is a subsequence of $(t_n)$.
\end{proof}

In particular, if one of the limits
$$
\tilde\chi (\infty ):=\lim_{n\to\infty}\left(\frac{\tilde X_i(n)}{\tilde X_1(n)+\ldots+\tilde X_{\tilde A}(n)}\right)_{i\in B},\ \chi^B (\infty ):=\lim_{n\to\infty}\left(\frac{ X_i(n)}{ X_1(n)+\ldots+X_{\tilde A}(n)}\right)_{i\in B}
$$
exists, then so does the other and both have the same distribution. This implies further neutrality of the limit $\chi (\infty)$ in the sense of \cite{James}, so that it has a (possibly degenerate) Dirichlet distribution on $\Delta_{A-1}$, whenever it exists. In the degenerate case, the Dirichlet distribution is either deterministic or concentrated on the vertices of $\Delta_{A-1}$ with $\P\left(\bigcup_{i=1}^AwMon_i(\chi(0), N)\right)=1$. This will be discussed in several examples in Sections \ref{sec: non-monopoly} and \ref{sec: specialCase}. Note that in the case of weak monopoly this corresponds to \textbf{hierarchical states}, where the asymptotic distribution among losing agents again exhibits a weak monopolist.

\subsection{Review of previous results}
\label{sec: literature}

As already described in the introduction, generalisations of Pólya urns have been studied in numerous papers. In this section, we shortly present a selection of results related to our work. To our knowledge, the most comprehensive result concerning the long time limit of the process $(\chi(n))_n$ of market shares is the following.

\begin{theorem}\cite[Theorem 3.1]{Arthur2}
\label{thm: Arthur}
Suppose that $p(x)\coloneqq\lim_{k\to\infty}p(k, x)$ (cf.\ \eqref{eq:transpr}) exists for all $x\in\Delta_{A-1}$ and that even
\begin{equation}
\label{eq: arthur}
    \sum_{k=1}^\infty\frac{\sup_{x\in\Delta_{A-1}}\|p(k, x)-p(x)\|}{k}<\infty
\end{equation}
holds. Moreover, assume that there is a twice differentiable Lyapunov function for the vector field $(G(x))_{x\in\Delta_{A-1}}=(p(x)-x)_{x\in\Delta_{A-1}}$. Then $\chi(n)$ converges almost surely for $n\to\infty$ and the limit is either in $\{x\in\Delta_{A-1}\colon G(x)=0\}$ or the border of a connected component of this set.
\end{theorem}

Note that a Lyapunov function does always exist in the case $A=2$ and when $p$ is differentiable with equal feedback functions for all agents. Moreover, \cite{Arthur2} shows under mild technical assumptions that each stable fixed point of the vector field $G$ is attained in the limit $n\to\infty$ with positive probability, whereas unstable fixed points are never attained. 

Theorem \ref{thm: Arthur} allows to compute the long time market shares in generic situations, like $F_i(k)=\alpha_ik^\beta$. Nevertheless, condition (\ref{eq: arthur}) is not fulfilled e.g. for $F_i(k)=\log(k)$ or $F_i(k)=e^k$.

In the monopoly case described in Theorem \ref{thm: theorem2}, the monopolist is in general random. Consequently, one is interested in the probability that a specific agent is the monopolist, at least in the limit $N\to\infty$. \cite{Mitzenmacher} derives such a result in a situation with only two symmetric agents.

\begin{theorem}\cite[Theorem 2]{Mitzenmacher}
\label{thm: oliveira}
Let $A=2$ and $F_1=F_2=F$. Assume that $F$ fulfills (\ref{eq: explosion}) and that 
\begin{equation*}
    \liminf_{x\to\infty}x\frac{d}{dx}\log F(x)>\frac{1}{2}\text{  and  } \lim_{x\to\infty}\frac{d}{dx}\log F(x)=0.
\end{equation*}
Moreover, suppose that there is a constant $C>0$ such  that for all $\epsilon\in(0, \frac{1}{2})$ and all $x>0$ large enough
\begin{equation*}
    \sup_{x\le t\le x^{1+\epsilon}}\left|\frac{t\frac{d}{dt}\log F(t)}{x\frac{d}{dx}\log F(x)}-1\right|\le C\epsilon
\end{equation*}
holds. Let $X(0)=\left(N+\lambda q(N), N-\lambda q(N)\right)$ for $N, \lambda>0$ and $q(a)\coloneqq\sqrt{\frac{a}{4a\frac{d}{da}\log F(a)-2}}$.  Then the probability of agent $1$ being the monopolist converges to $\Phi(\lambda)$ for $N\to\infty$, where $\Phi$ denotes the cumulative distribution function of the normal distribution.
\end{theorem}

For $F(x)=x^\beta, \, \beta>1$ these assumptions are fulfilled and $q(a)=\frac{\sqrt{a}}{\sqrt{4\beta-2}}$. Under similar assumptions as in Theorem \ref{thm: oliveira}, \cite{Oliveira} shows that the number of steps, in which the looser wins, has a heavy tailed distribution. Moreover, if $\chi_i(0)<\frac{1}{2}$ for an agent $i$, then $\P(sMon_i(\chi(0), N)$ is exponentially decreasing in $N$, i.e. the first steps of the process decide who wins. \cite{Dunlop} provides similar results for the asymmetric case $F_i(x)=\alpha_ix^\beta,\,i\in\{1, 2\}, \beta>1, \alpha_i>0$. More recently in \cite{Menshikov}, a result for polynomial feedback with different exponents was shown.

\begin{theorem}\cite[Theorem 2.2]{Menshikov}
Let $A=2$ and $F_i(k)=k^{\beta_i}$ with $1<\beta_1\le\beta_2$. Define the critical values
\begin{equation*}
    \alpha_{cr}=\frac{\beta_1-1}{\beta_2-1}\quad\text{and}\quad\nu_{cr}=\alpha_{cr}^{\frac{1}{\beta_2-1}}.
\end{equation*}
Morover, set $X(0)=(x,\, \nu x^\alpha+o(x^\alpha))$ for $\alpha\in(0, 1), \nu>0$.
\begin{enumerate}
    \item If either $\alpha<\alpha_{cr}$ or $\alpha=\alpha_{cr}$ and $\nu<\nu_{cr}$, then $\lim_{x\to\infty}\P(sMon_1(X(0))=1$.
    \item If either $\alpha>\alpha_{cr}$ or $\alpha=\alpha_{cr}$ and $\nu>\nu_{cr}$, then $\lim_{x\to\infty}\P(sMon_2(X(0))=1$.
\end{enumerate}
\end{theorem}

In addition, \cite{Menshikov} provides a result for the critical case $\alpha=\alpha_{cr},\,\nu=\nu_{cr}$. 

For $F_i(k)=k^\beta,\,i\in[A]$ with $\beta<1$, we know from Theorem \ref{thm: Arthur} that $\lim_{n\to\infty}\chi_i(n)=\frac{1}{A}$ almost surely for all $i\in[A]$ irrespective of the initial configuration $\chi(0)$. The rate of convergence is specified in \cite{Khanin}.

\begin{theorem}\cite[Propsition 3]{Khanin}
Let $F_i(k)=k^\beta$ for all $i\in[A]$ and $\beta\in(0, 1)$. 
\begin{enumerate}
    \item If $\frac{1}{2}<\beta<1$, then
    \begin{equation*}
        n^{1-\beta}\left(\chi(n)-\frac{1}{A}\right)\xrightarrow{n\to\infty}C\quad\text{almost surely}
    \end{equation*}
    for a random, nonzero vector $C$.
    \item If $0<\beta<\frac{1}{2}$ and $i\in[A]$, then
    \begin{equation*}
        \sqrt{n}\left(\chi_i(n)-\frac{1}{A}\right)\xrightarrow{n\to\infty}\mathcal{N}\left(0,\, \frac{A-1}{A^{1+2\beta}(1-2\beta)}\right)\quad\text{in distribution,}
    \end{equation*}
    where $\mathcal{N}$ denotes a Gaussian distribution.
    \item If $\beta=\frac{1}{2}$  and $i\in[A]$, then
    \begin{equation*}
        \sqrt{\frac{n}{\log(n)}}\left(\chi_i(n)-\frac{1}{A}\right)\xrightarrow{n\to\infty}\mathcal{N}\left(0,\, \frac{A-1}{A^2}\right)\quad\text{in distribution.}
    \end{equation*}
\end{enumerate}
\end{theorem}

The convergence in part 2 and 3 can be extended to the vector $\chi(n)$. Part 1 implies that the leading agent does only change finitely often. According to \cite[Theorem 1]{Oliveira2}, this happens in general if and only if
\begin{equation*}
    \sum_{k=1}^\infty\frac{1}{F(k)^2}<\infty,
\end{equation*}
where $F=F_1=\ldots=F_A$ fulfills $\liminf_{k\to\infty}F(k)>0$.

\subsection{Main contributions of this paper}

One important purpose of this paper is to provide a comprehensive approach and a complete picture for the asymptotics of the generalized P\'olya urn model, which applies for a large class of feedback functions. 
This allows us to fully characterize the emergence of monopoly in a transition from sub-linear to super-linear feedback, where the system exhibits interesting behaviour including hierarchical states and weak monopoly. We outline our main results  
for symmetric feedback $F_1=\ldots=F_A=F$, although most of the results in this paper do also hold in asymmetric situations:
\begin{enumerate}
    \item If $F$ satisfies (\ref{eq: explosion}), then the process exhibits strong monopoly (Theorem \ref{thm: theorem2} above). The monopolist is random, but can be predicted with high probability for large initial values, such that the space $\Delta_{A-1}$ can be dissected into explicitly computable attraction domains (Theorem \ref{thm: domains}).
    \item If $F$ does not satisfy (\ref{eq: explosion}), but $\lim_{k\to\infty}\frac{F(k)}{k}=\infty$ still holds, then the process exhibits weak monopoly with a random monopolist (Corollary \ref{cor: almostLinear}) and hierarchical states with weak monopoly among the losing agents (Proposition \ref{prop: partialProcess} above).
    \item If $\lim_{k\to\infty}\frac{F(k)}{k}\in(0, \infty)$, then $\chi (\infty)$ exists almost surely and has a non-degenerate Dirichlet distribution. This includes the classical Pólya urn (Corollary \ref{cor: almostLinear}).
    \item If  $F$ is sublinear, then $\chi(\infty )$ exists almost surely and is deterministic with limit given by Corollary \ref{cor: main4} (under mild, but necessary technical assumptions).
\end{enumerate}

These regimes react differently to unequal fitness of agents. For $F_i(k)=\alpha_i F(k)$ with $\alpha_i>0$ distinct, we will show the following properties:
\begin{enumerate}
    \item If $F$ satisfies (\ref{eq: explosion}), then the process still exhibits random strong monopoly with well-defined attraction domains, which continuously depend on $\alpha_i$. For exponentially increasing $F$, these domains do not depend on $\alpha_i$ (Theorem \ref{thm: typeP}, Corollary \ref{cor:44}).
    \item[2./3.] If $F$ does not satisfy (\ref{eq: explosion}), but $\lim_{k\to\infty}\frac{F(k)}{k}\in(0, \infty]$, then the agent with the largest fitness $\alpha_i$ is a deterministic weak monopolist (Proposition \ref{prop49}) and we have hierarchical states.
    \item[4.] If  $F$ is sublinear, then $\chi(\infty)$ still exists and is deterministic. The dependence of $\chi(\infty)$ on $\alpha_i$ can be either continuous (Corollar \ref{cor: main4}) or discontinuous (Proposition \ref{prop49}). For exponentially decreasing $F$, there is no dependence on $\alpha_i$ (Appendix \ref{sec: expDecreasingF}).
\end{enumerate}

Furthermore, we derive a scaling limit for the process of market shares in Theorem \ref{thm: dynamic} and characterize the fluctuations in the Functional Central Limit Theorem \ref{cor: fclt}. This part uses standard techniques from stochastic approximation, but we include it to provide a complete picture of the asymptotics of the generalized P\'olya urn.

\section{Asymptotics for the strong monopoly case}
\label{sec: monopoly}

We assume that at least one agent $i$ fulfills (\ref{eq: explosion}), so that a random strong monopoly occurs  with probability one. To characterize the asymptotics, we have to distinguish two different types of feedback functions with slightly different behavior.

\begin{definition}
Let agent $i$ (or $F_i$) fulfill (\ref{eq: explosion}). We call $i$ (or $F_i$) \textit{of type P} (for polynomial) if
\begin{equation}
\label{eq: typeP}
\lim_{k\to\infty}F_i(k)\sum_{l=k}^{\infty}\frac{1}{F_i(l)}=\infty
\end{equation}
and \textit{of type E} (for exponential) if
\begin{equation}
\label{eq: typeE}
\limsup_{k\to\infty}F_i(k)\sum_{l=k}^{\infty}\frac{1}{F_i(l)}<\infty.
\end{equation}
\end{definition}

For the rest of this section we assume that all agents with feedback functions that fulfill (\ref{eq: explosion}) are either of type P or type E. Of course it is possible to construct counter-examples (see Example \ref{counterexample}), but these two types still cover a very large range, including most previous results.

\begin{proposition}\label{prop: logConds}
If 
\begin{equation}
\label{eq: typePlog}
\frac{d}{dx}\log(F(x))\xrightarrow{x\to\infty}0
\end{equation}
then $F$ is of type P, and if 
\begin{equation}
\label{eq: typeElog}
\liminf_{x\to\infty}\frac{d}{dx}\log(F (x))>0
\end{equation}
then $F$ is of type E.
\end{proposition}

\begin{proof}
First we assume (\ref{eq: typePlog}) and observe that
\begin{equation}
\label{eq: incrementlog}
\frac{F(k+1)}{F(k)}=\exp\left\{\int_k^{k+1}\frac{d}{dx}\log(F(x)) dx\right\}\xrightarrow{k\to\infty}1.
\end{equation}
Consequently, for any given $\epsilon>0$ there exists $k_0$ such that $\forall k\ge k_0:F(k+1)/F(k)\le1+\epsilon$. Then we get for $k\ge k_0$:
\begin{equation*}
F(k)\sum_{l=k}^{\infty}\frac{1}{F(l)}=\sum_{l=k}^{\infty}\prod_{m=k+1}^l\frac{F(m-1)}{F(m)}\ge\sum_{l=k}^{\infty}\left(\frac{1}{1+\epsilon}\right)^{l-k}=\frac{1}{1-\frac{1}{1+\epsilon}}\xrightarrow{\epsilon\to0}\infty
\end{equation*}
The result for type E follows similarly.
\end{proof}

This means that functions that grow exponentially or faster are of type E whereas functions that grow slower than exponential (like polynomials) are of type P. Note that Oliveira's "valid feedback functions" in \cite{Oliveira} or \cite{Oliveira2} are of type P, which includes furthermore all regular varying functions. 

\begin{example} \label{counterexample}
\begin{enumerate}
    \item The conditions from Proposition \ref{prop: logConds} are not necessary for being type P resp. E. For instance take any function $F$ of type E and define $\tilde F(2k)=\tilde F(2k+1)=F(k)$. Then $\tilde F$ is also of type E, but does obviously not fulfill (\ref{eq: typeElog}).
    \item A possible construction of a feedback function that is neither of type P nor type E, but satisfies (\ref{eq: explosion}), is the following. Take a function $F$ such that
    \begin{equation*}
0<\lim_{k\to\infty}F(k)\sum_{l=k}^{\infty}\frac{1}{F(l)}<\infty.
\end{equation*}
holds, e.g. $F(k)=e^k$. Then define a new feedback function $\tilde F$ by replacing each $F(k)$ by $k$ elements that all equal $kF(k)$, i.e.
\[
\big(\tilde F(1),\tilde F(2),\ldots\big) =\Big( F(1),2F(2),2F(2),3F(3),3F(3),3F(3),\ldots \Big)\ .
\]
One can easily check that $\tilde F_i$ has the desired properties.
\end{enumerate}
\end{example}

\subsection{Asymptotic attraction domains}

If at least one agents fulfills the monopoly condition (\ref{eq: explosion}), we know by Theorem \ref{thm: theorem2} that there is a strong monopoly, where all agents satisfying (\ref{eq: explosion}) have a positive probability of being the monopolist. Thus, the monopolist is in general random. Nevertheless, in most situations it is possible to predict the winner with high probability for large initial market size.

\begin{definition}\label{def: domains}
    The \textit{asymptotic attraction domain} of an agent $i\in[A]$ is defined as
    $$D_i=\left\{\chi(0)\in\Delta_{A-1}^o\colon\lim_{N\to\infty}\P(sMon_i(\chi(0), N))=1\right\}\subset\Delta_{A-1}^o.$$
\end{definition}

Obviously, the asymptotic attraction domains are disjoint, since $\P(sMon_j(\chi(0), N))\le1-\P(sMon_i(\chi(0), N))$ for $j\ne i$. The main result of this section states that the asymptotic attraction domains cover the whole simplex up to boundaries under mild regularity conditions.

\begin{theorem}
\label{thm: domains}
    Let at least one agent satisfy (\ref{eq: explosion}) and all agents satisfying (\ref{eq: explosion}) are either of type P or type E. Moreover, assume that \textbf{one of the following} conditions holds:
    \begin{enumerate}
        \item At least one agent is of type E and for all $\chi(0)\in\Delta_{A-1}^o,\,i, j\in[A]$
        $$\liminf_{N\to\infty}\frac{F_i(\chi_i(0)N)}{F_j(\chi_j(0)N)}=0\quad\mbox{and}\quad\limsup_{N\to\infty}\frac{F_i(\chi_i(0)N)}{F_j(\chi_j(0)N)}=\infty\quad\mbox{do \textbf{not} hold simultaneously}\ .$$
        \item No agent is of type E and all agents of type P (there is at least one) fulfill 
        \begin{equation}
       \label{eq: existencePcond}
        \limsup_{k\to\infty}\frac{1}{k}F_i(k)\sum_{l=k}^\infty\frac{1}{F_i(l)}<\infty.
         \end{equation}
         In addition, suppose that $\lim_{N\to\infty}\frac{F_i(\chi_i(0)N)}{F_j(\chi_j(0)N)}\in[0, \infty]$ exists for all $\chi(0)\in\Delta_{A-1}^o, i, j\in[A]$.
    \end{enumerate}
    Then the asymptotic attraction domains are polytopes that dissect the simplex up to boundaries, i.e.
   $$\bigcup_{i=1}^A\overline{D_i}=\Delta_{A-1},$$
   where $\overline{(\cdot)}$ is the topological closure. If agent $i$ does not satisfy \eqref{eq: explosion} then $D_i =\emptyset$.
\end{theorem}
\begin{figure}
    \centering
  \subfloat[][$F_1(k)=F_2(k)=F_3(k)=k^2$]{\includegraphics[width=0.3\linewidth]{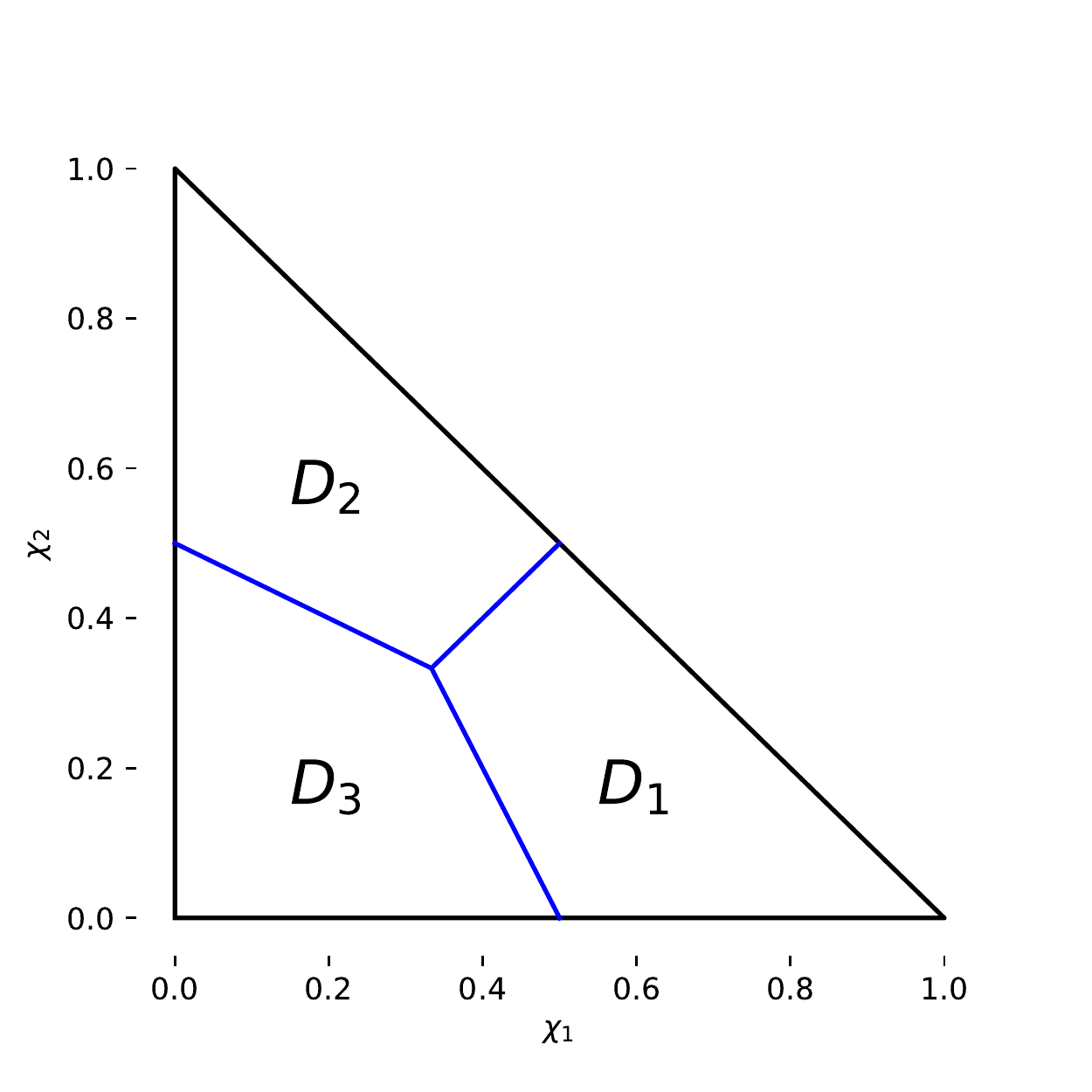}}\quad
  \subfloat[][$F_1(k)=F_2(k)=2k^2$, $F_3(k)=k^2$]{\includegraphics[width=0.3\linewidth]{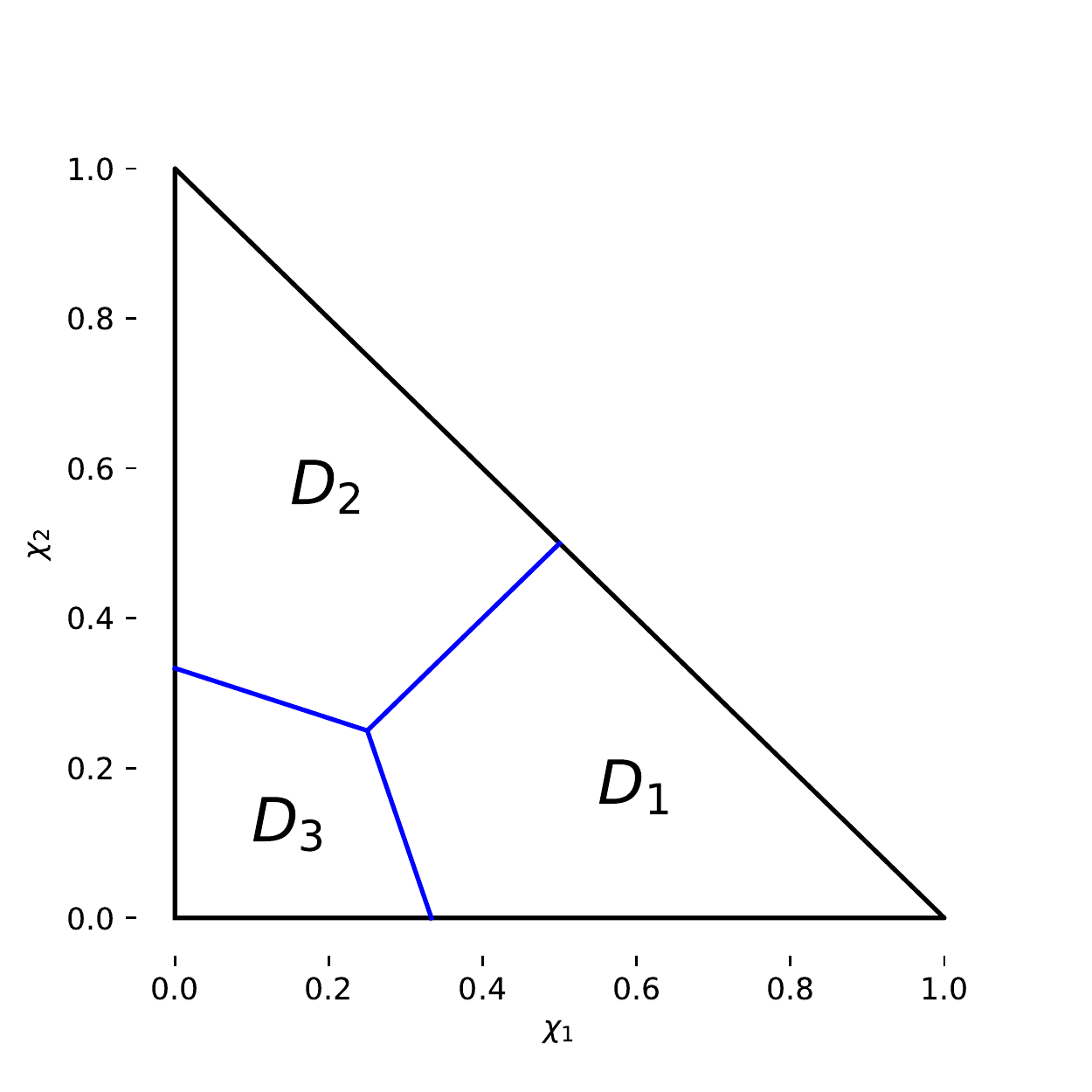}}\quad
  \subfloat[][$F_1(k)=F_2(k)=k^3$, $F_3(k)=k^2$]{\includegraphics[width=0.3\linewidth]{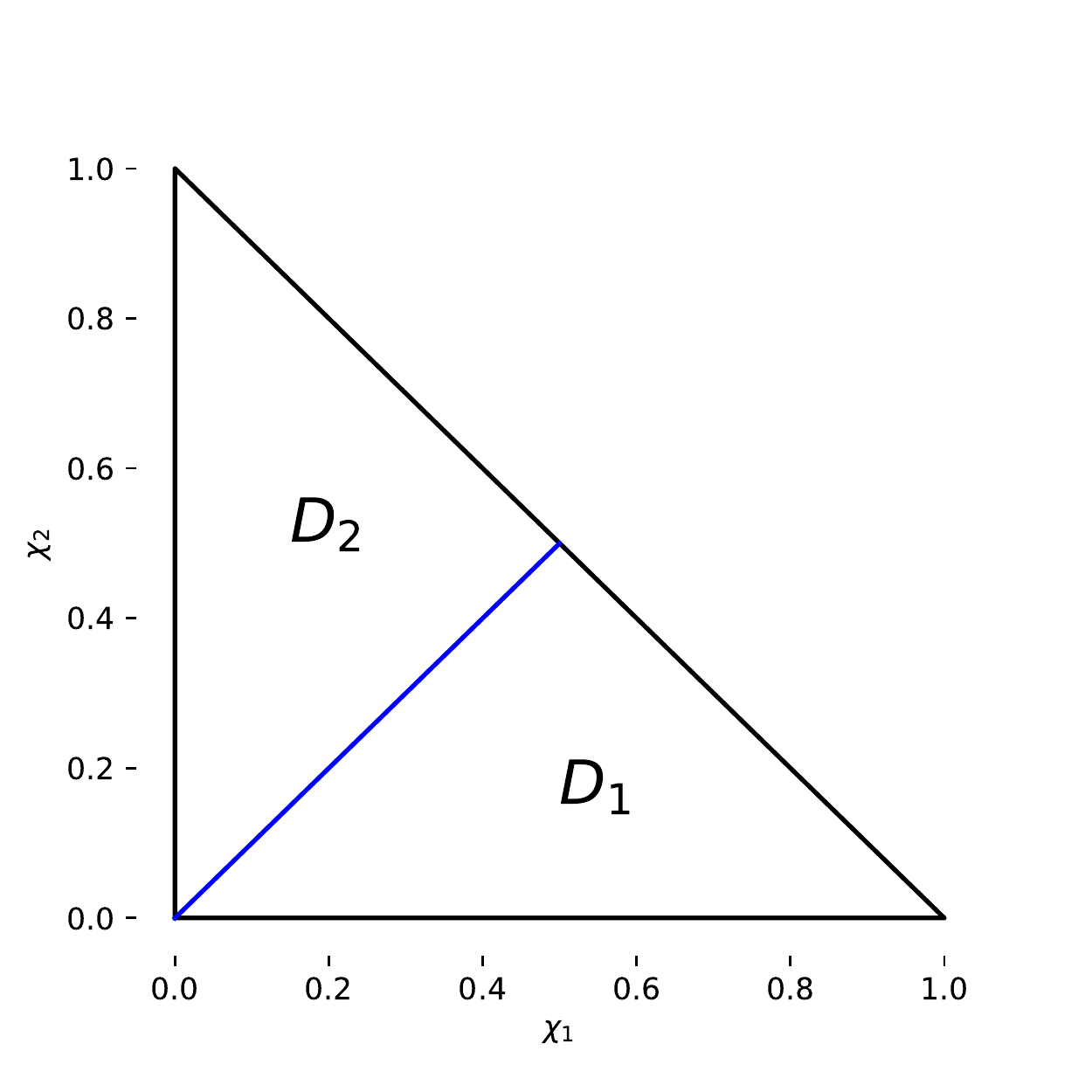}}%
  \caption{Asymptotic attraction domains in the case $A=3$ with various feedback functions.}
\end{figure}

As a direct consequence of the exponential embedding, $\P(sMon_i(\chi(0), N)=0$ for all agents that do not fulfill (\ref{eq: explosion}). Hence, their attraction domains are empty. The rest of Theorem \ref{thm: domains} basically follows from the results presented in the following subsections, where e.g. explicit conditions for
$$\lim_{N\to\infty}\P(sMon_i(\chi(0), N))=1$$
as well as bounds of $\P(sMon_i(\chi(0), N))$ are derived. The final proof of Theorem \ref{thm: domains} will be presented in subsection \ref{subsec: ProofDomain}. It will turn out that the explosion times $T_i$ from Section \ref{sec: model} concentrate on their expectations, i.e.
$$\lim_{N\to\infty}\frac{T_i(N\chi_i(0))}{\E T_i(N\chi_i(0))}=1\quad\text{almost surely,}$$
only for agents of type P, but not for type E, so we need to study these two types of feedback functions separately. The technical conditions in each case are mild and will be discussed in the following subsections. Another characteristic of type E is, that a strong monopoly is typically even a total monopoly, at least when $N$ is large.

\begin{theorem}
\label{thm: tMonTypeE}
Let Assumption 1 in Theorem \ref{thm: domains} be satisfied. If agent $i$ is of type E and $\chi (0)\in D_i^o$ is in the interior of $D_i$, then
\begin{align}
\label{eq: tMonConv}
\lim_{N\to\infty}&\P(tMon_i(\chi(0), N))=1.
\end{align}
\end{theorem}

Theorem \ref{thm: tMonTypeE} is a direct consequence of Theorem \ref{theorem1} given below. As explained in Corollary \ref{cor: criticalShare}, total monopoly does in general not occur, if $\chi(0)$ is on the boundary of the attraction domain. In adddition, it turns out that in generic situations the probability of total monopoly is bounded away from one, if all agents are of type P.

\subsection{Agents of type E and total monopoly}

This subsection examines the process, when at least one agent is of type E. The following results basically imply the first part of Theorem \ref{thm: domains} as well as Theorem \ref{thm: tMonTypeE} as described in Section \ref{subsec: ProofDomain}. The main result of this subsection provides a useful lower and upper bound for the probability of total monopoly.

\begin{theorem}
\label{theorem1}
Let agent $i$ fulfill (\ref{eq: explosion}). Then for all $\chi (0)\in\Delta_{A-1}^o$ and $N\geq 1$
\begin{align*}
\prod_{j\ne i}\exp\left\{-F_j(\chi_j(0)N)\sum_{k=\chi_i(0)N}^{\infty}\frac{1}{F_i(k)}\right\}&\le\P(tMon_i(\chi(0), N))\\
&\le\prod_{j\ne i}\exp\left\{-c_NF_j(\chi_j(0)N)\sum_{k=\chi_i(0)N}^{\infty}\frac{1}{F_i(k)}\right\}
\end{align*}
where
\begin{equation*}
c_N\coloneqq\inf_{k\in\N_0}\frac{F_i(\chi_i(0)N+k)}{F_i(\chi_i(0)N+k)+\sum_{j\ne i}F_j(\chi_j(0)N)}>0\ .
\end{equation*}
\end{theorem}

\begin{proof}
Direct calculation yields
\begin{align*}
&\P(tMon_i(\chi(0), N))=\prod_{k=0}^{\infty}\frac{F_i(\chi_i(0)N+k)}{F_i(\chi_i(0)N+k)+\sum_{j\ne i}F_j(\chi_j(0)N)}\\
&\ge\exp\left\{-\sum_{k=0}^{\infty}\frac{\sum_{j\ne i}F_j(\chi_j(0)N)}{F_i(\chi_i(0)N+k)}\right\}=\prod_{j\ne i}\exp\left\{-F_j(\chi_j(0)N)\sum_{k=\chi_i(0)N}^{\infty}\frac{1}{F_i(k)}\right\}
\end{align*}
using the inequality $e^{-x}\le \frac{1}{1+x}$ for $x>-1$. For the upper bound, we estimate
\begin{align*}
&\P(tMon_i(\chi(0), N))=\prod_{k=0}^{\infty}\frac{F_i(\chi_i(0)N+k)}{F_i(\chi_i(0)N+k)+\sum_{j\ne i}F_j(\chi_j(0)N)}\\
&=\exp\left\{\sum_{k=0}^{\infty}\left[\log(F_i(\chi_i(0)N+k)-\log\left(F_i(\chi_i(0)N+k)+\sum_{j\ne i}F_j(\chi_j(0)N)\right)\right]\right\}\\
&\overset{\star\star}{\le}\exp\left\{-\sum_{k=0}^{\infty}\frac{\sum_{j\ne i}F_j(\chi_j(0)N)}{F_i(\chi_i(0)N+k)+\sum_{j\ne i}F_j(\chi_j(0)N)}\right\}\\
&=\prod_{j\ne i}\exp\left\{-F_j(\chi_j(0)N)\sum_{k=0}^{\infty}\frac{1}{F_i(\chi_i(0)N+k)+\sum_{j\ne i}F_j(\chi_j(0)N)}\right\}\\
&\le\prod_{j\ne i}\exp\left\{-c_NF_j(\chi_j(0)N)\sum_{k=0}^{\infty}\frac{1}{F_i(\chi_i(0)N+k)}\right\}
\end{align*}
using $\log(x+y)-\log(x)\ge\frac{y}{x+y}$ in $\star\star$.
\end{proof}

An immediate consequence of Theorem \ref{theorem1} is, that for any agent fulfilling (\ref{eq: explosion}) the probability of a total monopoly is positive but less than one. In addition, the theorem reveals a significant behavioural difference between agents of type E and type P: whereas total monopoly is very likely for type E agents when the initial market size $N$ is large, it is rather untypical for type P, which is explained in the following corollary and example.

\begin{corollary}
\label{cor:44}
\begin{enumerate}
\item If agent $i$ is of type E , then for all $\chi (0)\in\Delta_{A-1}^o$ the following are equivalent:
\begin{align}
\label{eq: tMonCondE}
\lim_{N\to\infty}&\frac{F_i(\chi_i(0)N)}{F_j(\chi_j(0)N)}=\infty\quad\mbox{for all }j\ne i\\
\label{eq: tMon}
\lim_{N\to\infty}&\P(tMon_i(\chi(0), N))=1.
\end{align}
\item If agent $i$ fulfills (\ref{eq: explosion}), then for all $\chi (0)\in\Delta_{A-1}^o$
 \begin{equation}
 \label{eq: tMonCondP}
F_j(\chi_i(0)N)\sum_{k=\chi_i(0)N}^{\infty}\frac{1}{F_i(k)}\xrightarrow{N\to\infty} 0\quad\mbox{for all }j\ne i
\end{equation}
is sufficient for (\ref{eq: tMon}). If in addition $F_i(k)$ is monotone for large $k$, (\ref{eq: tMonCondP}) is equivalent to (\ref{eq: tMon}).
\end{enumerate}
\end{corollary}

\begin{proof} 1. 
If $i$ is of type E, then (\ref{eq: tMonCondE}) implies
\begin{equation*}
F_j(\chi_j(0)N)\sum_{k=\chi_i(0)N}^{\infty}\frac{1}{F_i(k)}\le F_j(\chi_j(0)N)\frac{const.}{F_i(\chi_i(0)N)}\xrightarrow{N\to\infty}0
\end{equation*}
using (\ref{eq: typeE}), and \eqref{eq: tMon} follows from the lower bound of Theorem \ref{theorem1}. The necessity of \eqref{eq: tMonCondE} follows from
\begin{equation*}
\P(tMon_i(\chi(0), N))\le \frac{F_i(\chi_i(0)N)}{\sum_{j=1}^AF_j(\chi_j(0)N)} =\Big(1+\sum_{j\neq i}\frac{F_j(\chi_i(0)N)}{ F_i(\chi_j(0)N)}\Big)^{-1} \ .
\end{equation*}

2. (\ref{eq: tMonCondP}) implies that the lower bound of Theorem \ref{theorem1} converges to one so that \eqref{eq: tMon} holds. Now we assume that (\ref{eq: tMonCondP}) does not hold. If $\frac{F_i(\chi_i(0)N)}{F_j(\chi_j(0)N)}$ does not converge to infinity for some $j\ne i$, then with 1., (\ref{eq: tMon}) cannot hold. Thus we can assume (\ref{eq: tMonCondE}) for all $j\ne i$, which implies $c_N\xrightarrow{N\to\infty}1$ for the upper bound in Theorem \ref{theorem1} due to asymptotic monotonicity of $F_i (\chi_i (0)N+k)$ as $N\to\infty$. The upper bound then implies that $\P(tMon_i(\chi(0), N))$ does not converge to one.
\end{proof}


\begin{example}
\label{tMonExample}
\begin{enumerate}
\item In the polynomial case $F_i(k)=\alpha_ik^{\beta_i}$ with $\alpha_i>0, i=1,...,A$ and $1<\beta_1\le\ldots\le\beta_A$ condition (\ref{eq: tMonCondP}) is equivalent to $\beta_A>\beta_{A-1}+1$ for all $\chi (0)\in\Delta_{A-1}^o$. If $\beta_A=\beta_{A-1}+1$, then
$$\lim_{N\to\infty}\P(tMon_A(\chi(0), N))=\prod_{j=1,\ldots,A-1\colon\atop \beta_A=\beta_j+1}\exp\left\{-\frac{\alpha_j}{\alpha_A}\left(\frac{\chi_j(0)}{\chi_A(0)}\right)^{\beta_A-1}\right\}\in(0, 1)$$
since $c_N\xrightarrow{N\to\infty}1$ and $\lim_{N\to\infty}\P(tMon_j(\chi(0), N))=0$ for $j\ne A$. If $\beta_A<\beta_{A-1}+1$, in particular if $\beta_1=\ldots=\beta_A$, then $\lim_{N\to\infty}\P(tMon_i(\chi(0), N))=0$ for all agents.

\item When $F_i(k)=\alpha_ie^{\beta_ik}$ for $\alpha_i>0, \beta_i>0, i=1,...,A$, then condition (\ref{eq: tMonCondE}) is equivalent to $\beta_i\chi_i(0)>\beta_j\chi_j(0)$.
\end{enumerate}
\end{example}

Remarkably for type E agents, if $F_i(k)=\alpha_i F(k)$ for all $i$ and a function $F$ fulfilling (\ref{eq: typeE}), then for large $N$ the almost surely deterministic monopolist does not depend on the attractiveness-parameters $\alpha_i$, but is only determined by the initial condition due to the strong feedback effect of type E functions. 

Moreover, Theorem \ref{theorem1} provides information about the rate of convergence in (\ref{eq: tMon}) and (\ref{eq: sMon}). If agent $i$ is of type E, then Theorem \ref{theorem1} states together with $1+x\le e^x$ and $\sum_{l=1}^k(1-x_l)\ge1-\sum_{l=1}^kx_i, \,x_1,\ldots x_k\ge0$
\begin{equation*}
\P(tMon_i(\chi(0), N))\ge\prod_{j\ne i}\left(1-C\frac{F_j(\chi_j(0)N)}{F_i(\chi_i(0)N)}\right)\ge1-C\sum_{j\ne i}\frac{F_j(\chi_j(0)N)}{F_i(\chi_i(0)N)},
\end{equation*}
where
\begin{equation*}
C\coloneqq\sup_{k\ge 1}F_i(k)\sum_{l=k}^\infty\frac{1}{F_i(l)}<\infty
\end{equation*}
because of (\ref{eq: typeE}). Thus the convergence can be considered as quite fast. For example for $A=3$, $F_1(k)=F_2(k)=F_3(k)=e^k$ and $X(0)=(6, 4, 4)$ the bounds in Theorem \ref{theorem1} are:
$$0.652\approx e^{-2/(e(e-1))}\le\P(tMon_1(\chi(0), N))\le e^{-2e/((e-1)(2-e^2))}\approx 0.714$$

Indeed, condition (\ref{eq: tMonCondE}) is fulfilled for an $i$ in most generic cases, when at least one agent is of type E. To be more precise: If the expression in (\ref{eq: tMonCondE}) neither tends to infinity nor to zero, then an arbitrarily small change in the initial market shares provides (\ref{eq: tMonCondE}).

\begin{proposition}
\label{existenceE}
Let agent $i$ be of type E.
\begin{enumerate}
\item If $j$ is of type P for all $j\ne i$, then (\ref{eq: tMonCondE}) holds.
\item If $j\ne i$ is of type E and
\begin{equation}
\label{eq: noConv}
\liminf_{N\to\infty}\frac{F_i(\chi_i(0)N)}{F_j(\chi_j(0)N)}>0,
\end{equation}
then for any $\epsilon>0$:
\begin{equation*}
\lim_{N\to\infty}\frac{F_i((\chi_i(0)+\epsilon)N)}{F_j(\chi_j(0)N)}=\infty
\end{equation*}
\end{enumerate}
\end{proposition}

\begin{proof}
1. 
By (\ref{eq: typeE}) we have for agent $i$ of type E that
\begin{equation}
\label{eq: est2}
\frac{\sum_{l=k+1}^{\infty}\frac{1}{F_i(l)}}{\sum_{l=k}^{\infty}\frac{1}{F_i(l)}}=1-\frac{1}{F_i(k)\sum_{l=k}^{\infty}\frac{1}{F_i(l)}}<1-c
\end{equation} 
for some $c\in (0,1)$ and $k$ large enough, thus the sequence $\left(\sum_{l=\chi_i(0)N}^{\infty}\frac{1}{F_i(l)}\right)_N$ converges to zero faster than $(1-c)^{N/\chi_i(0)}$. For an agent $j\ne i$ of type P we have by (\ref{eq: typeP}) for any $d>0$
\begin{equation*}
\frac{\sum_{l=k+1}^{\infty}\frac{1}{F_j(l)}}{\sum_{l=k}^{\infty}\frac{1}{F_j(l)}}=1-\frac{1}{F_j(k)\sum_{l=k}^{\infty}\frac{1}{F_j(l)}}>1-d
\end{equation*} 
for  $k$ large enough, thus the sequence $\left(\sum_{l=\chi_j(0)N}^{\infty}\frac{1}{F_j(l)}\right)_N$ converges to zero slower than $(1-d)^{N/\chi_j(0)}$. Together this yields
\begin{equation*}
\frac{\sum_{l=\chi_j(0)N}^{\infty}\frac{1}{F_j(l)}}{\sum_{l=\chi_i(0)N}^{\infty}\frac{1}{F_i(l)}}\xrightarrow{N\to\infty}\infty
\end{equation*}
exponentially fast as d is arbitrarily small. Finally (\ref{eq: tMonCondE}) follows from
\begin{equation}
\label{eq: est}
\frac{F_i(\chi_i(0)N)}{F_j(\chi_j(0)N)}\ge\frac{\sum_{l=\chi_j(0)N}^{\infty}\frac{1}{F_j(l)}}{\sum_{l=\chi_i(0)N}^{\infty}\frac{1}{F_i(l)}}\cdot\frac{1}{F_j(\chi_j(0)N)\sum_{l=\chi_j(0)N}^{\infty}\frac{1}{F_j(l)}}\xrightarrow{N\to\infty}\infty
\end{equation}
as $F_j(\chi_j(0)N)\sum_{l=\chi_j(0)N}^{\infty}\frac{1}{F_j(l)}\to\infty$ slower than exponentially.

2. Now let agent $j\ne i$ be of type E and assume (\ref{eq: noConv}). Then with (\ref{eq: typeE}):
\begin{equation*}
\liminf_{N\to\infty}\frac{\sum_{l=\chi_j(0)N}^{\infty}\frac{1}{F_j(l)}}{\sum_{l=\chi_i(0)N}^{\infty}\frac{1}{F_i(l)}}\ge \liminf_{N\to\infty}const.\frac{F_i(\chi_i(0)N)}{F_j(\chi_j(0)N)}>0
\end{equation*}
Iterated application of estimate (\ref{eq: est2}) yields
\begin{equation*}
\frac{\sum_{l=(\chi_i(0)+\epsilon)N}^{\infty}\frac{1}{F_i(l)}}{\sum_{l=\chi_i(0)N}^{\infty}\frac{1}{F_i(l)}}<(1-c)^{\lfloor\epsilon N\rfloor}\xrightarrow{N\to\infty}0\quad\mbox{for some }c\in (0,1)\ ,
\end{equation*}
and as a consequence
\begin{equation*}
\frac{\sum_{l=(\chi_i(0)+\epsilon)N}^{\infty}\frac{1}{F_i(l)}}{\sum_{l=\chi_j(0)N}^{\infty}\frac{1}{F_j(l)}}=\frac{\sum_{l=(\chi_i(0)+\epsilon)N}^{\infty}\frac{1}{F_i(l)}}{\sum_{l=\chi_i(0)N}^{\infty}\frac{1}{F_i(l)}}\cdot\frac{\sum_{l=\chi_i(0)N}^{\infty}\frac{1}{F_i(l)}}{\sum_{l=\chi_j(0)N}^{\infty}\frac{1}{F_j(l)}}\xrightarrow{N\to\infty}0
\end{equation*}
Once again, the estimate in (\ref{eq: est}) proves the claim together with (\ref{eq: typeE}).
\end{proof}

Corollary \ref{cor:44} implies that for any agent $i\in[A]$ of type E
$$\{\chi(0)\in\Delta_{A-1}^o\colon (\ref{eq: tMonCondE}) \text{ holds}\} \subseteq D_i.$$
Due to Proposition \ref{existenceE}, these sets are even equal up to boundaries under Assumption 1 of Theorem \ref{thm: domains}. Moreover, the first part of  Proposition \ref{existenceE} states that the attraction domains of all agents of type P are empty, if there is at least one agent of type E. Recall that for finite $N$ the probability of monopoly is positive for all agents satisfying (\ref{eq: explosion}).

Finally, one can ask what happens for large $N$ and critical market shares, i.e. for $\chi(0)$ lying exactly on the edge between the asymptotic attraction domains. It stands to reason that in this situation the monopolist remains random even for large $N$. Nevertheless, the exact limiting behaviour depends on whether the feedback functions grow exponentially or even super-exponentially.

\begin{corollary}
\label{cor: criticalShare}
Let all agents be of type E and consider $\chi(0)\in\Delta_{A-1}^o$, such that
\begin{equation}
\label{eq: criticalShare}
    \limsup_{N\to\infty}\frac{F_i(\chi_i(0)N)}{F_j(\chi_j(0)N)}<\infty
\end{equation}
for all $i, j\in[A]$. Then the following holds:
\begin{enumerate}
    \item For all agents $i\in[A]$ we have $\liminf_{N\to\infty}\P(tMon_i(\chi(0), N))>0$.
    \item If for all agents $i\in[A]$ we have super-exponentially growing feedback, i.e.
    \begin{equation*}
        \lim_{k\to\infty}\frac{F_i(k+1)}{F_i(k)}=\infty,
    \end{equation*}
    then $\lim_{N\to\infty}\P\left(\bigcup_{i=1}^AtMon_i(\chi(0), N)\right)=1$.
    \item If for all agents $i\in[A]$ we have at most exponentially growing feedback, i.e.
    \begin{equation*}
        \limsup_{k\to\infty}\frac{F_i(k+1)}{F_i(k)}<\infty,
    \end{equation*}
    then $\limsup_{N\to\infty}\P\left(\bigcup_{i=1}^AtMon_i(\chi(0), N)\right)<1$.
\end{enumerate}
\end{corollary}

\begin{proof}
1. This follows directly from Theorem \ref{theorem1} and (\ref{eq: typeE}):
\begin{align*}
    \P(tMon_i(\chi(0))&\ge\prod_{i\ne j}\exp\left\{-F_i(\chi_i(0)N)\sum_{k=\chi_j(0)N}^\infty\frac{1}{F_j(k)}\right\}\ge\prod_{i\ne j}\exp\left\{-const.\frac{F_i(\chi_i(0)N)}{F_j(\chi_j(0)N)}\right\}
\end{align*}

2. First, we write
\begin{align*}
    \P&\left(\bigcup_{i=1}^AtMon_i(\chi(0), N)\right)\\
    &=\sum_{j=1}^A\P\left(\bigcup_{i=1}^AtMon_i(\chi(0), N)\,\big| X(1)-X(0)=e^{(j)}\right)\P\left(X(1)-X(0)=e^{(j)}\right)\\
    &=\sum_{j=1}^A\P\left(tMon_j\left(\frac{1}{N+1}\left(\chi(0)N+e^{(j)}\right), N+1\right)\right)\P\left(X(1)-X(0)=e^{(j)}\right)
\end{align*}
and then apply Theorem \ref{theorem1} and (\ref{eq: typeE}):
\begin{align*}
    &\P\left(tMon_j\left(\frac{1}{N+1}\left(\chi(0)N+e^{(j)}\right), N+1\right)\right)\ge\prod_{i\ne j}\exp\left\{-F_i(\chi_i(0)N)\sum_{k=\chi_j(0)N+1}^\infty\frac{1}{F_j(k)}\right\}\\
    &\ge\prod_{i\ne j}\exp\left\{-const.\frac{F_i(\chi_i(0)N)}{F_j(\chi_j(0)N+1)}\right\}=\prod_{i\ne j}\exp\left\{-const.\frac{F_i(\chi_i(0)N)}{F_j(\chi_j(0)N)}\cdot\frac{F_j(\chi_j(0)N)}{F_j(\chi_j(0)N+1)}\right\}\\
    &\xrightarrow{N\to\infty}1
\end{align*}

3. Similarly to the second part, this follows from
\begin{align*}
    \P&\left(tMon_j\left(\frac{1}{N+1}\left(\chi(0)N+e^{(j)}\right), N+1\right)\right)\le\frac{F_j(\chi(0)N+1)}{F_j(\chi_i(0)N+1)+\sum_{i\ne j}F_i(\chi_i(0)N)}\\
    &=\left(1+\sum_{i\ne j}\frac{F_i(\chi_i(0))}{F_j(\chi_j(0))}\cdot\frac{F_j(\chi_j(0))}{F_j(\chi_j(0)+1)}\cdot\right)^{-1}.
\end{align*}
\end{proof}

\begin{example}
Let $F_i(k)=e^{\alpha_ik^\beta}$ for $\alpha_i>0, \beta>0$ and all $i\in[A]$. Then condition (\ref{eq: criticalShare}) is equivalent to $\alpha_i\chi_i(0)^\beta=\alpha_j\chi_j(0)^\beta$ for all $i,j\in[A]$. According to Corollary \ref{cor: criticalShare}, we have in this case $\lim_{N\to\infty}\P\left(\bigcup_{i=1}^AtMon_i(\chi(0), N)\right)=1$ for $\beta>1$ and \linebreak$\limsup_{N\to\infty}\P\left(\bigcup_{i=1}^AtMon_i(\chi(0), N)\right)<1$ for $\beta\le1$.
\end{example}

We summarize the main conclusions for total mononpoly in the limit of large initial market size $N\to\infty$: If for all agents the feedback functions grow super-exponentially, the winner of the first step will win all steps. This does not hold for any $\chi (0)\in \Delta_{A-1}^o$ if all feedback functions grow at most exponentially. In general, total monopoly of an agent $i$ can occur with probability one according to Corollary \ref{cor:44}: if $i$ is of type $E$ and \eqref{eq: tMonCondE} holds, or if \eqref{eq: tMonCondP} holds.


\subsection{Agents of type P}

Let us now turn to the more widely studied case when all agents are of type P. We already saw in Example \ref{tMonExample} that in this case a total monopoly is rather untypical. Since the definition of type P includes the monopoly condition (\ref{eq: explosion}), strong monopoly still occurs with probability one. Again, it is possible to predict the monopolist in the limit $N\to\infty$.

\begin{theorem}
\label{thm: typeP}
Let all agents be of type P or not fulfill (\ref{eq: explosion}). If there is an agent $i\in\{1,...,A\}$ of type P such that
\begin{equation}
\label{eq: sMonCondP}
\limsup_{N\to\infty}\frac{\sum_{k=\chi_i(0)N}^\infty\frac{1}{F_i(k)}}{\sum_{k=\chi_j(0)N}^\infty\frac{1}{F_j(k)}}<1\quad\mbox{for all }j\ne i\ ,
\end{equation}
then
\begin{equation}
\label{eq: sMon}
\lim_{N\to\infty}\P(sMon_i(\chi(0), N))=1\ .
\end{equation}
\end{theorem}

Note that condition (\ref{eq: sMonCondP}) can be replaced by the easier, but stricter condition
$$\limsup_{N\to\infty}\frac{F_j(\chi_j(0)N)}{F_i(\chi_i(0)N)}<\frac{\chi_j(0)}{\chi_i(0)}$$
due to de l'Hospital's Theorem. This implies that for regular varying $F_i(k)=\alpha_ik^\beta L(k)$, where $\beta>1$ and $L$ is a slowly varying function, the attraction domains are equal to the polynomial case, where $F_i(k)=\alpha_ik^\beta$. Moreover, the attraction domains do not change if $F_i$ is replaced by another function $\tilde F_i$ satisfying $\lim_{k\to\infty}\frac{\tilde F_i(k)}{F_i(k)}=1$.

\begin{proof}
This is an immediate consequence of the following Lemma \ref{lem: explosion} and the exponential embedding representation (\ref{eq: sMonEmbed}) of the strong monopoly via
\begin{align*}
\P\left(\left\{T_i(\chi_i(0)N)<T_j(\chi_j(0)N)\right\}\right)&=\P\left(\frac{T_i(\chi_i(0)N)}{\E T_i(\chi_i(0)N)}\cdot\frac{\E T_j(\chi_j(0)N)}{T_j(\chi_j(0)N)}\cdot\frac{\E T_i(\chi_i(0)N)}{\E T_j(\chi_j(0)N)}<1\right)\\
&\xrightarrow{N\to\infty}1,
\end{align*}
since $\E T_i(\chi_i(0)N) =\sum_{k=\chi_i(0)N}^\infty\frac{1}{F_i(k)}$.
\end{proof}

\begin{lemma}
\label{lem: explosion}
If agent $i$ is of type P, then:
\begin{equation*}
Var\left(\frac{T_i(\chi_i(0)N)}{\E T_i(\chi_i(0)N)}\right)\xrightarrow{N\to\infty}0
\end{equation*}
\end{lemma}

\begin{proof}
We can find an appropriate regular extension of $F_i$, such that for all $n\geq 1$
\begin{equation*}
\sum_{k=n}^\infty\frac{1}{F_i(k)}=\int_n^\infty\frac{dx}{F_i(x)}\quad\text{and}\quad\sum_{k=n}^\infty\frac{1}{F_i(k)^2} =\int_n^\infty\frac{dx}{F_i(x)^2} \ .
\end{equation*}
By the theorem of de L'Hospital and (\ref{eq: typeP}) this implies
\begin{align*}
\lim_{N\to\infty}Var\left(\frac{T_i(\chi_i(0)N)}{\E T_i(\chi_i(0)N)}\right)&=\lim_{N\to\infty}\frac{\sum_{k=\chi_i(0)N}^\infty\frac{1}{F_i(k)^2}}{\left(\sum_{k=\chi_i(0)N}^\infty\frac{1}{F_i(k)}\right)^2}=\lim_{N\to\infty}\frac{\int_{\chi_i(0)N}^\infty\frac{dx}{F_i(x)^2}}{\left(\int_{\chi_i(0)N}^\infty\frac{dx}{F_i(x)}\right)^2}\\
&=\lim_{N\to\infty}\frac{1}{2F_i(\chi_i(0)N)\sum_{k=\chi_i(0)N}\frac{1}{F_i(k)}}=0\ .
\end{align*}
\end{proof}

\begin{example}\label{example: sMonPolynom}
If $F_i(k)=\alpha_ik^\beta$ for all $i$ and $\beta>1$, then the condition (\ref{eq: sMonCondP}) is equivalent to $\alpha_i\chi_i(0)^{\beta-1}>\alpha_j\chi_j(0)^{\beta-1}$. Thus, in contrast to the type E case (Example \ref{tMonExample}), the attractiveness-parameters $\alpha_i$ affect the monopolist.
\end{example}

Lemma \ref{lem: explosion} uncovers another behavioral difference between type P and type E agents: For type P agents the explosion time concentrates on its expectation, whereas the variance of $T_i(\chi_i(0)N)/\E T_i(\chi_i(0)N)$ remains bounded from below for type E agents by an analogous argument, using \eqref{eq: typeE}. 
For many type P agents, including $F_i(k)=\alpha_ik^{\beta_i}$, it is possible to prove that the convergence of $T_i(\chi_i(0)N)/\E T_i(\chi_i(0)N)$ is even almost sure (see the proof of Proposition \ref{prop: prop1} together with the Lemma of Borel-Cantelli).

It is now natural to look for an analogy to Proposition \ref{existenceE} for type P agents in order to make sure that (\ref{eq: sMonCondP}) is fulfilled for almost all initial market shares $\chi(0)$. Unfortunately, this attempt is meant to fail as the example $F_i(k)=F_j(k)=k(\log k)^\alpha$ for $\alpha>1$ shows. In  this case
\begin{equation*}
\sum_{k=\lfloor\chi_i(0)N\rfloor}^\infty \frac{1}{F_i(k)}\sim\int_{\chi_i(0)N}^\infty\frac{1}{x(\log x)^\alpha}dx=\frac{1}{1-\alpha}\log (\chi_i(0)N)^{1-\alpha}\ ,
\end{equation*}
where for sequences $(a_N )_N$ and $(b_N )_N$ we write $a_N \sim b_N$ if $a_N /b_N \to 1$ for $N\to\infty$. Therefore (\ref{eq: sMonCondP}) is not fulfilled for all choices of $\chi_i(0), \chi_j(0)$, since
\begin{equation*}
\lim_{N\to\infty}\frac{\sum_{k=\chi_i(0)N}^\infty\frac{1}{F_i(k)}}{\sum_{k=\chi_j(0)N}^\infty\frac{1}{F_j(k)}}=\lim_{N\to\infty}\left(\frac{\log(N)+\log(\chi_i(0))}{\log(N)+\log(\chi_j(0))}\right)^{1-\alpha}=1.
\end{equation*}
Nevertheless, with a further condition we can find a similar result as Proposition \ref{existenceE}.

\begin{proposition}
\label{prop: typePcond}
Suppose that for some $i\neq j$
\begin{equation}
\label{eq: existenceP}
\lim_{N\to\infty}\frac{\sum_{k=\chi_i(0)N}^\infty\frac{1}{F_i(k)}}{\sum_{k=\chi_j(0)N}^\infty\frac{1}{F_j(k)}}=1.
\end{equation}
\begin{enumerate}
\item  If there exists $C<\infty$ such that for all $k\in\N$
\begin{equation*}
\frac{1}{k}F_i(k)\sum_{l=k}^\infty\frac{1}{F_i(l)}\le C,
\end{equation*}
then for all $\epsilon>0$
\begin{equation*}
\limsup_{N\to\infty}\frac{\sum_{k=(\chi_i(0)+\epsilon)N}^\infty\frac{1}{F_i(k)}}{\sum_{k=\chi_j(0)N}^\infty\frac{1}{F_j(k)}}<1\ .
\end{equation*}
\item If 
\begin{equation}
\label{eq:con2}
\lim_{k\to\infty}\frac{1}{k}F_i(k)\sum_{l=k}^\infty\frac{1}{F_i(l)}=\infty\quad \text{(i.e. $i$ is in particular of type P)}
\end{equation}
and (\ref{eq: existenceP}) holds for one choice of $\chi_i(0), \chi_j(0)$, then (\ref{eq: existenceP}) holds for all choices $\chi_i(0), \chi_j(0)\geq 0$ with $\chi_i(0)+ \chi_j(0)\leq 1$.
\end{enumerate}
\end{proposition}

\begin{proof} 1. We have by (\ref{eq: existencePcond})
\begin{equation}
\label{eq: estimate}
\frac{\sum_{k=\chi_i(0)N+1}^\infty\frac{1}{F_i(k)}}{\sum_{k=\chi_i(0)N}^\infty\frac{1}{F_i(k)}}=1-\frac{1}{F_i(\chi_i(0)N)\sum_{k=\chi_i(0)N}^\infty\frac{1}{F_i(k)}}\le1-\frac{1}{C\chi_i(0)N}
\end{equation}
and iterated application of this yields
\begin{equation*}
\frac{\sum_{k=(\chi_i(0)+\epsilon)N}^\infty\frac{1}{F_i(k)}}{\sum_{k=\chi_i(0)N}^\infty\frac{1}{F_i(k)}}\le\left(1-\frac{1}{C(\chi_i(0)+\epsilon)N}\right)^{\lfloor\epsilon N\rfloor}\xrightarrow{N\to\infty}e^{-\frac{\epsilon}{C(\chi_i(0)N+\epsilon)}}<1\ .
\end{equation*}
Finally, this implies
\begin{equation*}
\limsup_{N\to\infty}\frac{\sum_{k=(\chi_i(0)+\epsilon)N}^\infty\frac{1}{F_i(k)}}{\sum_{k=\chi_j(0)N}^\infty\frac{1}{F_j(k)}}=\limsup_{N\to\infty}\frac{\sum_{k=(\chi_i(0)+\epsilon)N}^\infty\frac{1}{F_i(k)}}{\sum_{k=\chi_i(0)N}^\infty\frac{1}{F_i(k)}}\cdot\frac{\sum_{k=\chi_i(0)N}^\infty\frac{1}{F_i(k)}}{\sum_{k=\chi_j(0)N}^\infty\frac{1}{F_j(k)}}<1\ .
\end{equation*}
2. The second part follows by similar arguments, using Condition \eqref{eq:con2} for an "$\geq$"-estimate in (\ref{eq: estimate}), where $C=C(N)$ is arbitrarily large.
\end{proof}


\begin{example}\label{example: klogk, beta>1}
If $F_i(k)=\alpha_iF(k)$ for all $i\in[A]$ and a feedback function $F$ fulfilling (\ref{eq:con2}), e.g. $F(k)=k\log(k)^\beta$ for $\beta>1$, then $D_i=\Delta_{A-1}^o$ if $\alpha_i>\alpha_j$ for all $j\ne i$.
\end{example}

If all agents are of type P, Theorem \ref{thm: typeP} implies that for any agent $i\in[A]$
$$\{\chi(0)\in\Delta_{A-1}^o\colon (\ref{eq: sMonCondP})\text{ holds}\}\subseteq D_i.$$
Assuming 2. in Theorem \ref{thm: domains}, we get from Proposition \ref{prop: typePcond} that the sets are equal up to boundaries.

In the situation of the second part of Proposition \ref{prop: typePcond}, the explosion times concentrate asymptotically on the same value, i.e. $T_i(\chi_i(0)N)/T_j(\chi_j(0)N)\xrightarrow{N\to\infty}1$ in distribution. Thus, it is not possible to predict the monopolist for large $N$ by the means of this section. If $\alpha_1=\ldots=\alpha_A$ and $\chi_i(0)=\frac{1}{A}$ for all $i$, then $\P(sMon_i(\chi(0), N))=\frac1A$ holds for all $N$ for symmetry reasons, i.e. $\chi(0)$ does not belong to any attraction domain as the monopolist remains random even in the limit $N\to\infty$. The following example underlines that this property does not hold in general, because in some cases the boundary between the attraction domains belongs to one of them. 

\begin{example}
Consider the process for $A=2$ and $F_1(k)=k^\beta,\,F_2(k)=\frac{k^\beta}{1+k^{-\delta}}$ for $\beta>1$ and $\delta\in(0, \frac14)$. Then we have
$$\sum_{k=\chi_2(0)N}^\infty\frac{1}{F_2(k)}=\sum_{k=\chi_2(0)N}^\infty\left(\frac{1}{k^\beta}+\frac{1}{k^{\beta+\delta}}\right)\sim\sum_{k=\chi_2(0)N}^\infty\frac{1}{k^\beta}=\sum_{k=\chi_2(0)N}^\infty\frac{1}{F_1(k)}$$
and
$$\sum_{k=\chi_2(0)N}^\infty\frac{1}{F_2(k)^2}\sim\sum_{k=\chi_2(0)N}^\infty\frac{1}{F_1(k)^2}$$
for $N\to\infty$. Moreover, set $\chi_1(0)=\chi_2(0)=\frac12$, such that (\ref{eq: existenceP}) holds, and define $\epsilon_N\coloneqq N^{\frac34-\beta}$. Chebyshev's inequality yields
$$\P\left(T_1(\chi_1(0)N)>\E T_1(\chi_1(0)N)+\epsilon_N\right)\le \frac{Var(T_1(\chi_1(0)N))}{\epsilon_N^2}=\frac{\sum_{k=N/2}^\infty\frac{1}{F_1(k)^2}}{\epsilon_N^2}\xrightarrow{N\to\infty}0$$
since $\sum_{k=N/2}^\infty\frac{1}{F_1(k)^2}\sim \frac{1}{2\beta-1}(N/2)^{1-2\beta}$ and
$$\P\left(T_2(\chi_1(0)N)<\E T_2(\chi_1(0)N)-\epsilon_N\right)\le \frac{Var(T_2(\chi_2(0)N))}{\epsilon_N^2}=\frac{\sum_{k=N/2}^\infty\frac{1}{F_2(k)^2}}{\epsilon_N^2}\xrightarrow{N\to\infty}0.$$
In addition, we have for large enough $N$ that
$$\E T_1(\chi_1(0)N)+\epsilon_N<\E T_2(\chi_2(0)N)-\epsilon_N,$$
because $\delta<\frac14$ implies
$$\E T_2(\chi_2(0)N)-\E T_1(\chi_1(0)N)=\sum_{k=N/2}^\infty\frac{1}{k^{\beta+\delta}}\sim \frac{1}{1-\beta-\delta}(N/2)^{1-\beta-\delta}<2N^{\frac34-\beta}.$$
Thus for large $N$
\begin{align*}
\P&(sMon_1(\chi_1(0), N))=\P\big( T_1(\chi_1(0)N)<T_2(\chi_2(0)N)\big)\\
&\ge\P\Big( T_1(\chi_1(0)N)<\E T_1(\chi_1(0)N)+\epsilon_N\,\wedge\,T_2(\chi_2(0)N)>\E T_2(\chi_2(0)N)-\epsilon_N \Big)\\
&=\P\Big(T_1(\chi_1(0)N)<\E T_1(\chi_1(0)N)+\epsilon_N\Big)\P\Big( T_2(\chi_2(0)N)>\E T_2(\chi_2(0)N)-\epsilon_N \Big)\\
&\xrightarrow{N\to\infty}1
\end{align*}
using the independence of $T_1(\chi_1(0)N)$ and $T_2(\chi_1(0)N)$. Hence, $\chi(0)\in D_1$.
\end{example}

We finish this subsection with a result on the rate of convergence in (\ref{eq: sMon}). \cite{Dunlop} presents a bound for $\P(sMon_i(\chi(0), N))$ in the case $F_i(k)=F_j(k)=k^\alpha$, but a straight-forward generalization of this procedure is possible.

\begin{proposition}
\label{prop: prop1}
Let all agents be of type P with monotone feedback functions, such that (\ref{eq: typePlog}) holds in addition. If (\ref{eq: sMonCondP}) holds for agent $i$, i.e. $\chi(0)\in D_i$, we have
\begin{align*}
\P(sMon_i(\chi(0), N))\ge1-\sum_{j=1}^A\exp\left(-(d_j-\epsilon)\sqrt{F_j(\chi_j(0)N)\sum_{k=\chi_j(0)N}^\infty\frac{1}{F_j(k)}}\right)
\end{align*}
for any $\epsilon>0$ and large enough $N$, where
\begin{equation*}
d_j\coloneqq g\left(\limsup_{N\to\infty}\frac{\sum_{k=\chi_i(0)N}^\infty\frac{1}{F_i(k)}}{\sum_{k=\chi_j(0)N}^\infty\frac{1}{F_j(k)}}\right)>0\quad\text{with}\quad g(x)\coloneqq\frac{1-x}{1+x}\quad\text{for }j\neq i
\end{equation*}
and $d_i\coloneqq\min_{j\ne i}d_j$.
\end{proposition}

This means that the rate of convergence in (\ref{eq: typeP}) gives a lower bound for the rate of convergence of $\P(sMon_i(\chi_i(0))$.

\begin{proof}
Once again, the proof uses the exponential embedding from Section 2. Let $t>0$ and $s\coloneqq\left(\sum_{l=k}^\infty\frac{1}{F_j(l)^2}\right)^{-\frac{1}{2}}$. Then the Markov-inequality and monotone convergence yield for all $j\in [A]$ and $t>0$:
\begin{align*}
&\P\left(\sum_{l=k}^\infty\tau_j(l)-\sum_{l=k}^\infty\frac{1}{F_j(l)}>t\sqrt{\sum_{l=k}^\infty\frac{1}{F_j(l)^2}}\right)\\
&\le\exp\left(-s\sum_{l=k}^\infty\frac{1}{F_j(l)}-st\sqrt{\sum_{l=k}^\infty\frac{1}{F_j(l)^2}}\right)\cdot\E\exp\left(s\sum_{l=k}^\infty\tau_j(l)\right)\\
&=\exp\left(-s\sum_{l=k}^\infty\frac{1}{F_j(l)}-t\right)\cdot\prod_{l=k}^\infty\E e^{s\tau_j(l)}=\exp\left(-s\sum_{l=k}^\infty\frac{1}{F_j(l)}-t\right)\cdot\prod_{l=k}^\infty\left(1+\frac{s}{F_j(l)-s}\right)\\
&\le\exp\left(-s\sum_{l=k}^\infty\frac{1}{F_j(l)}-t\right)\cdot\prod_{l=k}^\infty\exp\left(\frac{s}{F_j(l)-s}\right)=\exp\left(s\sum_{l=k}^\infty\left(-\frac{1}{F_j(l)}+\frac{1}{F_j(l)-s}\right)-t\right)\\
&\le\exp\left(\frac{s^2}{1-\frac{s}{F_j(k)}}\sum_{l=k}^\infty\frac{1}{F_j(l)^2}-t\right)=c_j(k)e^{-t},
\end{align*}
where $c_j(k)\coloneqq\exp\left(\frac{1}{1-\frac{s}{F_j(k)}}\right)$. Setting
\begin{equation*}
t=(d_j-\epsilon)\frac{\sum_{l=k}^\infty\frac{1}{F_j(l)}}{\sqrt{\sum_{l=k}^\infty\frac{1}{F_j(l)^2}}}
\end{equation*}
(which is positive for $\epsilon$ small enough since $d_j >0$ for all $j\in [A]$) yields
\begin{align*}
\P\left(\frac{\sum_{l=k}^\infty\tau_j(l)}{\sum_{l=k}^\infty\frac{1}{F_j(l)}}-1>(d_j-\epsilon)\right)&\le c_j(k)\exp\left(-(d_j-\epsilon)\frac{\sum_{l=k}^\infty\frac{1}{F_j(l)}}{\sqrt{\sum_{l=k}^\infty\frac{1}{F_j(l)^2}}}\right)\\
&\le c_j(k)\exp\left(-(d_j-\epsilon)\sqrt{F_j(k)\sum_{l=k}^\infty\frac{1}{F_j(l)}}\right) .
\end{align*}
The second estimate uses $F_j(l)\ge F_j(k)$ by monotonicity. Analogously, one can show
\begin{equation*}
\P\left(\frac{\sum_{l=k}^\infty\tau_j(l)}{\sum_{l=k}^\infty\frac{1}{F_j(l)}}-1<-(d_j-\epsilon)\right)\le e\cdot\exp\left(-(d_j-\epsilon)\sqrt{F_j(k)\sum_{l=k}^\infty\frac{1}{F_j(l)}}\right) ,
\end{equation*}
which will be used only for $j=i$. 
Both estimates then imply for large enough $N$ together with (\ref{eq: sMonEmbed}):
\begin{align*}
&\P(sMon_i(\chi(0), N))\ge1-\sum_{j\ne i}\Big( 1-\P\big( T_i (\chi_i(0)N)<T_j(\chi_j(0)N)\big)\Big)\\
&\ge2-A+\sum_{j\ne i}\P\left(\frac{T_i(\chi_i(0)N)}{\sum_{k=\chi_i(0)N}^\infty\frac{1}{F_i(k)}}>1-(d_j-\epsilon)\right)\cdot\P\left(\frac{T_j(\chi_j(0)N)}{\sum_{k=\chi_j(0)N}^\infty\frac{1}{F_j(k)}}<1+(d_j-\epsilon)\right)\\
&\ge2-A+\left(1-e\cdot\exp\left(-(d_i-\epsilon)\sqrt{F_i(\chi_i(0)N)\sum_{k=\chi_i(0)N}^\infty\frac{1}{F_i(k)}}\right)\right)\\
&\quad\cdot\sum_{j\ne i}\left(1-c_j(\chi_j(0)N)\exp\left(-(d_j-\epsilon)\sqrt{F_j(\chi_j(0)N)\sum_{k=\chi_j(0)N}^\infty\frac{1}{F_i(k)}}\right)\right)\\
&\ge1-A\cdot\exp\left(1-(d_i-\epsilon)\sqrt{F_i(\chi_i(0)N)\sum_{k=\chi_i(0)N}^\infty\frac{1}{F_i(k)}}\right)\\
&\quad-\sum_{j\ne i}c_j(\chi_j(0)N)\exp\left(-(d_j-\epsilon)\sqrt{F_j(\chi_j(0)N)\sum_{k=\chi_j(0)N}^\infty\frac{1}{F_i(k)}}\right)
\end{align*}
In the last inequality we use $(1-x)(1-y)\ge1-x-y$ for $x, y\ge0$. For large enough $N$ we have
\begin{align*}
1+\log(A)&<\epsilon\sqrt{F_i(\chi_i(0)N)\sum_{k=\chi_i(0)N}^\infty\frac{1}{F_i(k)}}\\
\log\left(c_j(\chi_j(0)N)\right)&<\epsilon\sqrt{F_j(\chi_j(0)N)\sum_{k=\chi_j(0)N}^\infty\frac{1}{F_j(k)}}
\end{align*}
for $j\ne i$ because of (\ref{eq: typeP}) and $c_j(k)\xrightarrow{k\to\infty}1$ due to (\ref{eq: typePlog}). Finally, this leads to
\begin{align*}
\P&(sMon_i(\chi(0), N))\ge1-\sum_{j=1}^A\exp\left(-(d_j-2\epsilon)\sqrt{F_j(\chi_j(0)N)\sum_{k=\chi_j(0)N}^\infty\frac{1}{F_j(k)}}\right)
\end{align*}
\end{proof}

For $F_i(k)=\alpha k^\beta$ we have $F_i(k)\sum_{l=k}^\infty\frac{1}{F_i(l)}\sim\frac{k}{1-\beta}$, thus the convergence of\linebreak $\P(sMon_i(\chi_i(0)N))$ can be considered as fast. Hence, $\P(sMon_i(\chi(0), N))$ is close to one even for moderate $N$, when $\chi(0)\in D_i$ is in the asymptotic attraction domain.

In the type E case we saw that a total monopoly is very likely whereas in the type P case the losers might also win in some steps. It is now a question of interest how many steps the losers win, i.e. the value of $X_j(\infty)=\lim_{n\to\infty}X_j(n)$ if agent $j$ is not the monopolist. Results on this question can be found in \cite{Oliveira} and \cite{Zhu}. It is remarkable that for polynomially growing feedback functions the distribution of $X_j(\infty)$ has heavy tails. \cite{Cotar2, Zhu} also present results on the time when the monopoly occurs. 
Further asymptotic results on strong monopoly, mainly in the type P case, can be found e.g. in \cite{Khanin,Mitzenmacher,Oliveira3,Dunlop,Drinea,Menshikov,Jiang}.

\subsection{Proof of Theorem \ref{thm: domains} and Theorem \ref{thm: tMonTypeE}}\label{subsec: ProofDomain}

Finally, we shortly explain how Theorem \ref{thm: domains} and Theorem \ref{thm: tMonTypeE} follow from the results of the previous sections. 

First, assume that Assumption 1 of Theorem \ref{thm: domains} is satisfied, i.e. at least one agent is of type E. Then Corollary \ref{cor:44} implies that for any agent $i\in[A]$ of type E
$$\tilde D_i\coloneqq\{\chi(0)\in\Delta_{A-1}^o\colon (\ref{eq: tMonCondE}) \text{ holds}\} \subseteq D_i.$$
Obviously:
$$\tilde D_i=\bigcap_{j\ne i}\left\{\chi(0)\in\Delta_{A-1}^o\colon \lim_{N\to\infty}\frac{F_i(\chi_i(0)N)}{F_j(\chi_j(0)N)}=\infty\right\}$$
Due to Proposition \ref{existenceE}, there is a ratio $r_{i, j}\in[0, \infty]$ such that
\begin{equation*}
    \lim_{N\to\infty}\frac{F_i(\chi_i(0)N)}{F_j(\chi_j(0)N)}=
    \begin{cases}
        \infty &\text{ if } \frac{\chi_i(0)}{\chi_j(0)}>r_{i, j}\\
        0 &\text{ if } \frac{\chi_i(0)}{\chi_j(0)}<r_{i, j}
    \end{cases}
\end{equation*}
for each pair $i\ne j$ of agents. Note that 
\begin{equation*}
    \lim_{N\to\infty}\frac{F_i(\chi_i(0)N)}{F_j(\chi_j(0)N)}=\lim_{N\to\infty}\frac{F_i\left(\frac{\chi_i(0)}{\chi_j(0)}N\right)}{F_j(N)}.
\end{equation*}
Hence, 
$$\tilde D_i=\bigcap_{j\ne i}\left\{\chi(0)\in\Delta_{A-1}^o\colon \frac{\chi_i(0)}{\chi_j(0)}>r_{i, j} \right\}$$
is an intersection of half-spaces and the simplex, i.e. a polytope. Moreover, $\tilde D_1,\ldots, \tilde D_A$ cover the whole simplex up to boundaries, since the "winning"-relation $\lim_{N\to\infty}\frac{F_i(\chi_i(0)N)}{F_j(\chi_j(0)N)}=\infty$ is transitive. Thus, $D_1,\ldots,D_A$ cover the simplex up to boundaries as well and $\tilde D_i$ equals $D_i$ up to boundaries. According to Corollary \ref{cor:44}, we even have $\P(tMon_i(\chi(0), N))\xrightarrow{} 1$ for $N\to\infty$, if $\chi(0)\in\tilde D_i$. Hence, Theorem \ref{thm: tMonTypeE} is proven, too.

If Assumption 2 of Theorem \ref{thm: domains} is satisfied, the proof is analogous using Theorem \ref{thm: typeP} and Proposition \ref{prop: typePcond}. Note that
$$\lim_{N\to\infty}\frac{F_i(\chi_j(0)N)}{F_j(\chi_i(0)N)}=c\in[0, \infty]\quad\Longrightarrow\quad\lim_{N\to\infty}\frac{\sum_{k=\chi_i(0)N}^\infty\frac{1}{F_i(k)}}{\sum_{k=\chi_j(0)N}^\infty\frac{1}{F_j(k)}}=c\frac{\chi_i(0)}{\chi_j(0)}$$
due to the Theorem of de l'Hospital.

In summary, for finite $N$ the monopolist is random and even disadvantageous agents can win. If the initial market size $N$ is large, it is possible to predict the winner with high probability depending on the initial market shares.


\section{The non-monopoly case}
\label{sec: non-monopoly}

Now we consider the case when no agent fulfills (\ref{eq: explosion}), such that no strong monopoly occurs. It is known that in the case of a standard Pólya urn, i.e. $F_i(k)=k$ for all agents, the limit $\chi(\infty)=\lim_{n\to\infty}\chi(n)$ exists almost surely and $\chi(\infty)$ has a Dirichlet-distribution with parameter $X(0)$ (see e.g. \cite{Freedman}). Thus, in the long run all agents have a stable, non-zero, random market share.

It is basically known (e.g. from \cite{Arthur2}) that if the feedback functions grow significantly slower than linear, then $\chi(\infty)$ is deterministic. We present an alternative approach to the sub-linear case, which allows some additional insights. For example, the case $F_i(k)=\log(k)$ is not included in the results of \cite{Arthur2}. In addition, our approach allows to construct feedback functions such that $\chi(n)$ does not even converge for $n\to\infty$. In order to get deterministic limits in our approach, we will need a condition, which ensures that the feedback functions grow slow enough. We will mainly use:
\begin{equation}
\label{eq: sublin}
\limsup_{k\to\infty}\frac{1}{k}F_i(k)\sum_{l=1}^k\frac{1}{F_i(l)}<\infty
\end{equation}

Note that this already implies that $i$ does not fulfill (\ref{eq: explosion}). We add some examples to gain an understanding of this restriction.

\begin{example}
\begin{enumerate}
\item For $F_i(k)=k(\log(k+1))^\alpha$ with $\alpha\in\R$ (\ref{eq: sublin}) is not fulfilled as $\sum_{l=1}^k\frac{1}{F_i(l)}\sim(\log k)^{1-\alpha}$.
\item If $F_i(k)=\alpha k^{\beta}$ for $\alpha>0, \beta<1$, then (\ref{eq: sublin}) is fulfilled as $\sum_{l=1}^k\frac{1}{F_i(l)}\sim\frac{k^{1-\beta}}{\alpha(1-\beta)}$.
\item For $F_i(k)=\log(k+1)$ (\ref{eq: sublin}) is fulfilled as $\sum_{l=1}^k\frac{1}{F_i(l)}\sim\frac{k}{\log(k)}$.
\item (\ref{eq: sublin}) is fulfilled if $\liminf_{k\to\infty}F_i(k)>0$ and $\limsup_{k\to\infty}F_i(k)<\infty$.
\end{enumerate}
\end{example}
In fact, condition (\ref{eq: sublin}) contains a monotonicity in the following sense.

\begin{proposition}
\label{prop: monotonicity}
If $F_i$ fulfills (\ref{eq: sublin}) and for some $j\neq i$
\begin{equation*}
\limsup_{x\to\infty}\frac{\frac{d}{dx}\log(F_j(x))}{\frac{d}{dx}\log(F_i(x))}<\infty,
\end{equation*}
then $F_j$ fulfills (\ref{eq: sublin}), too.
\end{proposition}

\begin{proof}
The assumption implies via (\ref{eq: incrementlog})
\begin{equation*}
\frac{F_i(k+l)}{F_i(k)}\ge const. \frac{F_j(k+l)}{F_j(k)}
\end{equation*}
for all $k, l\in\N$ and hence:
\begin{equation*}
\frac{1}{k}F_i(k)\sum_{l=1}^k\frac{1}{F_i(l)}=\frac{1}{k}\sum_{l=1}^k\frac{F_i(k)}{F_i(l)}\ge const.\frac{1}{k}\sum_{l=1}^k\frac{F_j(k)}{F_j(l)}= const.\frac{1}{k}F_j(k)\sum_{l=1}^n\frac{1}{F_j(l)}
\end{equation*}
\end{proof}

In general, our approach even allows feedback functions that converge to zero as long as this convergence is not to fast, which is ensured by the condition
\begin{equation}
\label{eq: sublin2}
\liminf_{k\to\infty}\frac{1}{k^p}F_i(n)\sum_{l=1}^k\frac{1}{F_i(l)}>0\quad\text{for some }p>\frac{1}{2}\ .
\end{equation}
Note that (\ref{eq: sublin2}) is fulfilled for any feedback function with $\liminf_{k\to\infty}F_i(k)>0$ as well as for $F_i(k)=k^{-\alpha}, \alpha>0$, but not for $F_i(k)=e^{-k}$. In analogy to Proposition \ref{prop: monotonicity} we get a monotonicity here in the sense that if $F_i$ fulfills (\ref{eq: sublin2}) and
\begin{equation*}
\limsup_{x\to\infty}\frac{\frac{d}{dx}\log(F_i(x))}{\frac{d}{dx}\log(F_j(x))}<\infty,
\end{equation*}
then $F_j$ fulfills (\ref{eq: sublin2}), too. We are now prepared for the main result of this section regarding the counting processes \eqref{eq:counting} of the exponential embedding from Section 2.

\begin{theorem}
\label{thm: asymptoticXi}
Let $F_i$ fulfill (\ref{eq: sublin}) and (\ref{eq: sublin2}). Then
\begin{equation*}
\frac{\Xi_i(t)}{a_i^{-1}(t)}\xrightarrow{t\to\infty}1\quad\text{almost surely}\ ,
\end{equation*}
where $a_i^{-1}$ denotes the inverse function of $a_i(t)\coloneqq\int_1^t\frac{dx}{F_i(x)}$.
\end{theorem}

Note that $a_i^{-1}$ exists as $a_i$ is strictly monotone. The asymptotics of birth processes have been studied in the literature before, e.g. in \cite{Barbour}. One main result of \cite{Barbour} will be used for a special case in Section \ref{sec: specialCase} to abandon condition (\ref{eq: sublin}). The following lemma provides the first step of the proof of Theorem \ref{thm: asymptoticXi}, using standard ideas from renewal theory.

\begin{lemma}
\label{lem: lemma3}
If $F_i$ fulfills (\ref{eq: sublin2}) for $p>\frac12$, then
\begin{equation*}
\frac{a_i(\Xi_i(t))}{t}\xrightarrow{t\to\infty} 1\quad\text{almost surely}\ .
\end{equation*}
\end{lemma}

\begin{proof}
(\ref{eq: sublin2}) implies
\begin{equation*}
\lim_{k\to\infty}\sum_{l=1}^kVar\left(\frac{\tau_i(l)}{a_i(l)}\right)=\lim_{k\to\infty}\sum_{l=1}^k\frac{1}{F_i(l)^2a_i(l)^2}\le\lim_{k\to\infty}\sum_{l=1}^k\frac{const.}{l^{2p}}<\infty\ ,
\end{equation*}
using $a_i(k)\sim\sum_{l=1}^k\frac{1}{F_i(l)}$ for $k\to\infty$. According to the Kolmogorov criterion (see e.g.\ \cite{Henze}, Section 6.2) this is sufficient for
\begin{equation*}
\frac{S_i(k)}{a_i(k)}\xrightarrow{k\to\infty}1\quad\text{almost surely},
\end{equation*}
where $S_i(k)\coloneqq\sum_{l=1}^k\tau_i(l)$. We use this and $\Xi_i (t)\to\infty\ a.s.$ for the final estimate:
\begin{align*}
1&=\lim_{t\to\infty}\frac{a_i(\Xi_i(t))}{S_i(\Xi_i(t))-S_i(X_i(0)-1)}\le\lim_{t\to\infty}\frac{a_i(\Xi_i(t))}{t}\le\lim_{t\to\infty}\frac{a_i(\Xi_i(t))}{S_i(\Xi_i(t)+1)-S_i(X_i(0)-1)}\\
&\le\lim_{t\to\infty}\frac{a_i(\Xi_i(t)+1)}{S_i(\Xi_i(t)+1)}=1
\end{align*}
\end{proof}

Now Theorem \ref{thm: asymptoticXi} is easy to prove.

\begin{proof}
Lemma \ref{lem: lemma3} states
\begin{equation*}
a_i(\Xi_i(t))=t+o(t)\quad\Leftrightarrow\quad\Xi_i(t)=a_i^{-1}(t+o(t))\quad\text{almost surely}
\end{equation*}
(using the Landau o-notation). It thus remains to show that
\begin{equation}\label{eq:toshowlem}
\lim_{t\to\infty}\frac{a_i^{-1}(t+o(t))}{a_i^{-1}(t)}=1.
\end{equation}
The condition (\ref{eq: sublin}) implies using $a_i(k)\sim\sum_{l=1}^k\frac{1}{F_i(l)}$ for $k\to\infty$
\begin{equation*}
\lim_{t\to\infty}o(a_i(t))\frac{F_i(t)}{t}=0
\end{equation*}
and hence, replacing $t$ by $a_i^{-1}(t)$ (note: $a_i^{-1}(t)\to \infty$), we get:
\begin{equation*}
\lim_{t\to\infty}o(t)\frac{F_i(a_i^{-1}(t))}{a_i^{-1}(t)}=0
\end{equation*}
Finally,
\begin{align*}
\log\left(\frac{a_i^{-1}(t+o(t))}{a_i^{-1}(t)}\right)&=\int_t^{t+o(t)}\frac{d}{dx}\log\left(a_i^{-1}(x)\right)dx=\int_t^{t+o(t)}\frac{\frac{d}{dx}a_i^{-1}(x)}{a_i^{-1}(x)}dx\\
&=\int_t^{t+o(t)}\frac{F_i(a_i^{-1}(t))}{a_i^{-1}(t)}dx\xrightarrow{t\to\infty}0\ ,
\end{align*}
which includes \eqref{eq:toshowlem}.
\end{proof}

Theorem \ref{thm: asymptoticXi} implies that the market shares in the exponential embedding are asymptotically given by
\begin{equation*}
\frac{\Xi_i(t)}{\Xi_1(t)+...+\Xi_A(t)}\sim\frac{a_i^{-1}(t)}{a_1^{-1}(t)+...+a_A^{-1}(t)} \quad\text{for } t\to\infty.
\end{equation*}

Via (\ref{eq: expemb}) we can now conclude for the discrete-time urn model.

\begin{corollary}
\label{cor: main4}
Let all agents fulfill (\ref{eq: sublin}) and (\ref{eq: sublin2}). If the limit
\begin{equation}
\label{eq: limshare}
\chi_i(\infty)\coloneqq\lim_{t\to\infty}\frac{a_i^{-1}(t)}{a_1^{-1}(t)+...+a_A^{-1}(t)}\in[0, 1]
\end{equation}
exists for an $i\in [A]$, then 
\begin{equation*}
\chi_i(n)\xrightarrow{n\to\infty}\chi_i(\infty)\quad\text{almost surely}\ .
\end{equation*}
If the limit in (\ref{eq: limshare}) does not exist, then $\chi_i(n)$ does not converge for $n\to\infty$.\\
If the limit in (\ref{eq: limshare}) exists for all $i\in [A]$, then $\chi (n)\xrightarrow{n\to\infty}\chi (\infty)\in\Delta_{A-1}$ almost surely.
\end{corollary}

Note that the $a_i^{-1}$ do not depend on $N$ and $\chi(0)$, thus the long time behavior of market shares $(\chi (n))_{n\in\N}$ in the generalized Pólya urn does not depend on initial conditions if (\ref{eq: sublin}) and (\ref{eq: sublin2}) are satisfied. If the limit in (\ref{eq: limshare}) exists, a market modeled by a Pólya urn under the assumptions of the corollary reveals stable and deterministic market shares in the long run and these market shares do not depend on the current market situation and can also take values in $(0,1)$. If the limit $\chi (\infty )$ exists it is in $\Delta_{A-1}$, since $\Delta_{A-1}$ is compact and therefore the laws of $\chi (n)$ form a tight sequence. The corollary provides a way to explicitly calculate these long-time market shares.

\begin{example}\label{ex46}
\begin{enumerate}
\item If $F_i(k)=\alpha_ik^\beta$ with $\alpha_i>0, \beta<1, i=1,...,A$, then
\begin{equation*}
a_i^{-1}(t)=\left(\alpha_i(1-\beta)t+1\right)^{\frac{1}{1-\beta}}
\end{equation*}
and hence:
\begin{equation*}
\chi_i(\infty)=\frac{\alpha_i^{\frac{1}{1-\beta}}}{\alpha_1^{\frac{1}{1-\beta}}+...+\alpha_A^{\frac{1}{1-\beta}}} \in (0,1)
\end{equation*}
Consequently, the impact of the fitness parameters $\alpha_i$ in the long-time limit increases with $\beta$, where
\begin{equation*}
    \chi_i (\infty)\to \frac{1}{A}\quad\text{for }\beta \to -\infty\quad\mbox{and}\quad\chi_i(\infty)=\frac{\alpha_i}{\alpha_1+...+\alpha_A}\quad\text{for }\beta =0\ .
\end{equation*}
The limiting case $\beta\to 1$ will be discussed later in Proposition \ref{prop49}.

\item If $F_i(k)=\alpha_i\log(k+1)$ with $\alpha_i>0, i=1,...,A$, then
\begin{equation*}
a_i^{-1}(t)\sim\alpha_it\log(\alpha_it)\quad\text{for   }t\to\infty
\end{equation*}
and thus:
\begin{equation*}
\chi_i(\infty)=\frac{\alpha_i}{\alpha_1+...+\alpha_A}
\end{equation*}
Note that this is the same asymptotic market share as if the customers' decisions were independent (with constant feedback functions as for $\beta =0$ above), so that the strong law of large numbers applies.
\end{enumerate}
\end{example}

It is also possible to find examples where the limit (\ref{eq: limshare}) does not exist. In the following situation the market share of the agents oscillates with constant amplitude but increasing period.

\begin{example}
\label{ex: noConv}
Take $A=2$ and set
\begin{equation*}
a_1^{-1}(t)=t^2\left(\sin(\log(t))+2\right)\quad\text{and}\quad a_2^{-1}(t)=t^2.
\end{equation*}
This corresponds to $F_2(t)=\sqrt{t}$ and $F_1(t)=\left(\frac{d}{dt}(a_1^{-1})^{-1}(t)\right)^{-1}$, which is well defined due to $\frac{d}{dt}a_1^{-1}(t)=t\left(2\sin(\log(t))+\cos(\log(t))+4\right)>0$. Then Theorem \ref{thm: asymptoticXi} implies
\begin{equation*}
\frac{\Xi_1(t)}{\Xi_1(t)+\Xi_2(t)}\sim\frac{\sin(\log(t))+2}{\sin(\log(t))+3}
\end{equation*}
and hence $\chi_1(n)$ oscillates between $1/2$ and $3/4$.
\end{example}

We now add a criterion that ensures the existence of the limit in (\ref{eq: limshare}).

\begin{corollary}\label{cor: LimShare}
Suppose that for an agent $i$ the following tightening of (\ref{eq: sublin}) holds,
\begin{equation}
\label{eq: sublin+}
\frac{F_i(k)}{k}\sum_{l=1}^k\frac{1}{F_i(l)}\xrightarrow{k\to\infty}c\in(0, \infty)\ ,
\end{equation}
and that the limits
\begin{equation}
\label{eq: relationasym}
\lim_{k\to\infty}\frac{F_i(k)}{F_j(k)} =c_{j}\in[0, \infty]\quad\text{exist for all }j\neq i\ .
\end{equation}
Then the limit in (\ref{eq: limshare}) exists and 
\begin{equation*}
\chi_i(\infty)=\left(1+\sum_{j\ne i}c_j^{-c}\right)^{-1} \ .
\end{equation*}
In particular, $\P(wMon_i(\chi(0), N))=1$ if and only if all $c_j$ are infinity, otherwise \linebreak $\P(wMon_i(\chi(0), N))=0$. If all $c_j$ are one, then the condition (\ref{eq: sublin+}) can be replaced by (\ref{eq: sublin}) and $\chi_j(\infty)=1/A$ for all $j=1,...,A$.
\end{corollary}

\begin{proof}
Recall that $a_i(t)=\int_1^t\frac{dx}{F_i(x)}$ is strictly increasing. For a fixed $j\ne i$ we show that $\Xi_j(t)/\Xi_i(t)$ converges to $c_j^{-1}$. First, we assume $0<c_j<\infty$, such that agents $i$ and $j$ fulfill (\ref{eq: sublin}) and (\ref{eq: sublin2}). (\ref{eq: relationasym}) implies via the theorem of de l'Hospital $a_j(t)/a_i(t)\to c_j$ for $t\to\infty$ and consequently $a_j^{-1}(t)=a_i^{-1}(\delta(t)t)$ for a function $\delta$ with $\delta(t)\xrightarrow{t\to\infty}1/c_j<\infty$. In combination with Theorem \ref{thm: asymptoticXi} it remains to show that $a_i^{-1}(\delta(t)t)/a_i^{-1}(t)$ converges to $c_j^{-1}$ for $t\to\infty$. For this we consider
\begin{align*}
\log\left(\frac{a_i^{-1}(\delta(t)t)}{a_i^{-1}(t)}\right)&=\int_t^{\delta(t)t}\frac{d}{dx}\log a_i^{-1}(x)dx=\int_t^{\delta(t)t}\frac{F_i(a_i^{-1}(x))}{a_i^{-1}(x)}dx\\
&\sim \int_t^{\delta(t)t}\frac{c}{x}dx=c\log(\delta(t))
\end{align*}
as (\ref{eq: sublin+}) implies via time-shift
\begin{equation*}
\frac{F_i(a_i^{-1}(t))}{a_i^{-1}(t)}\sim\frac{c}{t}\quad\text{for   }t\to\infty.
\end{equation*}
Thus:
\begin{equation*}
\frac{a_j^{-1}(t)}{a_i^{-1}(t)}=\frac{a_i^{-1}(\delta(t)t)}{a_i^{-1}(t)}\sim\delta (t)^c\xrightarrow{t\to\infty}c_j^{-c}
\end{equation*}
For agents $j$ with $c_j=0$ the asymptotic market share is for sure bigger than in a situation where $F_j$ is replaced by $ CF_i, C>0$,  i.e. $\Xi_j(t)/\Xi_i(t)$  is for $t\to\infty$ larger than any $C$. Hence, it converges to infinity. Similarly for agents with $c_j=\infty$.
\end{proof}

Note that in the case $c=1$ (including e.g.\ feedback functions such as $\log k$, $1/\log k$ or functions converging in $(0,\infty)$)  the limit $\chi_i (\infty )$ is equal to the case $F_i(k)=const.$, i.e. draws from the urn are independent and the usual strong law of large numbers applies. So this weak reinforcement does not play any role in the long run.\\

So far, we did not consider cases near the classical P\'olya urn with $F_i(k)=k$, where random limits $\chi_i (\infty )$ are possible. Nevertheless, as Lemma \ref{lem: lemma3} does not require (\ref{eq: sublin}), our approach provides some insight into such asymmetric cases as well. 
The symmetric case with feedback functions close to the classical Pólya urn is treated in Section \ref{sec: specialCase}.

\begin{proposition}\label{prop49}
Let an agent $i$ fulfill
\begin{equation}
\label{eq: lin}
\frac{F_i(k)}{k}\sum_{l=1}^k\frac{1}{F_i(l)}\xrightarrow{k\to\infty}\infty,
\end{equation}
but not (\ref{eq: explosion}), i.e.\quad $\sum_{k=1}^\infty \frac{1}{F_i (k)}=\infty$. If 
\begin{equation}\label{eq: lin2}
\limsup_{k\to\infty}\frac{F_j(k)}{F_i(k)}<1\quad\text{for all agents }j\ne i\ ,
\end{equation}
then $\P(wMon_i(\chi(0), N))=1$.
\end{proposition}

\begin{proof}
First note that via exponential embedding, the event $ wMon_i(\chi(0), N)$ is equivalent to $\Xi_i (t)/\Xi_j (t)\xrightarrow{t\to\infty}\infty$ for all $j\neq i$. 
Obviously, agent $i$ fulfills (\ref{eq: sublin2}). First, we assume that agent $j$ does, too. Define
\begin{equation*}
\psi(t)\coloneqq\int_0^t\frac{e^x}{F_i(e^x)}dx=\int_1^{e^t}\frac{1}{F_i(x)}dx=a_i(e^t)
\end{equation*}
and thus $a_i^{-1}(t)=e^{\psi^{-1}(t)}$. Assumption \eqref{eq: lin2} implies that for any $j\ne i$ there is a constant $c<1$ with $a_i(t)\le ca_j(t)$ for large enough $t$ and consequently $a_j^{-1}(t)\le a_i^{-1}(ct)$. Lemma \ref{lem: lemma3} states that $\Xi_i(t)=a_i^{-1}(t+o(t))$ and $\Xi_j(t)=a_j^{-1}(t+o(t))\le a_i^{-1}(ct+o(t))$ almost surely. Thus it remains to show that
\begin{equation*}
\frac{a_i^{-1}(t+o(t))}{a_i^{-1}(ct+o(t))}\xrightarrow{t\to\infty}\infty,
\end{equation*}
which is equivalent to $\psi^{-1}(t+o(t))-\psi^{-1}(ct+o(t))\xrightarrow{t\to\infty}\infty$. It is sufficient that
\begin{equation*}
t\cdot\frac{d}{dt}\psi^{-1}(t)=t\cdot\frac{F_i\left(e^{\psi^{-1}(t)}\right)}{e^{\psi^{-1}(t)}}=t\frac{F_i\left(a_i^{-1}(t)\right)}{a_i^{-1}(t)}\xrightarrow{t\to\infty}\infty,
\end{equation*}
which follows since $a_i^{-1} (t)\xrightarrow{t\to\infty}\infty$ and Assumption \eqref{eq: lin} is equivalent to $\frac{F_i(t)}{t}a_i(t)\xrightarrow{t\to\infty}\infty$.

If agent $j$ does not fulfill (\ref{eq: sublin2}), then $F_j$ is bounded from above and hence $\Xi_j$ is stochastically dominated by a homogeneous Poisson process (with constant rate). Consequently, $\Xi_j (t)$ grows asymptotically not faster than linear and hence $\Xi_i (t)/\Xi_j (t)\to\infty$ almost surely.
\end{proof}

Condition \eqref{eq: lin} includes feedback functions of the form $F_i(k)=\alpha_i k(\log k)^\beta$ for all $\beta\le 1$, including the linear case $F_i (k)=\alpha_ik$ for $\beta =0$. If in addition \eqref{eq: lin2} holds, i.e. $\alpha_i>\alpha_j$ for an agent $i$ and all $j\ne i$, then we have an almost sure weak monopoly for agent $i$. This is consistent with the strong monopoly for $\beta>1$ as described in Example \ref{example: klogk, beta>1}. Note that the weak monopoly in Proposition \ref{prop49} is almost sure even for finite $N$, in contrast to the results on strong monopoly derived in Section \ref{sec: monopoly}, where the strong monopolist is random and can only be predicted in the limit $N\to\infty$.

On the other hand, condition (\ref{eq: sublin}) includes sublinear feedback functions of the form $F_i (k) =\alpha_i k^\beta$ with $\beta <1$, which have positive long-time market shares for all agents as discussed in Example \ref{ex46}.


Exponentially decreasing feedback functions were not taken into account so far as they do not fulfill (\ref{eq: sublin2}). Since such cases do not seem to be of great importance for the mentioned interpretations of the model, we are content with an example, which we discuss in Appendix \ref{sec: expDecreasingF} using the method of stochastic approximation.

We conclude the presentation with a short overview of further related results. 
\cite{Oliveira4,Oliveira3,Kearney} discuss another change of behaviour that is not apparent from our approach. Consider the case
$$
F_1=F_2=...=F_A\ ,\quad\liminf_{k\to\infty}F_i(k)>0\quad\mbox{and}\quad\sum_{k=1}^\infty\frac{1}{F_i(k)}=\infty\ .
$$
If $\sum\limits_{k=1}^\infty\frac{1}{F_i(k)^2}<\infty$, e.g. for $F_i (k)=k$, then the leading agent changes only finitely often with probability one, whereas in the case $\sum\limits_{k=1}^\infty\frac{1}{F_i(k)^2}=\infty$ this probability is zero.

For $F_1(k)=...=F_A(k)=k^\beta$,  it is shown in \cite{Khanin} that $\chi_i(n)$ converges to $1/A$ at rate $n^{\beta-1}$ for $1/2<\beta<1$ (almost surely), at rate $n^{-1/2}$ for $0<\beta<1/2$ and at rate $\sqrt{\log(n)/n}$ for $\beta=1/2$ (in a weak sense). For $0<\beta\le1/2$, a central limit theorem holds.

In the case $A=2$ and $F_i(k)=\alpha_ik^\beta$ \cite{Jiang} derives the tail distributions of the number and last times of ties $X_1(n)=X_2(n)$.


\section{Feedback functions close to the classical Pólya urn}
\label{sec: specialCase}

We know from Theorem \ref{thm: theorem2} that a generalized Pólya urn reveals strong monopoly if and only if at least one feedback function grows significantly faster than linear, i.e. fulfills (\ref{eq: explosion}). As described in Section \ref{sec: non-monopoly}, linear feedback functions imply random long-time market shares, whereas a deterministic limit occurs for feedback functions growing significantly slower than linear, i.e. those fulfilling (\ref{eq: sublin}). Nevertheless, some feedback functions that are close to linear (like $F_i(k)=k(\log k)^\beta,\,\beta\ne0$) are not covered by our results so far. To our knowledge, the literature does not provide results on the long time behaviour of a generalized Pólya urn with almost linear feedback. For instance, if $F_i(k)=kL(k)$ for a slowly varying function $L$, then Theorem \ref{thm: Arthur} does not determine the long-time limit, since $\lim_{N\to\infty}p(N, x)=x$ for all $x\in\Delta_{A-1}$. We approach this question exploiting general results on birth processes, which require that $F_i$ does not fulfill (\ref{eq: explosion}) but inverted squares are summable, i.e.
\begin{equation}
\label{eq: squareSum}
    \sum_{k=1}^\infty\frac{1}{F_i(k)}=\infty\quad\mbox{and}\quad\sigma_i^2\coloneqq\sum_{k=X_i(0)}^\infty\frac{1}{F_i(k)^2}<\infty .
\end{equation}
Recall the exponential embedding from Section \ref{sec: model} and notations introduced therein. For this section, it is convenient to adapt previous definitions using
\begin{equation}\label{eq:aidef}
a_i(t)\coloneqq\int_{X_i(0)}^{X_i(0)+t}\frac{dx}{F_i(x)}\quad\text{and}\quad S_i(k)\coloneqq\sum_{l=X_i(0)}^{X_i(0)+k}\tau_i(l),
\end{equation}
and to extend $F_i$ on $(0, \infty)$ by a right-continuous step function. The key to the desired results is provided by the following result in \cite{Barbour}.

\begin{theorem}
\label{thm: Barbour}
\cite[Theorem 3.3', Theorem 3.4, Lemma 3.1]{Barbour}
Assume that $F_i$ fulfills (\ref{eq: squareSum}). 
Then $t-a_i(\Xi_i(t))$ and $S_i(k)-a_i(k)$  converge almost surely for $t\to\infty$ resp. $k\to\infty$ to the same random variable $U_i \in\R$. Moreover, $\sigma_i^2$ is the variance of $U_i$.
\end{theorem}

We can now apply this general result in our situation.

\begin{corollary}
\label{cor: specialCase1}
Assume that $F_i$ fulfills (\ref{eq: squareSum}). 
Then:
\begin{enumerate}
    \item If $\lim_{k\to\infty}\frac{F_i(k)}{k}=0$, then
    \begin{equation*}
        \frac{\Xi_i(t)}{a_i^{-1}(t)}\xrightarrow{t\to\infty}1\quad\text{almost surely.}
    \end{equation*}
    \item If $\lim_{k\to\infty}\frac{F_i(k)}{k}=c\in(0, \infty)$, then
    \begin{equation*}
        \frac{\Xi_i(t)}{a_i^{-1}(t)}\xrightarrow{t\to\infty}e^{-cU_i}\quad\text{almost surely.}
    \end{equation*}
    \item If $\lim_{k\to\infty}\frac{F_i(k)}{k}=\infty$, then
    \begin{equation*}
       \frac{\Xi_i (t)}{a_i^{-1}(t)}=\exp\left(\int_t^{t-U_i-o(1)}h_i(s)ds\right)
    \end{equation*}
    for a (deterministic) function $h_i$ with $\lim_{s\to\infty}h_i(s)=\infty$.
\end{enumerate}
\end{corollary}

\begin{proof}
Theorem \ref{thm: Barbour} implies $t-a_i(\Xi_i(t))=U_i+o(1)$ and hence 
$$\Xi_i(t)=a_i^{-1}(t-U_i+o(1))$$
Using $\frac{d}{dt}a_i^{-1}(t)=F_i(a_i^{-1}(t) +X_i(0))$ in the logarithmic derivative yields
\begin{align}\label{eq: a_iderivative}
    a_i^{-1}(t)=\exp\left(\int_0^t\frac{d}{ds}\log\left(a_i^{-1}(t)\right)ds\right)=\exp\left(\int_0^t\frac{F_i(a_i^{-1}(s)+X_i(0))}{a_i^{-1}(s)}ds\right)
\end{align}
and consequently
\begin{align*}
    \frac{\Xi_i (t)}{a_i^{-1}(t)} =\frac{a_i^{-1}(t-U_i-o(1))}{a_i^{-1}(t)}=\exp\left(\int_t^{t-U_i-o(1)}\frac{F_i(a_i^{-1}(s)+X_i(0))}{a_i^{-1}(s)}ds\right) .
\end{align*}
Now, be aware that $\lim_{t\to\infty}a_i^{-1}(t)=\infty$ as $F_i$ does not fulfill (\ref{eq: explosion}) and that the limit of $F(k+const.)/k$ for $k\to\infty$ is equal to the limit of $F(k)/k$. Then all parts of the corollary follow directly from their assumptions. 
\end{proof}

Like in Section \ref{sec: non-monopoly}, we can now conclude from the exponential embedding to the evolution of market shares in the Pólya urn via
\begin{equation*}
    \lim_{n\to\infty}\frac{\chi_i(n)}{\chi_1(n)+\ldots\chi_A(n)}=\lim_{t\to\infty}\frac{\Xi_i(t)}{\Xi_1(t)+\ldots\Xi_A(t)} ,
\end{equation*}
provided that the limit exists.

We are now particularly interested in cases with equal feedback functions for all agents, since agents with different attractiveness are already covered by Proposition \ref{prop49}. 

\begin{corollary}\label{cor: almostLinear}
Assume that all agents have the same feedback function $F_i \equiv F$ and that $F$ fulfills (\ref{eq: squareSum}). 
Then for all $\chi (0)\in\Delta_{A-1}^o$:
\begin{enumerate}
    \item If $\lim_{k\to\infty}\frac{F(k)}{k}=0$, then
    \begin{equation*}
        \chi_i(n)\xrightarrow{n\to\infty}\frac{1}{A}\quad\text{almost surely, for all }i\in [A]\ .
    \end{equation*}
    \item If $\lim_{k\to\infty}\frac{F(k)}{k}=c\in(0, \infty)$, then the limit $\chi (\infty )=\lim_{n\to\infty}\chi(n)$ exists almost surely and has a non-degenerate Dirichlet distribution on $\Delta_{A-1}$.
    \item If $\lim_{k\to\infty}\frac{F(k)}{k}=\infty$, then $\chi (\infty )=\lim_{n\to\infty}\chi(n)$ exists almost surely and the process exhibits a weak monopoly, i.e
    $$\P\left(\bigcup_{i=1}^AwMon_i(\chi(0), N)\right)=1\quad\mbox{such that}\quad \P\left( wMon_i(\chi(0), N)\right) >0\mbox{ for all }i\in [A].$$
\end{enumerate}
\end{corollary}

In other words: If the feedback function grows any slower than the identity, then the market shares converge to a deterministic limit as time tends to infinity, and the limit does not depend on the initial condition. If the feedback functions grow any faster than the identity, the process exhibits weak monopoly, which is not strong as (\ref{eq: explosion}) is necessary in Theorem \ref{thm: theorem2}. The weak monopoly can been seen in the Simulation shown in Figure \ref{figure: Simulation} (d). In contrast to the non-symmetric situation of Proposition \ref{prop49}, the monopolist is random with probability depending on the initial condition $\chi (0)$.

\begin{proof}
Note that the $U_i$ from Theorem \ref{thm: Barbour} are independent with distribution depending on $\chi(0)$ and $N$. In addition, their distribution is continuous as $U_i$ emerges from a sum of independent, centered exponentially distributed random variables. 
By definition \eqref{eq:aidef} we get $a_i(t)=a_j (t+const.)+const.$ with constants depending on the initial conditions and $F$, and after inversion we have $a_i^{-1}(t)=a_j^{-1}(t+const.)+const.$ for all $i, j\in[A]$. With (\ref{eq: a_iderivative}) this implies
\begin{equation}\label{eq:paff}
a_i^{-1} (t) =a_j^{-1} (t) \exp\left(\int_t^{t+const.}\frac{F(a_j^{-1}(s)+X_j(0))}{a_j^{-1}(s)}ds\right) +const.\ ,    
\end{equation}
and note that $a_j^{-1} (t)\to\infty$ in all cases.

1. In this case \eqref{eq:paff} implies that $a_i^{-1}(t)\sim a_j^{-1}(t)$. Then the claim follows directly from Corollary \ref{cor: specialCase1} via
\begin{equation*}
    \lim_{n\to\infty}\frac{\chi_i(n)}{\chi_1(n)+\ldots\chi_A(n)}=\lim_{t\to\infty}\frac{\Xi_i(t)}{\Xi_1(t)+\ldots\Xi_A(t)}=\lim_{t\to\infty}\frac{a_i^{-1}(t)}{a_1^{-1}(t)+\ldots a_A^{-1}(t)}=\frac{1}{A} \ .
\end{equation*}

2. Here, again with \eqref{eq:paff}, $a_i^{-1}(t)/a_j^{-1}(t)$ converges to a finite, non-zero constant for all $i, j\in[A]$, such that Corollary \ref{cor: specialCase1} yields
\begin{align*}
   \lim_{t\to\infty}\frac{\Xi_i(t)}{\Xi_1(t)+\ldots\Xi_A(t)}&=\lim_{t\to\infty}\frac{e^{-cU_i}a_i^{-1}(t)}{e^{-cU_1}a_1^{-1}(t)+\ldots e^{-cU_A}a_A^{-1}(t)}\\
   &=\left(1+\sum_{j\ne i}e^{c(U_i-U_j)}\lim_{t\to\infty}\frac{a_j^{-1}(t)}{a_i^{-1}(t)}\right)^{-1}.
\end{align*}
Hence, $\lim_{n\to\infty}\chi(n)$ exists almost surely and has a continuous distribution, which is of Dirichlet type according to Proposition \ref{prop: partialProcess} and \cite{James}.

3. Due to Lemma \ref{lem: wMon}, we can assume $X(0)=(1,\ldots, 1)$, so that $a_i=a_j$ and $h_i=h_j$. Then:
\begin{align*}
   \lim_{t\to\infty}&\frac{\Xi_i(t)}{\Xi_1(t)+\ldots\Xi_A(t)}=\lim_{t\to\infty}\left(1+\sum_{j\ne i}\frac{\Xi_j(t)}{\Xi_i(t)}\right)^{-1}\\
   &=\lim_{t\to\infty}\left(1+\sum_{j\ne i}\exp\left(\int_t^{t-U_j-o(1)}h_j(s)ds-\int_t^{t-U_i-o(1)}h_i(s)ds\right)\right)^{-1}\\
   &=\lim_{t\to\infty}\left(1+\sum_{j\ne i}\exp\left(\int_{t-U_i-o(1)}^{t-U_j-o(1)}h_i(s)ds\right)\right)^{-1}=
   \begin{cases}
       1 &\text{if }U_i<U_j\text{ for all }j\ne i\\
       0 &\text{else}
   \end{cases}
\end{align*}
Recall that $\lim_{s\to\infty}h_i(s)=\infty$. Again by Lemma \ref{lem: cgfU}, the unboundedness of the $U_i$ implies that $\P\left( wMon_i(\chi(0), N)\right)>0$ for all $i\in[A]$.
\end{proof}

\begin{lemma}\label{lem: wMon}
    For all choices of $F_1,\ldots,F_A$, we have
    \begin{align*}
        &\P\left(\bigcup_{i=1}^A wMon_i\left(\left(\frac1A,\ldots,\frac1A\right), A\right)\right)=1\\
        &\Leftrightarrow\quad\P\left(\bigcup_{i=1}^AwMon_i(\chi(0), N)\right)=1\text{ for all }\chi(0)\in\Delta_{A-1}^o,\,N\in\N.
    \end{align*}
\end{lemma}

\begin{proof}
    The implication $\Leftarrow$ is trivial. Thus, assume that the process $X(n)$ starts in $X(0)=(1,\ldots, 1)$ and that $\P\left(\bigcup_{i=1}^A wMon_i\left(\left(\frac1A,\ldots,\frac1A\right), A\right)\right)=1$. Moreover, take any $x\in\Delta_{A-1}^o,\,M\in\N$. Then the claim follows directly from the Markov property,
    $$1=\P\left(\bigcup_{i=1}^A wMon_i\left(\left(\frac1A,\ldots,\frac1A\right), A\right)\,\big|\, X(M-A)=Mx\right)=\P\left(\bigcup_{i=1}^AwMon_i(x, M)\right)\ ,$$
since $\P \left( X(M-A)=Mx\right) >0$\ .
\end{proof}

The following example presents a class of feedback functions, for which four different regimes are possible.

\begin{example}
Let $F_i(k)=k(\log k)^\beta$ for all $i\in[A]$ and $\beta\in\R$. Depending on $\beta$, four different regimes occur for $n\to\infty$:
\begin{enumerate}
    \item For $\beta<0$, $\chi_i (n)$ for each agent $i$ converges almost surely to $\frac{1}{A}$ independently of $\chi (0)$.
    \item For $\beta=0$, the market shares $\chi (n)$ converge almost surely to a random limit $\chi (\infty )\in\Delta_{A-1}$, which is not a corner point and its distribution depends on the initial condition $\chi (0)$.
    \item For $\beta\in(0, 1]$, the process exhibits a weak monopoly which is not strong, i.e. all agents win in infinitely many steps, but the market share of one agent converges to one. The monopolist is random, and the distribution of $\chi (\infty )$ on the corner points of $\Delta_{A-1}$ depends on $\chi (0)$.
    \item For $\beta>1$, there is a strong monopoly. The monopolist is random and the distribution of $\chi (\infty )$ on the corner points of $\Delta_{A-1}$ depends on the initial condition $\chi (0)$ as well.
\end{enumerate}
\end{example}

According to Theorem \ref{thm: Barbour}, we have $t_n-a_i(X_i(n))\xrightarrow{n\to\infty}U_i$ by definition of the exponential embedding with jump times $t_n$. If $\lim_{k\to\infty}F(k)/k=\infty$, this convergence can be specified by replacing $t_n$ by a deterministic function and by computing the distribution of $U_i$.

\begin{theorem}\label{thm: almostLinear2}
    Assume that $F_i \equiv F$ does not fulfill (\ref{eq: explosion}) and that $F(k)/k\xrightarrow{k\to\infty}\infty$ holds. Then there exist independent random variables $U_1,\ldots, U_A$ such that
    $$\left(a_i(n)-a_i(X_i(n))\right)_{i\in[A]}\xrightarrow{n\to\infty}\left(U_i-\sum_{k=1}^{X_i(0)-1}\frac{1}{F(k)}-\min_{j\in[A]}\left(U_j-\sum_{k=1}^{X_j(0)-1}\frac{1}{F(k)}\right)\right)_{i\in[A]}$$
     almost surely. Moreover, the cumulant generating function (CGF) of each $U_i$ is given by
    \begin{align}\label{eq: CGF}
    \lambda\mapsto\log\left(\E e^{\lambda U_i}\right)=\sum_{l=2}^\infty\frac{\lambda^l}{l}\sum_{k=X_i(0)}^\infty\frac{1}{F(k)^l}
    \end{align}
    and the radius of convergence is $\min_{k\ge X_i(0)}F(k)$.
\end{theorem}

In particular, there is exactly one agent, namely the weak monopolist, such that the limit of $a_i(n)-a_i(X_i(n))$ is zero. 
For the proof, we characterize the distribution of $U_i$ by computing its CGF. For that, we exploit that $U_i$ is also the limit of $S_i(k)-a_i(k)$  for $k\to\infty$ according to Theorem \ref{thm: Barbour}.

\begin{lemma}\label{lem: cgfU}
    Assume that $F_i \equiv F$ fulfills (\ref{eq: squareSum}). Then the CGF of $U_i$ is given by (\ref{eq: CGF}) and the radius of convergence is $\min_{k\ge X_i(0)}F(k)$.
\end{lemma}

\begin{proof}
The CGF of the limit $U_i=\lim_{k\to\infty}S_i(k)-a_i(k)=\lim_{k\to\infty}\sum_{l=X_i(0)}^k\left(\tau_i(l)-\frac{1}{F(l)}\right)$ is the pointwise limit of the CGFs:
\begin{align*}
    \log\left(\E e^{\lambda U_i}\right)&=\lim_{k\to\infty} \log\left(\E\exp\left(\lambda\sum_{l=X_i(0)}^k\left(\tau_i(l)-\frac{1}{F(l)}\right)\right)\right)\\
    &= \log\left(\prod_{k=X_i(0)}^\infty e^{-\lambda/F(k)}\E e^{\lambda\tau_i(k)}\right)=\sum_{k=X_i(0)}^\infty\left[\log\left(\E e^{\lambda\tau_i(k)}\right)-\frac{\lambda}{F(k)}\right]\\
    &=\sum_{k=X_i(0)}^\infty\left[\log\left(\frac{F(k)}{F(k)-\lambda}\right)-\frac{\lambda}{F(k)}\right]=-\sum_{k=X_i(0)}^\infty\left[\log\left(1-\frac{\lambda}{F(k)}\right)+\frac{\lambda}{F(k)}\right]
\end{align*}
We now use the series representation of $x\mapsto\log(1+x)$ and change the order of summation due to absolut convergence:
\begin{align*}
     \log\left(\E e^{\lambda U_i}\right)&=\sum_{k=X_i(0)}^\infty\left[\sum_{l=1}^\infty\frac{1}{l}\left(\frac{\lambda}{F(k)}\right)^l-\frac{\lambda}{F(k)}\right]=\sum_{k=X_i(0)}^\infty\sum_{l=2}^\infty\frac{1}{l}\left(\frac{\lambda}{F(k)}\right)^l\\
     &=\sum_{l=2}^\infty\frac{\lambda^l}{l}\sum_{k=X_i(0)}^\infty\frac{1}{F(k)^l}
\end{align*}

Note that $\sum_{k=X_i(0)}^\infty\frac{1}{F(k)^l}<\infty$ for $l\ge 2$. Now, define $M\coloneqq\{k\ge X_i(0)\colon F(k)=\min_{l\ge X_i(0)}F(l)\}$ and let $k_0\in M$. The radius of convergence of the power series representation of the CGF is given by
\begin{align*}
    \liminf_{l\to\infty}\left(\frac 1l\sum_{k=X_i(0)}^\infty\frac{1}{F(k)^l}\right)^{-1/l}=F(k_0) \liminf_{l\to\infty}\left(\# M+\sum_{k\ge X_i(0),\, k\notin M}\frac{F(k_0)^l}{F(k)^l}\right)^{-1/l}=F(k_0)
\end{align*}
since $\sum_{k\ge X_i(0),\, k\notin M}\frac{F(k_0)^l}{F(k)^l}\xrightarrow{l\to\infty}0$ and $\# M<\infty$ if $F(k)/k\xrightarrow{k\to\infty}\infty$.
\end{proof}

In particular, $\E U_i=0$ since the first term in the series is $\lambda^2$, and the $l$-th cumulant of $U_i$ is $(l-1)!\sum_{k=X_i(0)}^\infty\frac{1}{F(k)^l}$ for $l\ge2$. For the proof of Theorem \ref{thm: almostLinear2}, it remains to show that $t_n-a_i(n)$ converges as desired.

\begin{lemma}
In the situation of Theorem \ref{thm: almostLinear2} we have
    $$t_n-a_i(n)\xrightarrow{n\to\infty}\min_{j\in[A]}(U_j-c_{i, j})\,,$$
    where $c_{i, j}\coloneqq\sum_{k=1}^{X_j(0)-1}\frac{1}{F(k)}-\sum_{k=1}^{X_i(0)-1}\frac{1}{F(k)}$ for $i, j\in[A]$.
\end{lemma}

\begin{proof}
    By definition of $t_n$ and $a_i(n)$ and by Theorem \ref{thm: Barbour}, we have
    \begin{align*}
        t_n-a_i(n)&\le \min_{j\in[A]}S_j(n)-a_i(n)\\
        &=\min_{j\in[A]}\left(S_j(n)-a_j(n)+a_j(n)-a_i(n)\right)\xrightarrow{n\to\infty}\min_{j\in[A]}(U_j-c_{i, j})
    \end{align*}
as $a_j(n)-a_i(n)\xrightarrow{n\to\infty}c_{i, j}$. Furthermore, 
 \begin{align*}
        &t_n-a_i(n)\ge \min_{j\in[A]}S_j(n/A)-a_i(n)\\
        &= \min_{j\in[A]}\left(S_j(n/A)-a_j(n/A)+a_j(n/A)-a_j(n)+a_j(n)-a_i(n)\right)\xrightarrow{n\to\infty}\min_{j\in[A]}(U_j-c_{i, j})\,.
    \end{align*}
This holds because
\begin{align*}
    a_j(n)-a_j(n/A)=\int_{X_j(0)+n/A}^{X_j(0)+n}\frac{1}{F(s)}ds= \left(n-\frac nA\right)\frac{1}{F(m_n)}\xrightarrow{n\to\infty}0
\end{align*}
for a mean value $m_n\in(X_j(0)+n/A,\, X_j(0)+n)$ using $F(k)/k\xrightarrow{k\to\infty}\infty$.
\end{proof}

According to Theorem \ref{thm: almostLinear2}, $X_i(n)$ is asymptotically well described by 
$$X_i(n)\approx a_i^{-1}\left(a_i(n)-\tilde U_i+\min_{j\in[A]}\tilde U_j\right)\,,$$
where $\tilde U_j\coloneqq U_j-\sum_{k=1}^{X_j(0)-1}\frac{1}{F(k)}$ for all $j\in[A]$. Now, consider two distinct agents $i, j$ and assume for simplicity of notation that $X_i(0)=X_j(0)$, such that $a_i(k-X_i(0))=a_j(k-X_j(0))=:a(k)$. Then Theorem \ref{thm: almostLinear2} states that
$$a(X_i(n))-a(X_j(n))\xrightarrow{n\to\infty}U_j-U_i\quad\text{almost surely.}$$
Moreover, the CGF of $U_j-U_i$ is the sum of the CGFs of $U_j$ and $U_i$ due to independence. Hence, $\E e^{\lambda(U_j-U_i)}$ is finite if and only if  $|\lambda|<\min_{k\ge X_i(0)}F(k)$. 
Thus, the distribution of $U_j-U_i$ has exponential tails, and these findings can be used as follows.

\begin{example}
    Let $F_i (k)\equiv F(k)=k\log(k)$ and $X_i(0)=X_j(0)=1$ for two agents $i, j\in[A]$, so that $a_i(t)=a_j(t)=\log\log t$. Then the continuous mapping theorem yields
    $$\frac{\log X_i(n)}{\log X_j(n)}\xrightarrow{n\to\infty}e^{U_j-U_i}\quad\text{almost surely,}$$
    where $e^{U_j-U_i}$ has a power-law distribution due to the explanations above. 
    Remarkably, the log-ratios $\frac{\log X_i(n)}{\log X_j(n)}$ and $\frac{\log X_{i'}(n)}{\log X_{j'}(n)}$ are asymptotically also independent for distinct pairs of agents $(i, j), (i', j')$.
\end{example}

An important application of Theorem \ref{thm: almostLinear2} is its implication for the rate of convergence. In fact, the convergence of the process of market shares $\chi(n)$ to an edge of the simplex can be considered as logarithmically slow.

\begin{corollary}\label{cor: almostLinear2}
    Assume that $F_i =F$ and $L(k)\coloneqq F(k)/k$ is increasing, but (\ref{eq: explosion}) does not hold. Then there is a random constant $c>0$ such that 
    $$\chi_i(n)\ge e^{-cL(n)}\quad\mbox{for all $n\geq 1$ and $i\in [A]$.}$$
\end{corollary}

\begin{proof}
   Since the limit in Theorem \ref{thm: almostLinear2} is finite, there is a constant $c>0$ such that
   \begin{align*}
       c&\ge\int_{X_i(n)}^{n+X_i(0)}\frac{1}{F(s)}ds=\int_{\chi_i(n)}^{1+X_i(0)/n}\frac{1}{sL(ns)}ds\ge\frac{1}{L(n)}\int_{\chi_i(n)}^{1+X_i(0)/n}\frac{1}{s}ds\\
       &=\frac{\log(1+X_i(0)/n)}{L(n)}-\frac{\log(\chi_i(n))}{L(n)}\,.
   \end{align*}

   Since $\lim_{n\to\infty}\frac{\log(1+X_i(0)/n)}{L(n)}=0$, we have
   $$c\ge -\frac{\log(\chi_i(n))}{L(n)}$$
   for an updated constant $c$, which proves the claim.
\end{proof}

In particular, $\chi_i(n)$ converges to zero slower than any polynomial when\linebreak $\lim_{n\to\infty}L(n)/\log(n)=0$. The following example discusses that bound in a generic situation.

\begin{example}
    Let $F_i (k)\equiv F(k)=k(\log k)^\beta$ for $\beta\ge0$. For $\beta=0$, the lower bound $e^{-cL(n)}$ is constant since $\chi_i(n)$ does converge to a non-zero limit. For $\beta\in (0, 1)$, the bound converges to zero slower than any polynomial, whereas it is of order $n^{-c}$ for $\beta=1$. Note that $c$ is random and unbounded. Finally for $\beta>1$, the process reveals strong monopoly such that $\chi(n)$ converges to an edge of the simplex at rate $1/n$. In that specific case for $\beta\le1$, we can also derive an upper bound for $\chi_i(0)$, provided that agent $i$ is not the monopolist. Since the limit in Theorem \ref{thm: almostLinear2} is non-zero and $a_i(t)\sim(\log t)^{1-\beta}$, there is a positive constant such that
    \begin{align*}
        0<const.&\le (\log n)^{1-\beta}-(\log(X_i(n)))^{1-\beta}
    \end{align*}
    and consequently
    $$\log(X_i(n))\le\left( (\log n)^{1-\beta}-const.\right)^{\frac{1}{1-\beta}}\,.$$
    Defining $\epsilon(n)=\frac1n\left( n^{1-\beta}-const.\right)^{\frac{1}{1-\beta}}$ yields:
    \begin{align*}
        X_i(n)\le e^{(\log n)\epsilon(\log n)}\quad\Leftrightarrow\quad \frac{X_i(n)}{n}&\le e^{(\log n)\epsilon(\log n)-\log n}= e^{-(\log n)(1-\epsilon(\log n))}
    \end{align*}
    Note that $1-\epsilon(n)>0$ converges to zero at rate $1/n^{1-\beta}$, so that we finally get the following estimate:
    $$\chi_i(n)\le const. e^{-const.(\log n)^\beta}$$
    Thus, the bound in Corollary \ref{cor: almostLinear2} can be considered as sharp.
\end{example}

If the second part of (\ref{eq: squareSum}) is not fulfilled, i.e. $\sigma_i^2 =\infty$, then $\frac{S_i(k)-a(k)}{\sum_{l=1}^k\frac{1}{F_i(l)^2}}$ fulfills the Lindeberg condition. Hence, Theorem \ref{thm: Barbour} and its implications are wrong if we drop the condition $\sigma_i^2 <\infty$. As already described at the end of Section \ref{sec: non-monopoly}, \cite{Khanin} derives a central limit theorem for polynomial feedback functions with $\sigma_i^2 =\infty$. Moreover, \cite{Oliveira4,Oliveira3,Kearney} present another transition between functions satisfying this condition and those who do not. 

Another remarkable property is the following: The proof of part 3 of Corollary \ref{cor: almostLinear} reveals that  $\frac{X_i (n)}{X_j (n)}\to 0$ or $\infty$ for $n\to\infty$ for all $i\neq j$. This corresponds to a hierarchical structure of asymptotic market shares consistent with weak monopoly and the consistency property in Proposition \ref{prop: partialProcess}, such that within each subset of agents a weak monopolist has full relative market share. Such hierarchical structures are often observed at phase transition points, in our case the transition between strong monopoly and deterministic limit shares.


\section{A Law of Large Numbers for the dynamics}\label{sec:dynamics}

So far our investigations focused on the analysis of the long-time behavior of a generalized Pólya urn. This section examines the dynamics of the process in the limit for large initial market size $N$, based on the concept of stochastic approximation (see e.g. \cite{Arthur2, Pemantle,Ruszel, Borkar}), which traces back to \cite{Monro}. 
Note that $X(n)$ and $\chi(n)$ depend on $N$, thus we establish the notation $X^{(N)}(n)=\left(X_1^{(N)}(n),..., X_A^{(N)}(n)\right)=X(n)$ and $\chi^{(N)}(n)=\left(\chi_1^{(N)}(n),..., \chi_A^{(N)}(n)\right)=\chi(n)$ for this section and assume that $\chi^{(N)}(0)$ is equal for all $N$ (up to roundings).

\begin{theorem}
\label{thm: dynamic}
Define for $x\in\Delta_{A-1}$
\begin{equation}\label{eq:gdef}
G(k, x)=p(k, x)-x\quad\text{and}\quad G(x)=\lim_{k\to\infty}G(k, x),
\end{equation}
where we assume that $G(k, (\cdot))$ converges for $k\to\infty$ uniformly to a Lipschitz-continuous function $G$ on an open neighborhood $D\subset\Delta_{A-1}$ of the image of the solution $Z\colon (0, \infty)\to\Delta_{A-1}$ of the differential equation 
\begin{equation}
\label{eq: ODE}
\frac{d}{dt}Z(t)=\frac{G(Z(t))}{1+t}\quad\text{with}\quad Z(0)=\chi(0).
\end{equation}
Moreover, we define the following sequence of stochastic processes in $\Delta_{A-1}$:
\begin{equation*}
(Z^{(N)} )_N \coloneqq\left(Z^{(N)}(t)\colon t\ge0\right)_N\coloneqq\left(\chi^{(N)}(\lfloor Nt\rfloor)\colon t\ge0\right)_N
\end{equation*}
Then: $Z^{(N)}$ converges to $Z$ weakly on the Skorochod space $\D([0, \infty), \Delta_{A-1})$.
\end{theorem}

Similar ODE approximations of the generalized Pólya urn model have also been derived in \cite{Schreiber, Benaim2}, but they rather  focus on the embedded process $\Xi(t)$ from the exponential embedding and on the limit $t\to\infty$ instead of $N\to\infty$.

\begin{proof}
By construction, we have $\|Z^{(N)}(t)-Z^{(N)}(s)\|\le\frac{N|t-s|+1}{N}$ for all $t, s\ge0$, where $\|\cdot\| =\|\cdot\|_\infty$ denotes the supremum norm. This implies by \cite[Proposition VI.3.26]{Jacod} that the sequence $(Z^{(N)})_N$ is tight in $\D([0, \infty), \Delta_{A-1})$, with the additional property that all weak limits of converging subsequences are concentrated on the subspace of continuous functions. We now take any converging subsequence and show that the limit solves (\ref{eq: ODE}). As the solution of (\ref{eq: ODE}) is unique due to the assumed Lipschitz-continuity of $G$, this implies the claim. For simplicity of notation assume that the subsequence is $(Z^{(N)})_N$ itself. Then we can write the increments as
\begin{align*}
\chi^{(N)}(n+1)-\chi^{(N)}(n)&=\frac{X^{(N)}(n+1)}{N+n+1}-\chi^{(N)}(n)\\
&=\frac{(N+n)\chi^{(N)}(n)+X^{(N)}(n+1)-X^{(N)}(n)}{N+n+1}-\chi^{(N)}(n)\\
&=\frac{1}{N+n+1}\left(-\chi^{(N)}(n)+X^{(N)}(n+1)-X^{(N)}(n)\right)\\
&=\frac{1}{N+n+1}\left(G(N+n, \chi^{(N)}(n))+\xi^{(N)}(n)\right)
\end{align*}
with $\xi^{(N)}(n)\coloneqq X^{(N)}(n+1)-X^{(N)}(n)-G(N+n, \chi^{(N)}(n))-\chi^{(N)}(n)$. Note that $\xi^{(N)}(n)$ is $\mathcal{F}_{n+1}^{(N)}$-measurable, where $(\mathcal{F}_n^{(N)})_{n\ge0}$ is the filtration generated by the process $(\chi^{(N)}(n))_{n\ge0}$. Furthermore, 
\begin{equation*}
\E\left[\xi^{(N)}(m) \mid\mathcal{F}_n^{(N)}\right]=0\quad\text{for }m\ge n
\end{equation*}
since with \eqref{eq:gdef} $\E\left[ X^{(N)}(n+1)-X^{(N)}(n) \mid\mathcal{F}_n^{(N)}\right] =G(N+n, \chi^{(N)}(n))+\chi^{(N)}(n)$\ .
The $\xi^{(N)}(n)$ are also uncorrelated, as for $m>n$
\begin{align}
\label{eq: uncor}
\E\left[\xi_i^{(N)}(n)\xi_j^{(N)}(m)\right]=
\E\left[\xi_i^{(N)}(n)\,\E\left[\xi_j^{(N)}(m) \mid \mathcal{F}_{n+1}^{(N)}\right]\right]=0\quad\text{for all }i,j\in [A]\ .
\end{align}
Summing up the increments yields the standard Doob-Meyer decomposition
\begin{equation*}
\chi^{(N)}(n)=\chi^{(N)}(0)+H^{(N)}(n)+M^{(N)}(n)
\end{equation*}
with predictable and martingale part, respectively
\begin{equation}\label{hmdef}
H^{(N)}(n)\coloneqq\sum_{k=0}^{n-1}\frac{G(N+k, \chi^{(N)}(k))}{N+k+1}\quad\text{and}\quad M^{(N)}(n)\coloneqq\sum_{k=0}^{n-1}\frac{1}{N+k+1}\xi^{(N)}(k)\ .
\end{equation}
With uncorrelated and centered increments $(M^{(N)}(n))_{n\ge0}$ is a centered martingale with respect to the filtration $(\mathcal{F}_n^{(N)})_{n\ge0}$, thus Doob's inequality yields for any $\epsilon, t>0$:
\begin{align}
\label{eq: martingaleConv}
&\P\left(\exists s\le t : \|M^{(N)}(\lfloor Ns\rfloor)\|\ge\epsilon\right)\le\frac{A}{\epsilon^2}\E\left[\|M^{(N)}(\lfloor Nt\rfloor)\|^2\right]\\
&=\frac{A}{\epsilon^2}\sum_{k=0}^{\lfloor Nt\rfloor-1}\frac{1}{(N+k+1)^2}\E\left[\|\xi^{(N)}(k)\|^2\right]\le\frac{A}{\epsilon^2}\sum_{k=0}^{\lfloor Nt\rfloor-1}\frac{1}{(N+k+1)^2}\xrightarrow{t, N\to\infty}0
\end{align}
since $\|\xi^{(N)}(k)\| \leq 1$ almost surely by definition. 
Hence, the sequence $\left(M^{(N)}(Nt)\colon t\ge0\right)_N$ of stochastic processes converges to zero weakly on $\D([0, \infty), \R^A)$. 

Now we turn to the predictable part $H^{(N)}$. By the Skorochod representation theorem we can find a probability space such that the convergence of $(Z^{(N)})_N$ is almost sure. Then for fixed $\omega\in\Omega$ $(Z^{(N)})_N$ converges with respect to the Skorochod norm to a process $\hat Z$ on $\Delta_{A-1}$. As $\hat Z$ is continuous, the convergence is uniform on bounded time intervals. Denote $t_0\in(0, \infty]$ the stopping time, when $\hat Z$ first leaves $D$. Then for any $t<t_0$ and large enough $N=N(t)$ we have $Z^{(N)}(t)\in D$ and consequently
\begin{align*}
H^{(N)}(\lfloor Nt\rfloor)&=\sum_{k=0}^{\lfloor Nt\rfloor-1}\frac{G(N+k, \chi^{(N)}(k))}{N+k+1}=\sum_{k=0}^{\lfloor Nt\rfloor-1}\frac{1}{N}\cdot\frac{G(N+k, \chi^{(N)}(N\cdot\frac{k}{N}))}{1+\frac{k}{N}+\frac{1}{N}}\\
&\xrightarrow{N\to\infty}\int_0^t\frac{G(Z(u))}{1+u}du
\end{align*}
as the sequence $\left(u\mapsto\frac{G(N+k, \chi^{(N)}(Nu))}{1+u+\frac{1}{N}}\right)_N$ of functions converges uniformly to $u\mapsto\frac{G(Z(u))}{1+u}$ on bounded time intervals. Thus, we have for $t<t_0$ that $(Z^{(N)})_N$ converges weakly on $\D([0, \infty), \Delta_{A-1})$ to $\hat Z(t)=\chi(0)+\int_0^t\frac{G(Z(x))}{1+x}dx$ which fulfills (\ref{eq: ODE}) and by uniqueness of solutions we have $\hat Z=Z$ and $t_0 =\infty$. 
\end{proof}

\begin{figure}
  \centering
  \subfloat[][$F_1(k)=F_2(k)=k^2,\, F_3(k)=2k^2$\\$G(x)=\left(\frac{F_i(x_i)}{x_1^2+x_2^2+2x_3^2}-x_i\right)_{i=1, 2, 3},\, x\in\Delta_2$]{\includegraphics[width=0.5\linewidth]{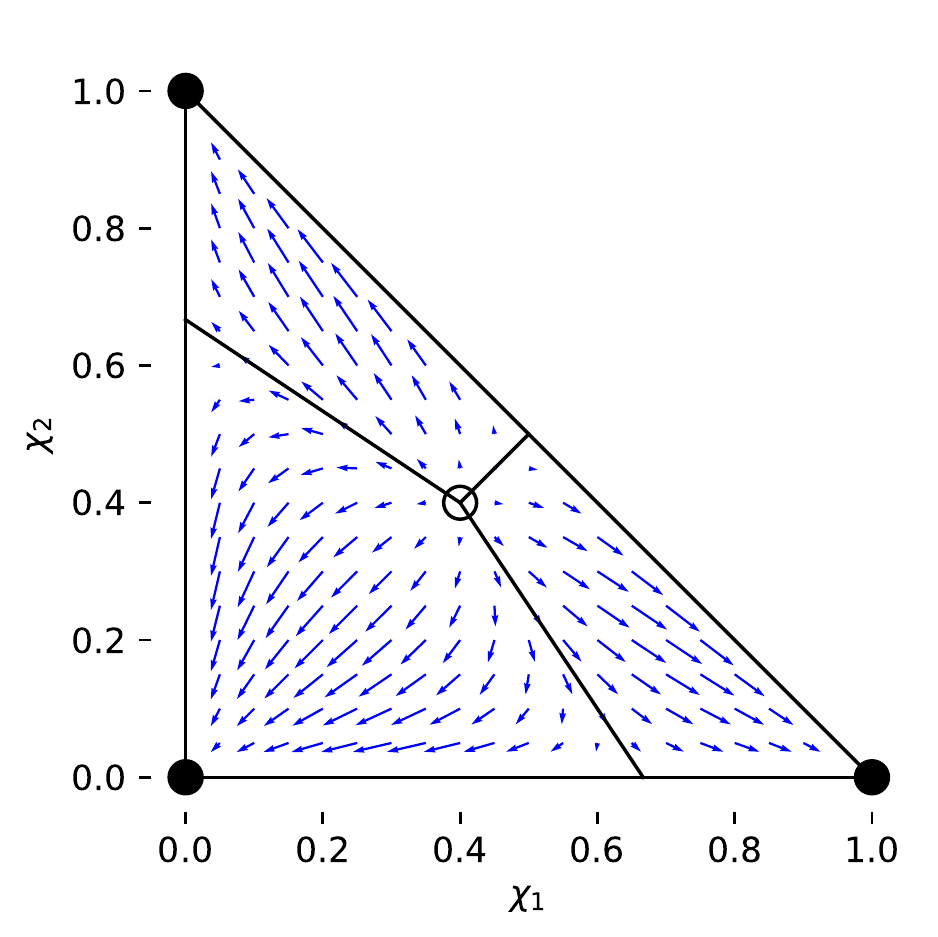}}
  \subfloat[][$F_1(k)=F_2(k)=F_3(k)=\sqrt{k}$\\$G(x)=\left(\frac{\sqrt{x_i}}{\sqrt{x_1}+\sqrt{x_2}+\sqrt{x_3}}-x_i\right)_{i=1, 2, 3},\, x\in\Delta_2$]{\includegraphics[width=0.5\linewidth]{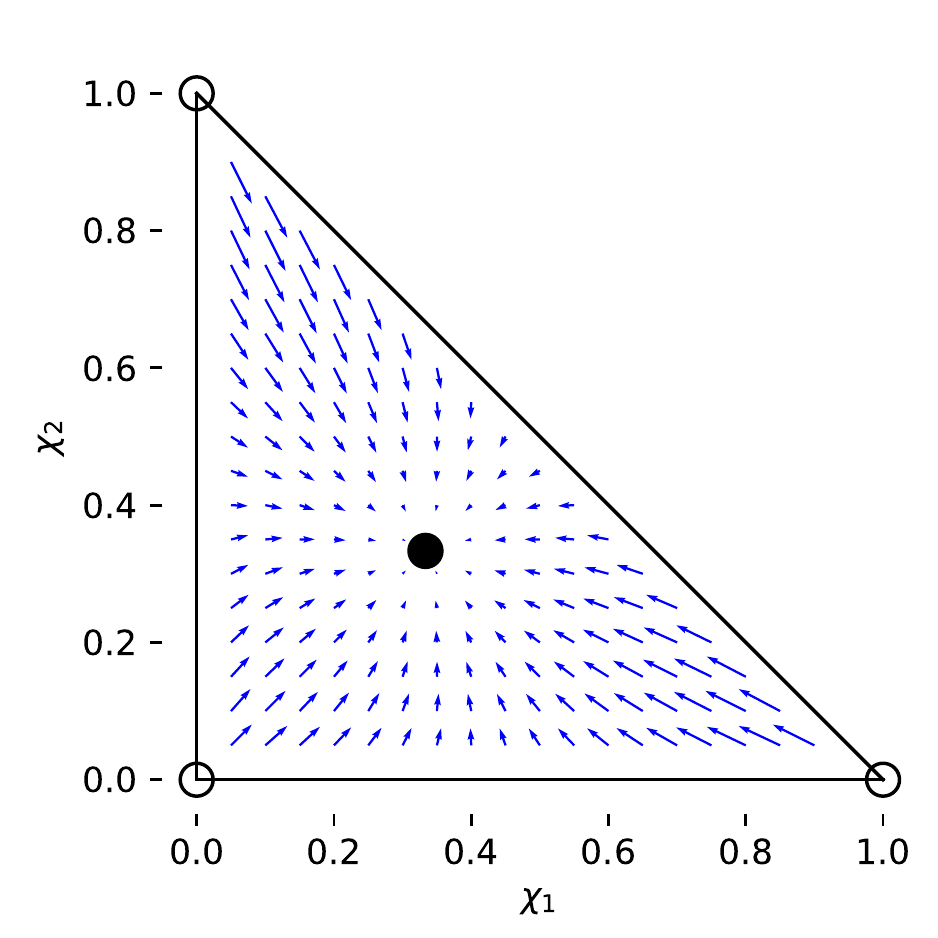}}\\
  \subfloat[][$F_1(k)=F_2(k)=k,\,F_3(k)=2k$\\$G(x)=\left(\frac{F_i(x_i)}{x_1+x_2+2x_3}-x_i\right)_{i=1, 2, 3},\, x\in\Delta_2$]{\includegraphics[width=0.5\linewidth]{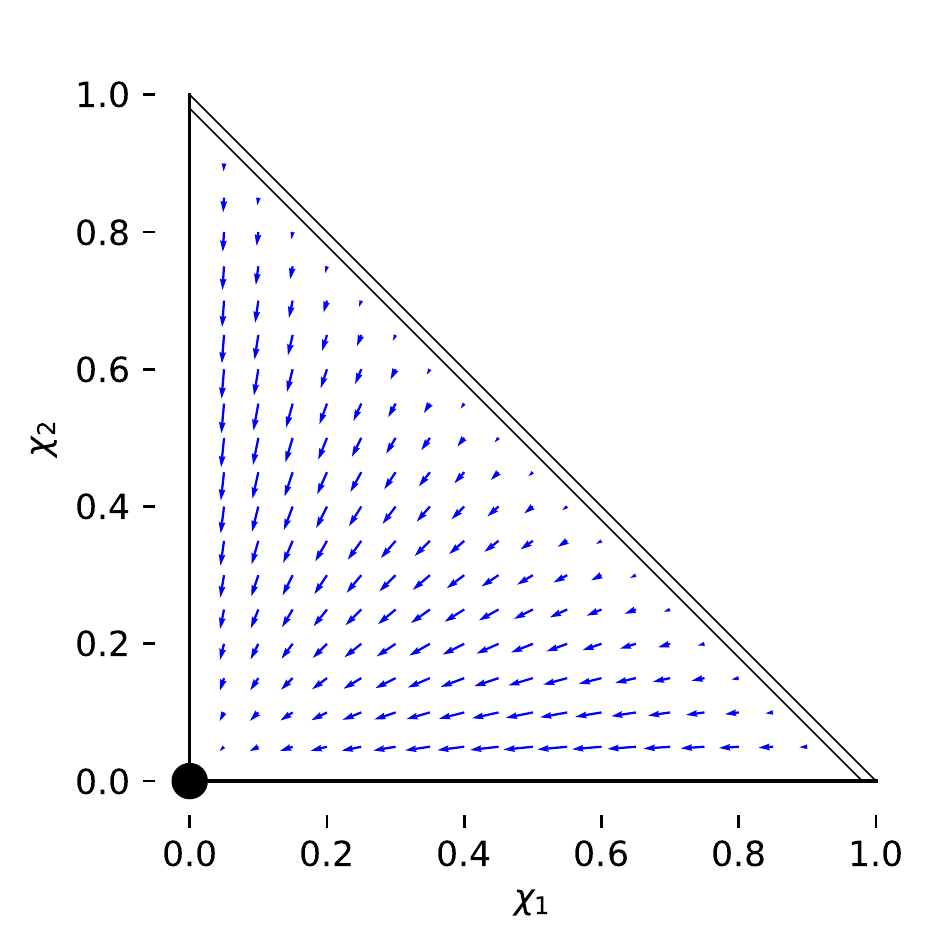}}
  \subfloat[][$F_1(k)=F_2(k)=2k, \,F_3(k)=k$\\$G(x)=\left(\frac{F_i(x_i)}{2x_1+2x_2+x_3}-x_i\right)_{i=1, 2, 3},\, x\in\Delta_2$]{\includegraphics[width=0.5\linewidth]{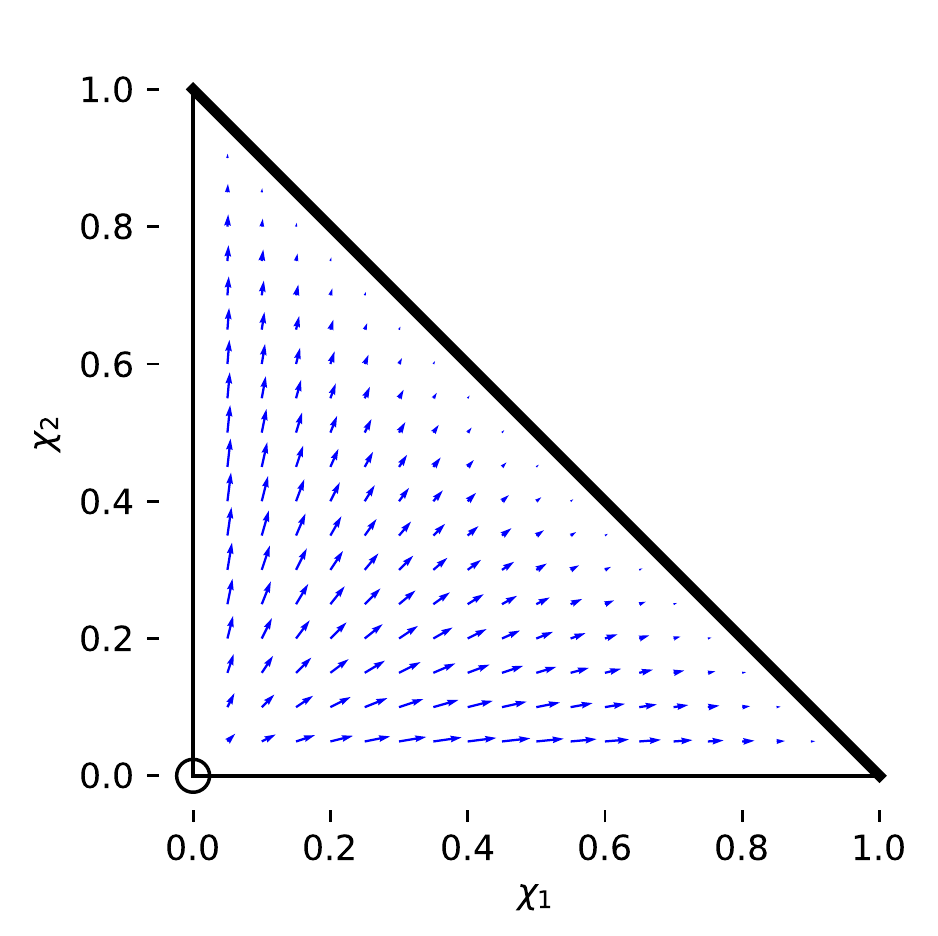}}
  \caption{The vector field $G$ for different feedback functions and $A=3$. Here $\bullet$ marks the stable and $\circ$ the unstable fixed points of the dynamics \eqref{eq: ODE}. In addition, Figure (a) shows the asymptotic attraction domains as derived in Example \ref{example: sMonPolynom}.}
  \label{fig: G}
\end{figure}

This means, that $(\chi^{(N)}(n))_{n\ge0}$ is asymptotically deterministic and driven by the vector-field $(G(x))_{x\in\Delta_{A-1}}$ modulo a time change. 
Let $Y\colon [0, \infty)\to\Delta_{A-1}$ be the solution of the time-homogeneous differential equation
\begin{equation}
\label{eq: ODEhom}
\frac{d}{dt}Y(t)=G(Y(t))\quad\text{with}\quad Y(0)=\chi(0)\ ,
\end{equation}
so that $Z(t)=Y(\log (1+t))$. 
Then for large $N$ the process $(\chi^{(N)}(n))_{n\ge0}$ is approximately given by $\left(Y\left(\log\left( 1+\frac{n}{N}\right)\right)\right)_{n\ge0}$. 
We can use this result e.g. to estimate the number of steps until the process reaches a given neighborhood of its long-time limit for large $N$.

\begin{corollary}
In the situation above let $D\subset\Delta_{A-1}$ be an open neighborhood of $\lim_{t\to\infty}Y(t)$ and define the following last entrance times:
\begin{align*}
t^*&\coloneqq\sup\{t\ge0\colon Y(t)\notin D\}\\
t_N&\coloneqq\sup\{n\ge0: \chi^{(N)}(n)\notin D\}
\end{align*}
Then we have
\begin{equation*}
\frac{t_N}{N}\xrightarrow{N\to\infty}e^{t^*}-1\quad\text{in probability}\ .
\end{equation*}
\end{corollary}

This follows directly from the Theorem \ref{thm: dynamic} via the continuous mapping theorem.

Another interesting consequence of Theorem \ref{thm: dynamic} is the following. In the monopoly case described in Section \ref{sec: monopoly}, we may start our process in an unstable fixed point $\chi (0)$ of the vector field $G$. Although we know that the process exhibits strong monopoly, we have $Z(t)\equiv\chi(0)$ for all times $t\ge0$ in Theorem \ref{thm: dynamic}. This implies that a linear scaling of time is not sufficient to capture the escape from an unstable equilibrium.

\begin{corollary}
\label{cor: escape}
In the situation of Theorem \ref{thm: dynamic}, let $G(\chi(0))=0$. For $\epsilon>0$ define the escape time 
\begin{equation*}
    t_N(\epsilon)\coloneqq\inf\{n\ge0\colon \|\chi(n)-\chi(0)\|\ge\epsilon\}.
\end{equation*}
with the convention $\inf\emptyset=\infty$. Then
\begin{equation*}
    \frac{t_N(\epsilon)}{N}\xrightarrow{N\to\infty}\infty\quad\text{in probability}\ .
\end{equation*}
\end{corollary}
\begin{proof}
This follows from Theorem \ref{thm: dynamic} via
\begin{align*}
    \P\left(\frac{t_N(\epsilon)}{N}>t\right)&=\P\left(\|Z^{(N)}(s)-\chi(0)\|<\epsilon\text{ for all }s\le t\right)\\
    &=\P\left(\sup_{0\le s\le t}\|Z^{(N)}(s)-Z(s)\|<\epsilon\right)\xrightarrow{N\to\infty}1
\end{align*}
for all $t>0$ since $Z(s)\equiv \chi (0)$.
\end{proof}

Simulations for $F_i(k)=k^\beta,\,\beta>1$ indicate that the escape from an unstable equilibrium is faster the larger $\beta$ is. Recall that for superexponential feedback functions (see Corollary \ref{cor: criticalShare}) the winner of the first step wins in all further steps with high probability if $N$ is large. Hence, it only takes $O(N)$ time to escape from an unstable equilibrium in this case. Nevertheless, this does not pose a contradiction to  Corollary \ref{cor: escape} since the convergence of $G(k, (\cdot))$ to $G$ is not uniform in an unstable equilibrium. Thus, Theorem \ref{thm: dynamic} is not applicable and the assumption of uniform convergence can not be removed. 

Figure \ref{fig: G} shows the dynamics of the process $(\chi(n))_n$ in various generic situations. The fixed points of the dynamics, i.e. the zeros of the vector-field $G$, are the long-time market-shares of our generalized Pólya-urn, but only the stable fixed points are attained with positive probability. Figure (a), (b) and (c) comply with the properties found in the sections before, i.e. monopoly in the superlinear case and stable, non-zero market-shares in the sublinear case. Figure (d) underlines that the set of stable fixed points is not necessarily discrete. Note that when $F_i(k)=kL(k)$ for all agents $i\in[A]$ and a slowly varying function $L$, then the field $G$ is constantly zero, such that all points are fixed points. In particular, this holds for the original Pólya urn, where $L$ is a constant function. If $L$ diverges, then the process exhibits weak monopoly resp. deterministic limits  for finite $N$ (see Section \ref{sec: specialCase}), which is again not captured by Theorem \ref{thm: dynamic} as it takes more than $O(N)$ steps to reach the long-time limit.

Moreover, the assumptions of Theorem \ref{thm: dynamic} are not fulfilled for exponential feedback, since $G$ is not continuous. Nevertheless, the dynamics in the limit $N\to\infty$ are already described by Corollary \ref{cor:44}, which states that all steps are won by the same agents as long as $\chi(0)$ is not on the boundary between the attraction domains. Note that this is consistent with Theorem \ref{thm: dynamic}, i.e. (\ref{eq: ODE}) still holds.

Since $F_i$ only depends on $X_i$ and not $X_j$, $j\neq i$ there are no limit cycles and the dynamics tends to a fixed point, as opposed to models discussed in \cite{Costa}.


\section{A  Functional Central Limit Theorem for the dynamics}\label{sec: clt}

In Section \ref{sec:dynamics} we derived a functional law of large numbers for the process of market shares for large initial values, which states that the time-scaled process $Z^{(N)}$ can be well approximated by a deterministic process $Z$ for large $N$. In order to gain an understanding of the fluctuations around this limit, we prove a corresponding functional central limit theorem in this section. Let us first state our main result. We use the notations introduced in Section \ref{sec:dynamics} and establish furthermore the notation
\begin{equation}\label{eq: p}
   p(x)=(p_i(x))_{i\in[A]}=\lim_{k\to\infty} p(k, x)\ ,
\end{equation}
for all $x\in\Delta_{A-1}$. Note that the existence of $p$ is equivalent to the existence of $G$. Denote by 
\begin{equation}\label{eq:tspace}
T\Delta_{A-1}\coloneqq\left\{(x_1,\ldots,x_A)\in\R^A\colon \sum_{i=1}^Ax_i=0\right\}    
\end{equation}
the tangent space of $\Delta_{A-1}$.

\begin{theorem}\label{cor: fclt}
Suppose that 
\begin{equation}\label{eq: condThmH}
    \lim_{k\to\infty}\sqrt{k}\sup_{x\in\Delta_{A-1}}\|G(k, x)-G(x)\|=0.
\end{equation}
Moreover, let $G$ be continuously differentiable on $\Delta_{A-1}^o$. Then we have
\begin{equation*}
    \sqrt{N}\left(Z^{(N)}(t)-Z(t)\right)_{t\ge0}\xrightarrow{N\to\infty}(\tilde Z(t))_{t\ge0}\quad\text{weakly on }\mathbb{D}([0, \infty), T\Delta_{A-1}),
\end{equation*}
where $\tilde Z$ is the solution of the system of stochastic differential equations
\begin{equation}\label{eq: CLT}
    d\tilde Z_i(t)=\frac{DG_i(Z(t))}{1+t}\bullet\tilde Z(t)dt+\sum_{j\ne i}\frac{\sqrt{ p_i(Z(t)) p_j(Z(t))}}{1+t}dB_{i, j}(t),\quad i\in[A].
\end{equation}
Here, $DG_i$ denotes the differential operator of $G_i$ and $B_{i,j}$ is a standard Brownian motion, which is independent of $B_{k, l}$ if $\{i, j\}\ne\{k, l\}$ and $B_{j,i}=-B_{i, j}$ for $i\ne j$.
\end{theorem}

The differential operator $DG_i(z)\colon T\Delta_A\to\R,\,\tilde z\mapsto T\Delta_{A-1}(z)\bullet\tilde z$ for $z\in\Delta_{A-1}^o$ is the product with the gradient $\nabla G_i(z)$, when $G$ is defined on an open neighbourhood of $T\Delta_{A-1}$ in $\R^A$. \cite{Borkar} presents a central limit theorem  in a general stochastic approximation setting. Further functional central limit theorems in the context of Pólya urns have recently been studied in \cite{Borovkov} and \cite{Dean}.

For the proof, we use again the method of stochastic approximation. In the Doob decomposition (\ref{hmdef}), we prove separately a limit theorem for the martingale part $M^{(N)}$ in Subsection \ref{subsec: martingale} and for the predictable part $H^{(N)}$ in Subsection \ref{subsec: H}, which directly imply Theorem \ref{cor: fclt} by summing up both. Note that Theorem \ref{thm: CLT} for the martingale does not use the rather restrictive condition (\ref{eq: condThmH}). Within these Subsections, we discuss in detail the properties and interpretation of the diffusion part and the drift part of (\ref{eq: CLT}).

Figure \ref{fig: Z^N} shows the process $Z^{(N)}-Z$ for large $N$. We can observe that $Z^{(N)}(t)-Z(t)$ is close to zero for large $t$. Indeed, this complies with formula (\ref{eq: RODE}). 

\begin{proposition}\label{prop: tildeZ}
  In the situation of Theorem \ref{cor: fclt}, assume that $Z(\infty)\coloneqq\lim_{t\to\infty}Z(t)$ exists and that $DG(Z(\infty))$ is a negative definite operator. Then
  $$\tilde Z(t)\xrightarrow{t\to\infty}0\quad\text{in }L^2\text{ and almost surely}.$$
\end{proposition}

\begin{proof}
    As explained in Subsection \ref{subsec: martingale}, the generator of $\tilde Z$ is given by
    \begin{align*}
     L_tf(x)&=\sum_{i=1}^A\frac{DG_i(Z(t))\cdot x}{1+t}\cdot\frac{\partial}{\partial x_i}f(x)+\sum_{i=1}^A\frac{p_i(Z(t))(1-p_i(Z(t)))}{2(1+t)^2}\frac{\partial^2}{\partial x_i^2}f(x)\\
     &+\sum_{i=1\atop i\ne j}^A\frac{p_i(Z(t))p_j(Z(t))}{(1+t)^2}\frac{\partial^2}{\partial x_i\partial x_j}f(x)
    \end{align*}
    for $x=(x_1,\ldots,x_A)\in T\Delta_{A-1}$. Thus, for $f(x)=x_1^2+\ldots+x_A^2$ we have
    \begin{align*}
       L_tf(x)&=\sum_{i=1}^A\frac{DG_i(Z(t))\cdot x}{1+t}\cdot2x_i+\sum_{i=1}^A\frac{p_i(Z(t))(1-p_i(Z(t)))}{(1+t)^2}\\
        &=\frac{2}{1+t}\langle DG(Z(t))x, x\rangle +\frac{b(t)}{(1+t)^2}
    \end{align*}
    for a bounded function $b(t)$. Since $t\mapsto DG(Z(t))$ is continuous and $DG(Z(\infty))$ is negative definite, $DG(Z(t))$ is also negative definite for $t\ge t_0$, when $t_0>0$ is large enough. Thus, there is $\lambda>0$ such that
    $$\langle DG(Z(t))x, x\rangle\le-\lambda\|x\|^2$$
    for all $x\in\R^A$ and $t\ge t_0$. In summary, we get
    $$L_tf(x)\le-\frac{2\lambda}{1+t}\|x\|^2+\frac{b(t)}{(1+t)^2}$$
     for $t\ge t_0$. Now, applying Dynkin's formula yields
    \begin{align*}
    \frac{d}{dt}\E\|\tilde Z(t)\|^2=\E L_tf(\tilde Z(t))\le-\frac{2\lambda}{1+t}\E\|\tilde Z(t)\|^2+\frac{b(t)}{(1+t)^2}.
    \end{align*}
    for $t\ge t_0$. Finally, the claim follows from Grönwall's inequality:
    \begin{align*}
     \E\|\tilde Z(t)\|^2\le\left(\int_{t_0}^t\frac{b(s)}{(1+s)^2}ds+\E\|\tilde Z(t_0)\|^2\right)\exp\left(\int_{t_0}^t-\frac{2\lambda}{1+s}ds\right)\xrightarrow{t\to\infty}0
    \end{align*}
    For the almost sure convergence we fix a realisation $\omega\in\Omega$, such that $m\coloneqq\lim_{t\to\infty}M(t)(\omega)$ exists. Then we get from (\ref{eq: RODE}) and the Cauchy-Schwarz inequality that
    \begin{align*}
        &\frac{d}{dt}\|H(t)+m\|^2\\
        &=\frac{2}{1+t}\big(\langle DG(Z(t))(H(t)+m), H(t)+m)\rangle+\langle H(t)+m, DG(Z(t))(M(t)-m)\rangle\big)\\
        &\le \frac{2}{1+t}\big(-\lambda\|H(t)+m\|^2+\|H(t)+m)\|\cdot\|DG(Z(t))(M(t)-m)\|\big)
    \end{align*}
    for $t\ge t_0$. Hence
    $$ \frac{d}{dt}\|H(t)+m\|^2>0\quad\Longrightarrow\quad \frac{\|DG(Z(t))(M(t)-m)\|}{\lambda}>\|H(t)+m\|,$$
    which implies $\|H(t)+m\|\xrightarrow{t\to\infty}0$ as $\|DG(Z(t))(M(t)-m)\|\xrightarrow{t\to\infty}0$.
\end{proof}

In generic examples one can show that $DG(Z(\infty))$ is indeed negative definite, but it is also possible to find a counterexample.

\begin{figure}
  \centering
    \subfloat[][$F_1(k)=F_2(k)=F_3(k)=\sqrt{k}$,\\$\chi(0)=\left(\frac{1}{10}, \frac{1}{10}, \frac{4}{5}\right)$]{\includegraphics[width=0.5\linewidth]{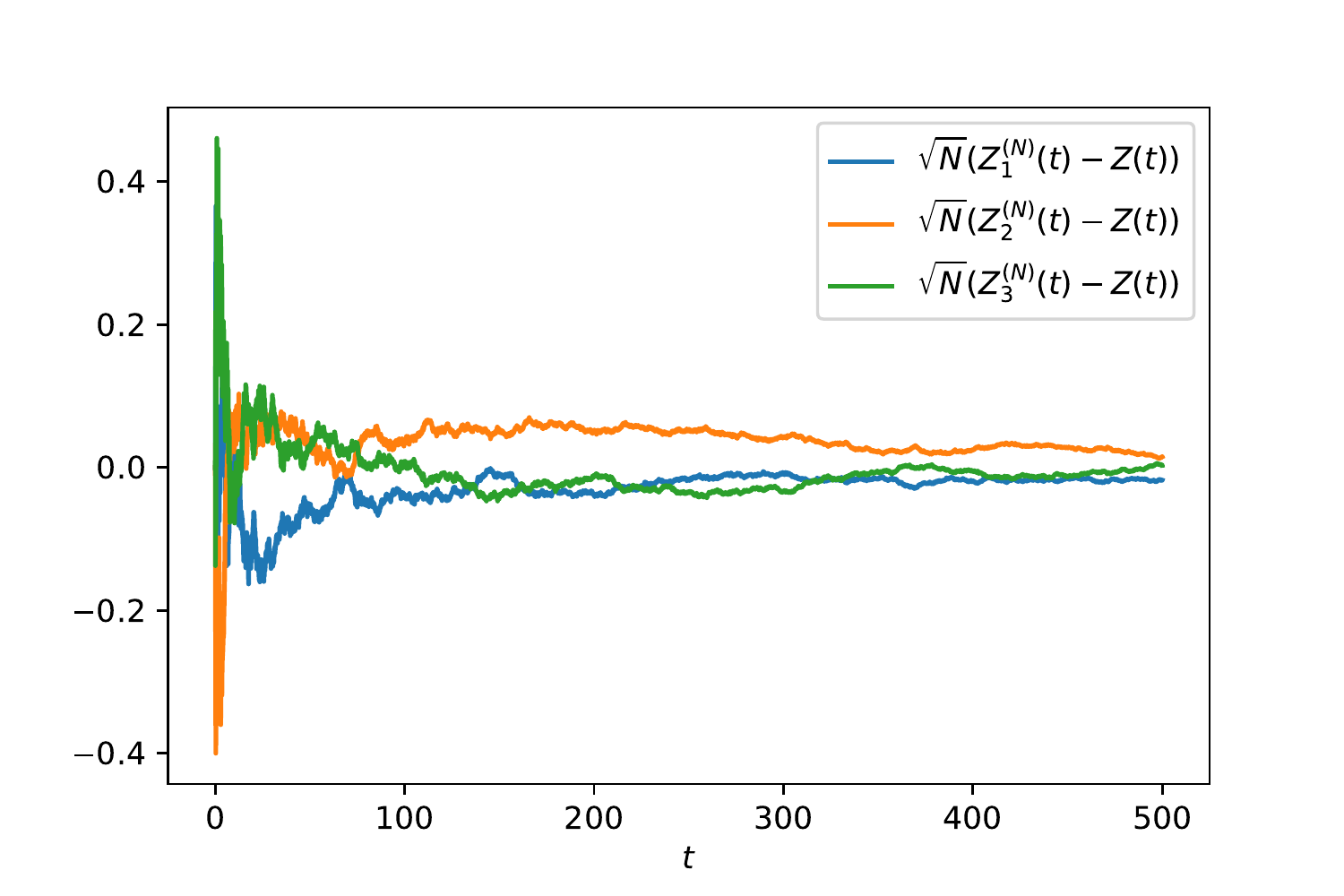}}
  \subfloat[][$F_1(k)=F_2(k)=F_3(k)=k^2$,\\$\chi(0)=\left(\frac{3}{10}, \frac{3}{10}, \frac{2}{5}\right)$]{\includegraphics[width=0.5\linewidth]{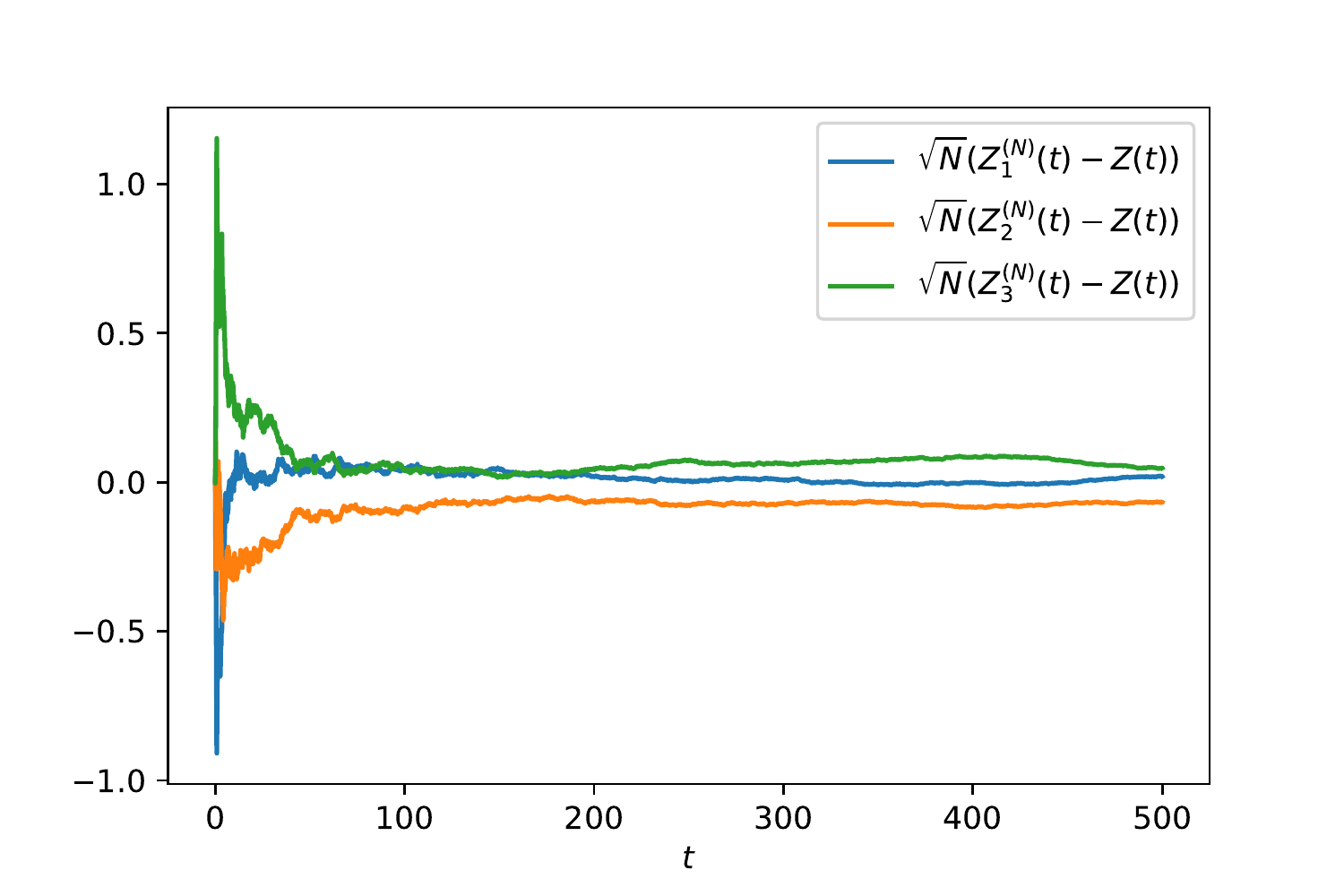}}
  \caption{The processes $\sqrt{N}\left(Z^{(N)}(t)-Z(t)\right)$ for $A=3$ and $N=10.000$.}
  \label{fig: Z^N}
\end{figure}

\begin{example}\label{ex: negDefinit}
Let $F_i(k)=\alpha_ik^\beta$ for $\alpha_i>0,\,\beta>0$, such that
$$G_i(x)=\frac{\alpha_ix_i^\beta}{\alpha_1x_1^\beta+\ldots\alpha_Ax_A^\beta}-x_i\quad\text{for all }x\in\Delta_{A-1}\ .$$
Since there is an obvious extension of $G$ to $\R^A$, the operator $DG(x)$ is negative definite if and only if the well-defined differential matrix $\left(\frac{\partial}{\partial x_j}G_i(x)\right)_{i,j=1,\ldots,A}$ is negative definite.
\begin{enumerate}
    \item  Consider the monopoly case $\beta>1$. Moreover, let $\chi(0)$ be in the attraction domain of agent $i$, i.e. $Z(t)\xrightarrow{t\to\infty}e^{(i)}$. A simple computation shows $\nabla G_j(e^{(i)})=(-\delta_{l, j})_{l=1,\ldots,A}$ for all $j\in[A]$, where $\delta_{i, j}$ denotes the Kronecker delta. Hence, $DG(e^{(i)})$ is negative definite. 
   \item In the monopoly case $\beta>1$ assume that $\chi(0)$ is the unique unstable fixpoint of the vector field $G$. Then $Z(\infty)=\chi(0)$ and $DG(Z(\infty))$ is positive definite. Thus, $\E\|\tilde Z(t)\|^2\xrightarrow{t\to\infty}\infty$ follows by similar argumentation.
   \item For $\beta=1$, we have $H(t)\equiv0$ since $G(x)\equiv0$. In this case $\tilde Z(t)$ does not converge to zero for $t\to\infty$.  This is due to the fact that for $\beta=1$ and large (but finite) $N$ the time-limit $\lim_{n\to\infty}\chi^{(N)}(n)$ is close to $\chi(0)$, but still random. For $\beta\ne1$, the long-time limit can be predicted precisely for large $N$ (at least with high probability).
    \item Now, let $\beta<1$. For simplicity, assume $\alpha_i=1$ for all $i\in[A]$, but a similar argument is possible in a non-symmetric situation. Then $Z(\infty)\coloneqq\lim_{t\to\infty}Z(t)=\left(\frac{1}{A}\right)_{i=1,\ldots,A}$. It can be shown that $\nabla G_i(Z(\infty))=(c\delta_{i, j}+d(1-\delta_{i, j}))_{j=1,\ldots,A}$ for some $c<d<0$, i.e. $DG(Z(\infty))$ is negative definite.
\end{enumerate}
\end{example}

Note that the time-change factor $\frac{1}{1+t}$ in (\ref{eq: CLT}) does not change the long-time limit of the dynamics, but slows down the rate of convergence. The Grönwall estimate in the proof of Proposition \ref{prop: tildeZ} implies that $\tilde Z(t)$ converges to zero at least at rate $t^{-2\lambda}$. For the classical P\`olya urn we have $\lambda=0$, such that there is no convergence to zero.

As we can see, the first steps of our process are of particular interest. In order to put focus on this, Appendix \ref{appendix: beta} examines the limiting behaviour of $(\chi(\lfloor N^\beta t\rfloor))_{t\ge 0}$ for $N\to\infty$ and non-linear time scale $\beta\in(0, 1)$.

\subsection{Convergence of the martingale part}\label{subsec: martingale}

This subsection examines the martingale $M^{(N)}=(M_1^{(N)},\ldots, M_A^{(N)})$ as defined in (\ref{hmdef}). We have already seen in Section \ref{sec:dynamics} that $M^{(N)}$ vanishes for $N\to\infty$. Under appropriate scaling, we can yield the following central limit theorem, which accounts for the diffusion part of (\ref{eq: CLT}). For simplicity we will at first only consider one fixed agent (without loss of generality agent 1) while keeping $A\ge2$ general.

\begin{theorem}
\label{thm: CLT}
We assume that the convergence (\ref{eq: p}) is uniform on an open neighborhood of the image of $Z$ and that $p$ is a Lipschitz continuous function on this neighbourhood. Moreover, denote by $(M_1(t))_{t\ge0}$ a time-inhomogeneous Markov process with generator
\begin{align*}
    L_sf\coloneqq\frac{f''}{2(1+s)^2}p_1(Z(s))(1-p_1(Z(s))),\quad s\ge 0
\end{align*}
and $M_1(0)=0$. Then
\begin{equation*}
    \sqrt{N}\left(M_1^{(N)}(\lfloor Nt\rfloor)\right)_{t\ge0}\xrightarrow{N\to\infty}(M_1(t))_{t\ge0}\quad\text{weakly on }\mathbb{D}([0, \infty), \mathbb{R})\ .
\end{equation*}
\end{theorem}
Alternatively, the inhomogeneous Markov-process $M_1$ is characterized as the solution of the stochastic differential equation
\begin{equation*}
    dM_1(t)=\frac{\sqrt{p_1(Z(t))(1-p_1(Z(t)))}}{1+t}dB(t),\quad M_1(0)=0,
\end{equation*}
where $B$ denotes a standard Brownian motion. Thus, $M_1$ is a time-changed Brownian motion. To be more precise, $M_1(t)=B(\langle M\rangle_t)$, where
\begin{equation*}
    t\mapsto\langle M_1\rangle (t)\coloneqq\int_0^t\frac{p_1(Z(s))(1-p_1(Z(s)))}{(1+s)^2}dt\le \int_0^t\frac{1}{4(1+s)^2}ds<\frac{1}{4}
\end{equation*}
is the quadratic variation process of $M_1$. Note that $\langle M_1\rangle (t)$ is deterministic and monotone increasing in $t$, and thus $M_1(t)$ converges almost surely for $t\to\infty$ and the limit has a centered Gaussian distribution with variance $\lim_{t\to\infty}\langle M_1\rangle (t)$.

For the proof of Theorem \ref{thm: CLT}, we first show tightness of the sequence \linebreak $\left(\sqrt{N}M_1^{(N)}(\lfloor Nt\rfloor)\colon t\ge0\right)_N$ on $\mathbb{D}([0, \infty),\R )$ and then prove that the limit of any converging subsequence is a Markov-process with generator $(L_s)_{s>0}$. For later use in Appendix \ref{appendix: beta}, we keep the tightness result a bit more general than necessary.

\begin{lemma}
\label{lemma: CLT1}
The sequence of martingales $\left(N^{1-\frac{\beta}{2}}M_1^{(N)}(\lfloor N^\beta t\rfloor)\colon t\ge0\right)_N$ is tight for all $\beta\in(0, 1]$.
\end{lemma}

\begin{proof}
According to a version of the Aldous criterion in \cite[Lemma 3.11]{Whitt}, the following two properties are sufficient for the tightness.

1. Stochastic Boundedness: For $C, T>0$ we have by Doob's inequality and (\ref{eq: uncor})
\begin{align*}
   \P&\left(\sup_{0<t\le T} N^{1-\frac{\beta}{2}}\big|M_1^{(N)}(\lfloor N^\beta t\rfloor)\big|>C\right)\le\frac{N^{2-\beta}}{C^2}\E\left(M_1^{(N)}(\lfloor N^\beta T\rfloor)^2\right)\\
   &=\frac{N^{2-\beta}}{C^2}\sum_{k=0}^{\lfloor N^\beta T\rfloor-1}\frac{1}{(N+k+1)^2}\E\left(\xi_1^{(N)}(k)^2\right)\le\frac{N^{2-\beta}}{C^2}\sum_{k=0}^{\lfloor N^\beta T\rfloor-1}\frac{1}{(N+k+1)^2}\\
   &\le\frac{N^{2-\beta}}{C^2}\int_N^{\lfloor N^\beta T\rfloor+N}\frac{1}{s^2}ds=\frac{N^{2-\beta}}{C^2}\left(\frac{1}{N}-\frac{1}{\lfloor N^\beta T\rfloor+N}\right)\le \frac{N^{2-\beta}}{C^2}\cdot\frac{\lfloor N^\beta T\rfloor}{N^2}\\
   &\le const.(T)/C^2 \xrightarrow{C\to\infty}0
\end{align*}
uniformly in $N$.

2. Similarly, we get for $0<t\le T$ and $0<u\le\delta$:

\begin{align*}
    &\E\left[\left(N^{1-\frac{\beta}{2}}M_1^{(N)}(\lfloor N^\beta (t+u)\rfloor)-N^{1-\frac{\beta}{2}}M_1^{(N)}(\lfloor N^\beta t\rfloor)\right)^2\big|\mathcal{F}_{\lfloor N^\beta t\rfloor}^{(N)}\right]\\
    &\le N^{2-\beta}\sum_{k=\lfloor N^\beta t\rfloor}^{\lfloor N^\beta (t+u)\rfloor-1}\frac{1}{(N+k+1)^2}\E\left[\xi_1^{(N)}(k)^2\big| \mathcal{F}_{\lfloor N^\beta t\rfloor}^{(N)}\right]\le N^{2-\beta}\sum_{k=\lfloor N^\beta t\rfloor}^{\lfloor N^\beta (t+\delta)\rfloor-1}\frac{1}{(N+k+1)^2}\\
    &\le N^{2-\beta}\int_{\lfloor N^\beta t\rfloor+N}^{\lfloor N^\beta(t+\delta)\rfloor+N}\frac{1}{s^2}ds=N^{2-\beta}\left(\frac{1}{\lfloor N^\beta t\rfloor+N}-\frac{1}{\lfloor N^\beta (t+\delta)\rfloor+N}\right)\\
    &\le N^{2-\beta}\cdot\frac{\lfloor N^\beta (t+\delta)\rfloor-\lfloor N^\beta t\rfloor}{N^2}\le const.(\delta)\xrightarrow{\delta\to0}0
\end{align*}
uniformly in $N$.
\end{proof}

By the definition of tightness and Theorem \ref{thm: dynamic}, we also get tightness of the joint sequence $(Z^{(N)}, N^{1-\frac{\beta}{2}}M_1^{(N)}(\lfloor N^\beta (\cdot)\rfloor))_N$. Before we turn to the proof of Theorem \ref{thm: CLT}, we add another helpful lemma.

\begin{lemma}
\label{lemma: CLT2}
With $p$ as defined in (\ref{eq: p}) and $\beta\in(0, 1]$, we have for all smooth test-functions $f\colon\R\to\R$ with compact support
\begin{align*}
    &\E\left[f\left(N^{1-\frac{\beta}{2}}M_1^{(N)}(k+1)\right)-f\left(N^{1-\frac{\beta}{2}}M_1^{(N)}(k)\right)\big|\mathcal{F}_k^{(N)}\right]\\
    &=\frac{N^{2-\beta}}{2(N+k+1)^2}f''\left(N^{1-\frac{\beta}{2}}M_1^{(N)}(k)\right)p_1(N+k, \chi^{(N)}(k))\left(1-p_1(N+k, \chi^{(N)}(k))\right)+o\left(N^{-\beta}\right)
\end{align*}
as $N\to\infty$.
\end{lemma}

\begin{proof}
Taylor-expansion of $f$  with Lagrange's remainder yields:
\begin{align*}
  &E\left[f\left(N^{1-\frac{\beta}{2}}M_1^{(N)}(k+1)\right)-f\left(N^{1-\frac{\beta}{2}}M_1^{(N)}(k)\right)\big|\mathcal{F}_k^{(N)}\right]\\
  &=N^{1-\frac{\beta}{2}}f'\left(N^{1-\frac{\beta}{2}}M_1^{(N)}(k)\right)\E\left[M_1^{(N)}(k+1)-M_1^{(N)}(k)\big|\mathcal{F}_k^{(N)}\right]\\
  &\quad+\frac{N^{2-\beta}}{2}\E\left[f''\left(m^{(N)}(k)\right)\left(M_1^{(N)}(k+1)-M_1^{(N)}(k)\right)^2\big|\mathcal{F}_k^{(N)}\right]\\
  &=\frac{N^{2-\beta}}{2(N+k+1)^2}\biggl[\E\left[f''\left(N^{1-\frac{\beta}{2}}M_1^{(N)}(k)\right)\xi_1^{(N)}(k)^2\big|\mathcal{F}_k^{(N)}\right]+o(1)\biggr]\\
  &=\frac{N^{2-\beta}}{2(N+k+1)^2}f''\left(N^{1-\frac{\beta}{2}}M_1^{(N)}(k)\right)\Big(\left(1-p_1(N+k, \chi^{(N)}(k))\right)^2p_1(N+k, \chi^{(N)}(k))\\
  &\quad+p_1(N+k, \chi^{(N)}(k))^2\left(1-p_1(N+k, \chi^{(N)}(k))\right)\Big)+o\left(N^{-\beta}\right)\\
  &=\frac{N^{2-\beta}}{2(N+k+1)^2}f''\left(N^{1-\frac{\beta}{2}}M_1^{(N)}(k)\right)p_1(N+k, \chi^{(N)}(k))\left(1-p_1(N+k, \chi^{(N)}(k))\right)+o\left(N^{-\beta}\right)
\end{align*}
Here, $m^{(N)}(k)$ denotes a (random) intermediate value between $N^{1-\frac{\beta}{2}}M_1^{(N)}(k)$ and \linebreak $N^{1-\frac{\beta}{2}}M_1^{(N)}(k+1)$. Note that $m^{(N)}(k)-N^{1-\frac{\beta}{2}}M_1^{(N)}(k)\xrightarrow{N\to\infty}0$ at rate $N^{-\frac{\beta}{2}}$.
\end{proof}

Now we are well prepared for the proof of Theorem \ref{thm: CLT}.

\begin{proof}
We show that for any limit $(Z, M_1)$ of a convergent subsequence of\linebreak $(Z^{(N)}, \sqrt{N}M_1^{(N)}(\lfloor N(\cdot)\rfloor))_N$, $M_1$ is a Markov process with generator $(L_s)_{s>0}$. For simplicity of notation, assume that the sequence is convergent itself.

Take a smooth test-function $f\colon\R\to\R$ with compact support. Then for each $N$ 
\begin{equation}
\label{eq: martingaleSeq}
    f\left(\sqrt{N}M_1^{(N)}(\lfloor Nt\rfloor)\right)-f(0)-\sum_{k=0}^{\lfloor Nt\rfloor-1}\E\left[f\left(\sqrt{N}M_1^{(N)}(k+1)\right)-f\left(\sqrt{N}M_1^{(N)}(k)\right)\big|\mathcal{F}_k^{(N)}\right],\quad
\end{equation}
is a martingale in continuous time $t\geq 0$ as $(Z^{(N)}, M_1^{(N)})$ is a discrete-time Markov process. The continuous mapping theorem implies that $f\left(\sqrt{N}M_1^{(N)}(\lfloor N(\cdot)\rfloor)\right)$ converges to $f(M_1)$ in $\mathbb{D}((0, \infty), \mathbb{R})$. Due to Lemma \ref{lemma: CLT2}, the sum converges as follows:
\begin{align*}
    &\sum_{k=0}^{\lfloor Nt\rfloor-1}\E\left[f\left(\sqrt{N}M_1^{(N)}(k+1)\right)-f\left(\sqrt{N}M_1^{(N)}(k)\right)\big|\mathcal{F}_k^{(N)}\right]\\
    &=\sum_{k=0}^{\lfloor Nt\rfloor-1}\biggl[\frac{N}{2(N+k+1)^2}f''\left(\sqrt{N}M_1^{(N)}(k)\right)p_1(N+k, \chi^{(N)}(k))\left(1-p_1(N+k, \chi^{(N)}(k))\right)\\ &\qquad\qquad+o(1/N)\biggr]\\
    &=\sum_{k=0}^{\lfloor Nt\rfloor-1}\!\!\frac{1}{2N(1+\frac{k}{N}+\frac{1}{N})^2}f''\left(\sqrt{N}M_1^{(N)}\left(N\frac{k}{N}\right))\right)p_1\left(N{+}k, Z^{(N)}\Big(\frac{k}{N}\Big)\right)\\
    &\quad\cdot\left(1{-}p_1\left(N{+}k, Z^{(N)}\Big(\frac{k}{N}\Big)\right)\right)+o(1)\\
    &\xrightarrow{N\to\infty}\int_0^t\frac{f''(M_1(s))}{2(1+s)^2}p_1(Z(s))\left(1-p_1(Z(s)\right)ds=\int_0^tL_sf(M_1(s))ds
\end{align*}
Convergence for $N\to\infty$ holds almost surely on an appropriate probability space by Skorochod's representation theorem, which implies weak convergence. 
Summing up, we have that (\ref{eq: martingaleSeq}) converges to
\begin{equation}
\label{eq: martingaleLimit}
    f(M_1(t))-f(0)-\int_0^tL_sf(M_1(s))ds
\end{equation}
for $N\to\infty$. As $f$ and $f''$ are bounded, the sequence in (\ref{eq: martingaleSeq}) is obviously uniformly integrable in $N$. Thus, \cite[Theorem 5.3]{Whitt} implies that (\ref{eq: martingaleLimit}) is a martingale as well. Moreover, the solution of the martingale problem \eqref{eq: martingaleLimit} is unique as a time-changed Brownian motion is always the unique solution if its corresponding martingale problem. Hence, $M_1$ is a time-inhomogeneous Markov-process with generator $(L_s)_{s\ge0}$.
\end{proof}

\begin{example}
\begin{enumerate}
     \item Let $F_i(k)=e^{\alpha_ik},\, \alpha_i>0,\, i\in[A] $ and suppose that $\chi_i(0)\alpha_i>\chi_j(0)\alpha_j$ for an $i\in[A]$ and all $j\ne i$. Then $M_1(t)=0$ almost surely for all $t\ge0$, since $p(x)=e^{(i)}$ for $x\in D_i$, in particular on the path of $Z$. This complies with the idea of a total monopoly described in Section \ref{sec: monopoly}.
     \item If $F_i(k)=k,\, i\in[A]$, then $Z(s)\equiv \chi(0)$ for all $s\ge0$ and $p(x)=x$ for all $x\in\Delta_{A-1}$. Hence, $\langle M_1\rangle (t)=\chi_1(0)(1-\chi_1(0))\left(1-\frac{1}{1+t}\right)$ for all $t\ge0$. Note that in this case the martingale part $M^{(N)}=\chi^{(N)}-\chi^{(N)}(0)$ encompasses the whole dynamic as $H^{(N)}(t)\equiv 0$ for all $t\ge0$.
     \item Let $F_i(k)=k^\beta,\,\chi_i(0)=\frac{1}{A}$ for all $i\in[A]$ and $\beta>0$. Since we start in a stable or unstable equilibrium point, we have $Z(t)\equiv \chi (0)$ and hence $\langle M_1\rangle (t)=\frac{A-1}{A^2}\left(1-\frac{1}{1+t}\right)$ for all $t\ge0$. In particular, $M_1$ does not depend on $\beta$.
\end{enumerate}
\end{example}

For non-linear, polynomial feedback functions and general initial market shares, the expressions for $Z$ are lengthy or even not explicit.  Figure \ref{fig: M} shows some realisations of the process $M_1$. It can be seen that the convergence of $M_1(t)$ for $t\to\infty$ is faster the faster the feedback functions grow. In the monopoly case, the variation of $M_1$ is small if $\chi(0)$ is already close to zero or one.


So far in this section, we only considered one fixed agent. Nevertheless, one can obtain an extension of Theorem \ref{thm: CLT} for all agents by a completely analogous, but lengthy argument, which we leave to the reader.

\begin{theorem}
\label{thm: CLT_A-dim.}
Suppose that the assumptions of Theorem \ref{thm: CLT} are fulfilled. Moreover, denote by $(M(t))_{t\ge0}$ an $A$-dimensional time-inhomogeneous Markov process with generator
\begin{align*}
    \tilde L_sf(x)&\coloneqq\sum_{i=1}^A\frac{ p_i(Z(s))(1- p_i(Z(s))}{2(1+s)^2}\frac{\partial^2}{(\partial x_i)^2}f(x)-\sum_{i, j=1\atop j\ne i}^A\frac{ p_i(Z(s)) p_j(Z(s))}{(1+s)^2}\frac{\partial^2}{\partial x_i \partial x_j}f(x)
\end{align*}
with $x\in\R^A$ and $M(0)=0$. Then
\begin{equation*}
    \sqrt{N}\left(M^{(N)}(\lfloor Nt\rfloor)\right)_{t\ge0}\xrightarrow{N\to\infty}( M(t))_{t\ge0}\quad\text{weakly on }\mathbb{D}([0, \infty), \mathbb{R}^A) .
\end{equation*}
\end{theorem}

The specific form of the generator is due to the conditioned covariance matrix of the increments $\xi^{(N)}$, which is for $j\ne i$:
\begin{align*}
    \E&\left[\xi_i^{(N)}(k)\xi_j^{(N)}\big|\mathcal{F}_k^{(N)}\right]=- p_i(k, \chi^{(N)}(k))\left(1- p_i(k, \chi^{(N)}(k))\right) p_j(k, \chi^{(N)}(k))\\
    &\quad- p_j(k, \chi^{(N)}(k)) p_i(k, \chi^{(N)}(k))\left(1- p_j(k, \chi^{(N)}(k))\right)\\
    &\quad+\left(1- p_i(k, \chi^{(N)}(k))-\tilde p_i(k, \chi^{(N)}(k))\right) p_i(k, \chi^{(N)}(k)) p_j(k, \chi^{(N)}(k))\\
    &=- p_i(k, \chi^{(N)}(k)) p_j(k, \chi^{(N)}(k))
\end{align*}

\begin{figure}
  \centering
  \subfloat[][$F_1(k)=F_2(k)=F_3(k)=\sqrt{k},\,\chi(0)=\left(\frac{8}{10}, \frac{1}{10}, \frac{1}{10}\right),$\\ Here $\lim_{t\to\infty}\langle  M_1\rangle_t\approx0.2474$.]{\includegraphics[width=0.5\linewidth]{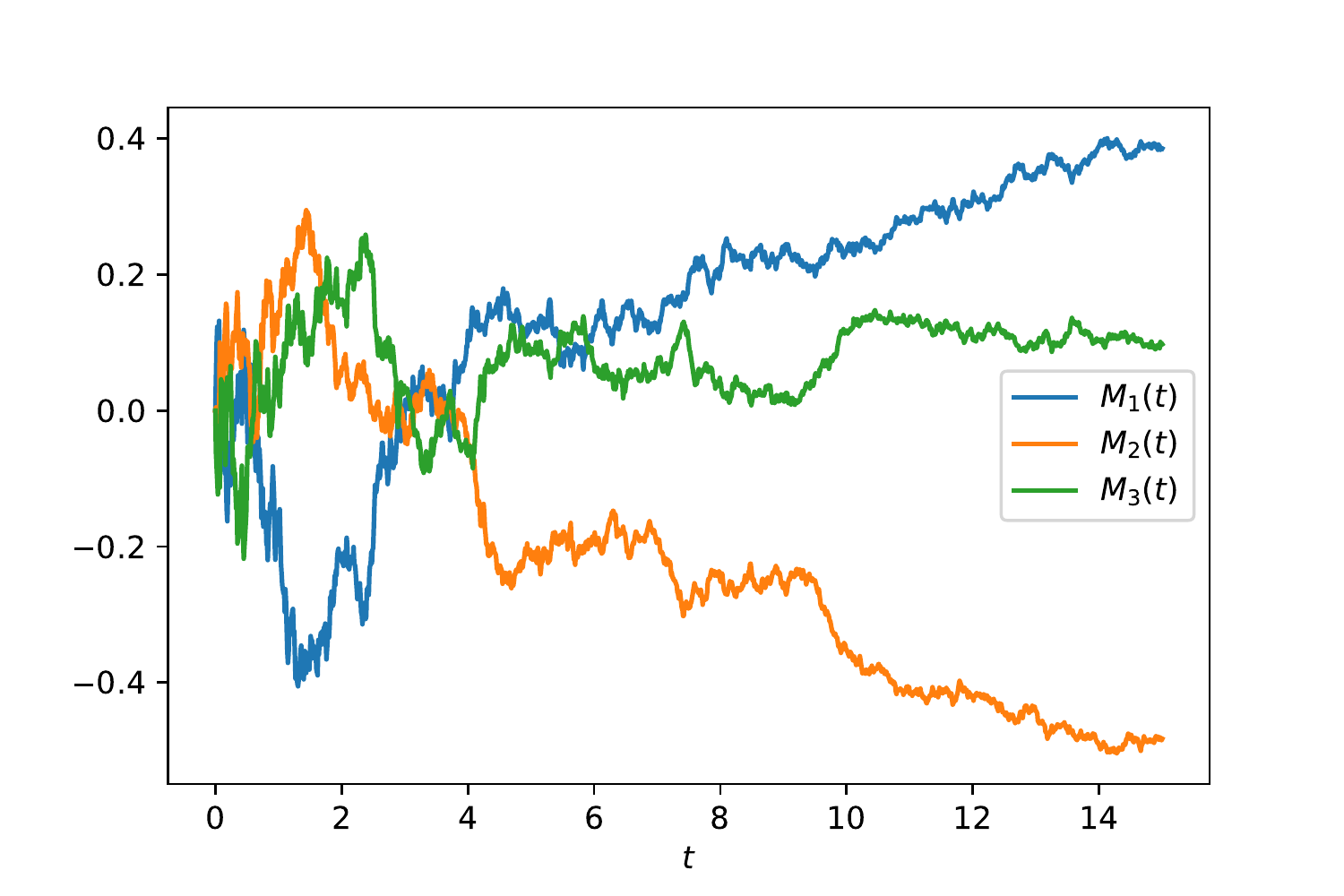}}
  \subfloat[][$F_1(k)=F_2(k)=F_3(k)=k^2,\,\chi(0)=\left(\frac{2}{5}, \frac{3}{10}, \frac{3}{10}\right)$,\\ Here $\lim_{t\to\infty}\langle M_1\rangle_t\approx0.1908$.]{\includegraphics[width=0.5\linewidth]{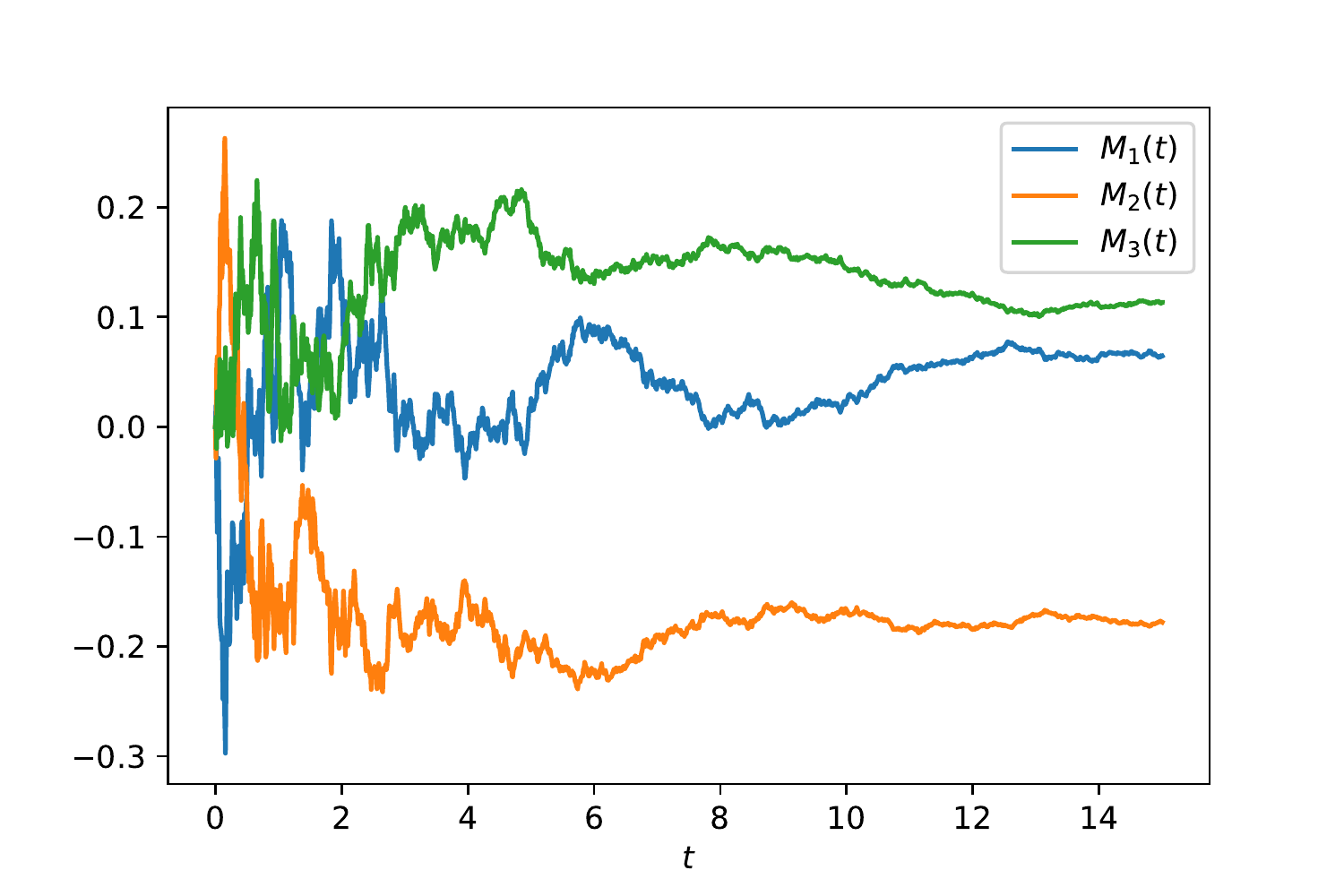}}\\
  \caption{Realisations of the process $M$ for different feedback functions and $A=3$ generated by the Euler-Maruyama method for \eqref{eq:sdes} with bandwidth $\frac{1}{100}$.}
  \label{fig: M}
\end{figure}

Alternatively, the $A$-dimensional generator $\tilde L_s$ can be rewritten as
\begin{equation*}
    \tilde L_sf(x)=\sum_{i,j=1\atop i<j}^A\frac{ p_i(Z(s)) p_j(Z(s))}{2(1+s)^2}\left(\frac{\partial}{\partial x_i}-\frac{\partial}{\partial x_j}\right)^2f(x),\quad x\in\Delta_{A-1},
\end{equation*}
where $\left(\frac{\partial}{\partial x_i}-\frac{\partial}{\partial x_j}\right)^2\coloneqq\frac{\partial^2}{(\partial x_i)^2}+\frac{\partial^2}{(\partial x_j)^2}-2\frac{\partial^2}{\partial x_i \partial x_j}$ is the second derivative along the diagonal $x_i=x_j$. From this form of the generator it is easy to see (e.g. by a coordinate transformation) that $M$ solves the system of stochastic differential equations
\begin{equation}\label{eq:sdes}
    d M_i(t)=\sum_{j\ne i}\frac{\sqrt{ p_i(Z(t)) p_j(Z(t))}}{1+t}dB_{i, j}(t)\ ,\quad i=1,\ldots ,A
\end{equation}
where $B_{i,j}$ is a standard Brownian motion, which is independent of $B_{k, l}$ if $\{i, j\}\ne\{k, l\}$ and $B_{j,i}=-B_{i, j}$ for $i\ne j$. It follows immediately that $\left(\sum_{i=1}^Ad M_i(t)\right)=0$ for all $t>0$. Hence, the sum $\sum_{i=1}^A M_i(t)=0$ is a conserved quantity. Consequently, the state space of $M$ is the tagent space $T\Delta_{A-1}$. This allows the following interpretation of the limit process $ M$: Each pair of agents exchanges mass according to a time-changed Brownian motion and the exchange of several distinct pairs of agents is independent. Figure \ref{fig: M} finally shows two simulations of the process $M$ with polynomial feedback.


\subsection{Convergence of the predictable part}\label{subsec: H}

In  order to complete the proof of Theorem \ref{cor: fclt}, let us now turn to the predictable part $H^{(N)}$ in the Doob decomposition (\ref{hmdef}), which accounts for the drift part of (\ref{eq: CLT}). It is important to recall that $H^{(N)}(\lfloor Nt\rfloor)$ is deterministic when $M^{(N)}(\lfloor Ns\rfloor)$ is given for $s\le t$. Because of that, it is possible to express the limit process of $\sqrt{N}\big(\chi(0)+H^{(N)}(\lfloor Nt\rfloor)-Z(t)\big)$ for $N\to\infty$ in terms of the limit $M$ of $\sqrt{N}M^{(N)}$. In Section \ref{sec:dynamics}, we derived that $\chi(0)+H^{(N)}$ converges to the deterministic process $Z$ for $N\to\infty$ and the following result describes the deviation under appropriate scaling.

\begin{theorem}
\label{thm: H}
Suppose that the assumptions of Theorem \ref{cor: fclt} are fulfilled. Then
\begin{equation*}
    \sqrt{N}\left(\chi(0)+H^{(N)}(\lfloor Nt\rfloor)-Z(t)\right)_{t\ge0}\xrightarrow{N\to\infty}(H(t))_{t\ge0}\quad\text{weakly on }\mathbb{D}([0, \infty), T\Delta_{A-1})\ ,
\end{equation*}
where $H$ is the solution of the system of random ordinary differential equations (RODE)
\begin{equation}
\label{eq: RODE}
    \frac{d}{dt}H(t)=\frac{DG(Z(t))}{1+t}\bullet(H(t)+M(t))\ ,\quad H(0)=0\ .
\end{equation}
\end{theorem}

Here, $DG(z)\colon T\Delta_{A-1}\to\R^A$ denotes the differential operator of $G$ at the point $z\in\Delta_{A-1}^o \subset\R^A$, i.e. $DG(z)\bullet x$ is the derivative of $G$ at $z$ in direction $x\in T\Delta_{A-1}$. Note that $H(t)$ as well as $M(t)$ (as described in the previous section) are in the tangent space $T\Delta_{A-1} \subset\R^A$ \eqref{eq:tspace}, and therefore also $H(t)+M(t)\in T\Delta_{A-1}$. If $G$ is well defined on an open neighbourhood of $\Delta_{A-1}$ in $\R^A$ (like in Example \ref{ex: negDefinit}), then $DG$ can be interpreted as the common differential matrix and $\bullet$ as the matrix-vector product.

The solution of a RODE is defined pathwise, in the sense that for any fixed realisation $\omega\in\Omega$ $M(t)=M(t, \omega)$ is a deterministic function, such that $H(t)=H(t, \omega)$ is the solution of the ordinary differential equation (\ref{eq: RODE}). Further details on the theory of RODEs can be found e.g. in \cite{Han}.

Consequently for fixed $\omega\in\Omega$,  (\ref{eq: RODE}) is a linear, time inhomogeneous ordinary differential equation, whose solution can be expressed as the matrix exponential
\begin{equation*}
    H(t)=e^{\int_0^t\frac{DG(Z(s))}{1+s}ds}\int_0^te^{-\int_0^s\frac{DG(Z(u))}{1+u}du}\frac{DG(Z(s))}{1+s}\bullet M(s)ds.
\end{equation*}

An important part of the proof of Theorem \ref{thm: H} will be the tightness of the sequence of processes $\sqrt{N}\left(\chi(0)+H^{(N)}(\lfloor Nt\rfloor)-Z(t)\right)_{t\ge0}$. For that, we bound its increments by the supremum of the martingale $M^{(N)}$.

\begin{lemma}
\label{lemma: Gronwall}
In the situation of Theorem \ref{thm: H} we have with probability one for all $0\le s<t\le T$
\begin{align*}
    \|H^{(N)}(\lfloor Nt\rfloor)-&Z(t)-H^{(N)}(\lfloor Ns\rfloor)+Z(s)\|\\
    &\le const.\left((t-s)\sup_{0\le u\le T}\|M^{(N)}(\lfloor Nu\rfloor)\|+\frac{t-s}{\sqrt{N}}+\frac{1}{N}\right),
\end{align*}
where $const.$ is a constant only depending on $G$ and $T$.
\end{lemma}

\begin{proof}
Let $L>0$ be a Lipschitz constant for $G$. We use (\ref{hmdef}) and calculate:
\begin{align*}
    &\|H^{(N)}(\lfloor Nt\rfloor)-Z(t)-H^{(N)}(\lfloor Ns\rfloor)-Z(s)\|\\
    &=\left\|\sum_{k=\lfloor Ns\rfloor}^{\lfloor Nt\rfloor-1}\frac{G(N+k, \chi^{(N)}(k))}{N+k+1}-\int_s^t\frac{G(Z(u))}{1+u}du\right\|\\
    &\le\left\|\sum_{k=\lfloor Ns\rfloor}^{\lfloor Nt\rfloor-1}\frac{G( \chi^{(N)}(k))}{N+k+1}-\int_s^t\frac{G(Z(u))}{1+u}du\right\|\\
    &\quad+\left\|\sum_{k=\lfloor Ns\rfloor}^{\lfloor Nt\rfloor-1}\frac{G(N+k, \chi^{(N)}(k))}{N+k+1}-\sum_{k=\lfloor Ns\rfloor}^{\lfloor Nt\rfloor-1}\frac{G(\chi^{(N)}(k))}{N+k+1}\right\|\\
     &\le\left\|\int_s^{t}\frac{G(\chi^{(N)}(\lfloor Nu\rfloor))-G(Z(u))}{1+u}du\right\|+\frac{const.}{N}+\frac{const.}{\sqrt{N}}\sum_{k=\lfloor Ns\rfloor}^{\lfloor Nt\rfloor-1}\frac{1}{N+k+1}\\
     &\le\int_s^{t}\frac{\|G(\chi^{(N)}(\lfloor Nu\rfloor))-G(Z(u))\|}{1+u}du+\frac{const.}{N}+const.\frac{t-s}{\sqrt{N}}\\
     &\le L\int_s^{t}\frac{\|\chi^{(N)}(\lfloor Nu\rfloor)-Z(u)\|}{1+u}du+const.\frac{t-s}{\sqrt{N}}+\frac{const.}{N}\\
     &\le L\int_s^{t}\|\chi^{(N)}(\lfloor Nu\rfloor)-Z(u)\|du+const.\frac{t-s}{\sqrt{N}}+\frac{const.}{N}\\
     &\le L\int_s^{t}\|\chi(0)+H^{(N)}(\lfloor Nu\rfloor)-Z(u)\|du+L\int_s^{t}\|M^{(N)}(\lfloor Nu\rfloor)\|du+const.\frac{t-s}{\sqrt{N}}+\frac{const.}{N}\\
     &\le L\int_s^{t}\|H^{(N)}(\lfloor Nu\rfloor)-Z(u)-H^{(N)}(\lfloor Ns\rfloor)+Z(s)\|du+const.\frac{t-s}{\sqrt{N}}+\frac{const.}{N}\\
     &\quad+L\int_s^t\|\chi(0)+H^{(N)}(\lfloor Ns\rfloor)-Z(s)\|du+L(t-s)\sup_{0\le u\le T}\|M^{(N)}(\lfloor Nu\rfloor)\|\\
      &= L\int_s^{t}\|H^{(N)}(\lfloor Nu\rfloor)-Z(u)-H^{(N)}(\lfloor Ns\rfloor)+Z(s)\|du+const.\frac{t-s}{\sqrt{N}}+\frac{const.}{N}\\
      &\quad+L(t-s)\|\chi(0)+H^{(N)}(\lfloor Ns\rfloor)-Z(s)\|+L(t-s)\sup_{0\le u\le T}\|M^{(N)}(\lfloor Nu\rfloor)\|
\end{align*}
In line 2, the second summand is of order $1/\sqrt{N}$ due to assumption (\ref{eq: condThmH}). Now Grönwall's inequality yields:
\begin{align*}
   &\|H^{(N)}(\lfloor Nt\rfloor)-Z(t)-H^{(N)}(\lfloor Ns\rfloor)+Z(s)\|\le e^{L(t-s)}\cdot\Big(const.\frac{t-s}{\sqrt{N}}+\frac{const.}{N}\\
   &\quad+L(t-s)\|\chi(0)+H^{(N)}(\lfloor Ns\rfloor)-Z(s)\|+L(t-s)\sup_{0\le u\le T}\|M^{(N)}(\lfloor Nu\rfloor)\|\Big)
\end{align*}
Repeating the same calculation with 0 in the place of s and s instead of t yields:
\begin{align*}
    \|\chi(0)+H^{(N)}(\lfloor Ns\rfloor)-Z(s)\|\le e^{Ls}\cdot\left(Ls\sup_{0\le u\le T}\|M^{(N)}(\lfloor Nu\rfloor)\|+const.\frac{s}{\sqrt{N}}+\frac{const.}{N}\right)
\end{align*}
Combining these two inequalities proves the claim.
\end{proof}

We are now ready for the proof of Theorem \ref{thm: H}.

\begin{proof}
Via \cite[Proposition VI.3.26]{Jacod}, we get tightness of \linebreak
$\left(\sqrt{N}\left(\chi(0)+H^{(N)}(\lfloor Nt\rfloor)-Z(t)\right)_{t\ge0}\right)_N$ from Lemma \ref{lemma: Gronwall} and the stochastic boundedness of the sequence $(\sqrt{N}M^{(N)}(\lfloor Nt\rfloor))_{t\ge0}$ (see proof of Lemma \ref{lemma: CLT1}). Now we show that the limit of any convergent subsequence is as desired. For simplicity of notation, assume that the sequence is convergent itself. Since Theorem \ref{thm: CLT} applies we can take an appropriate probability space $\Omega$, such that the convergence $\sqrt{N}M^{(N)}(\lfloor Nt\rfloor, \omega)\xrightarrow{N\to\infty}M(t, \omega)$ holds locally uniformly almost surely. Note that this already implies $Z^{(N)}(\omega)\xrightarrow{N\to\infty}Z$ locally uniformly. Now, fix $\omega\in\Omega$. Using (\ref{hmdef}) and the mean value theorem, we get
\begin{align*}
    &\sqrt{N}\left(\chi(0)+H^{(N)}(\lfloor Nt\rfloor)-Z(t)\right)=\sqrt{N}\left(\sum_{k=0}^{\lfloor Nt\rfloor-1}\frac{G(N+k, \chi^{(N)}(k))}{N+k+1}-\int_0^t\frac{G(Z(s))}{1+s}ds\right)\\
    &=\sqrt{N}\left(\sum_{k=0}^{\lfloor Nt\rfloor-1}\frac{G(\chi^{(N)}(k))}{N+k+1}-\int_0^t\frac{G(Z(s))}{1+s}ds+\sum_{k=0}^{\lfloor Nt\rfloor-1}\frac{G(N+k, \chi^{(N)}(k))-G(\chi^{(N)}(k))}{N+k+1}\right)\\
    &=\sqrt{N}\left(\int_0^t\frac{G( \chi^{(N)}(\lfloor Ns\rfloor))-G(Z(s))}{1+s}ds+O\left(\frac{1}{N}\right)+o\left(\frac{1}{\sqrt{N}}\right)\sum_{k=0}^{\lfloor Nt\rfloor-1}\frac{1}{N+k+1}\right)\\
    &=\sqrt{N}\int_0^t\frac{G(\chi^{(N)}(\lfloor Ns\rfloor))-G(Z(s))}{1+s}ds+o(1)\\
    &=\sqrt{N}\int_0^t\frac{DG(m^{(N)}(s))\bullet(\chi^{(N)}(\lfloor Ns\rfloor)-Z(s))}{1+s}ds+o(1)\\
    &=\int_0^tDG(m^{(N)}(s))\bullet\frac{\sqrt{N}\left((\chi(0)+H^{(N)}(\lfloor Ns\rfloor)-Z(s)+M^{(N)}(\lfloor Ns\rfloor\right)}{1+s}ds+o(1)\\
    &\xrightarrow{N\to\infty}\int_0^t\frac{DG(Z(s))\bullet(H(s)+M(s))}{1+s}ds,
\end{align*}
where $m^{(N)}(s)$ is an intermediate value between $Z(s)$ and $\chi^{(N)}(\lfloor Ns\rfloor)$. In line 3, we used assumption (\ref{eq: condThmH}) once again. The claim follows since (\ref{eq: RODE}) has a unique solution due to the Theorem of Picard-Lindelöf, and $H(t)\in T\Delta_{A-1}$ since $H^{(N)} (k)\in T\Delta_{A-1}$ for all $N\geq 1$.
\end{proof}

Figure \ref{fig: H} shows a simulation of the process $\sqrt{N}\left(\chi(0)+H^{(N)}(\lfloor Nt\rfloor)-Z(t)\right)$ for large $N$ and small $t$.  Note that the limit process \eqref{eq: RODE} has continuously differentiable paths, their regularity is equivalent to that of integrated Brownian motion. As a consequence of Proposition \ref{prop: tildeZ}, $H(t)$ is convergent for $t\to\infty$ with random limit $-\lim_{t\to\infty}M(t)$ in generic examples.
Combining Theorem \ref{thm: CLT_A-dim.} and Theorem \ref{thm: H} yields the desired central limit theorem for the difference $Z^{(N)}-Z=\chi(0)+H^{(N)}-Z+M^{(N)}$.

\begin{figure}
  \centering
    \subfloat[][$F_1(k)=F_2(k)=F_3(k)=\sqrt{k}$,\\$\chi(0)=\left(\frac{1}{10}, \frac{1}{10}, \frac{4}{5}\right)$]{\includegraphics[width=0.5\linewidth]{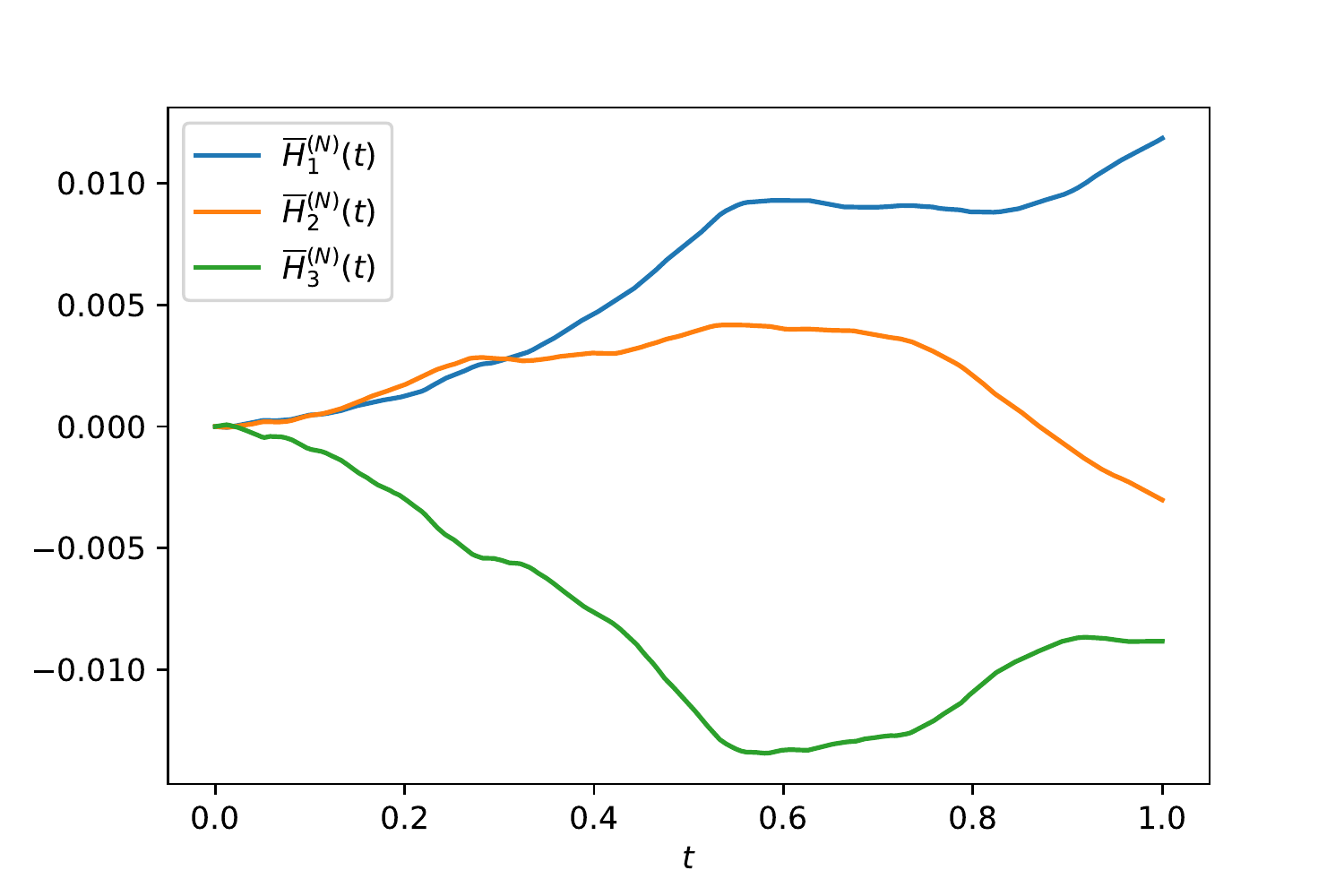}}
  \subfloat[][$F_1(k)=F_2(k)=F_3(k)=k^2$,\\$\chi(0)=\left(\frac{3}{10}, \frac{3}{10}, \frac{2}{5}\right)$]{\includegraphics[width=0.5\linewidth]{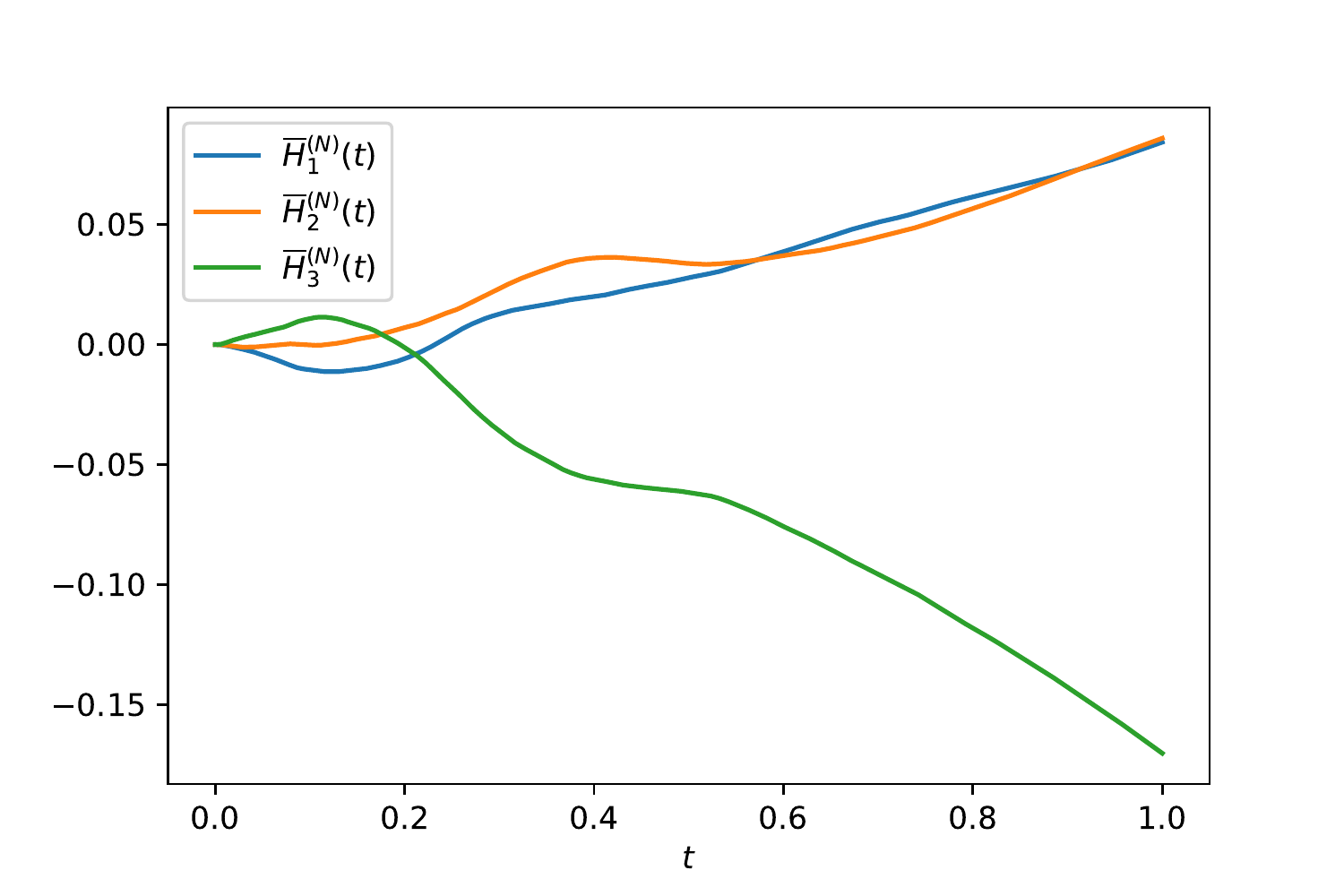}}
  \caption{The processes $\bar H^{(N)}(t)=(\bar H^{(N)}_1(t), \bar H^{(N)}_2(t), \bar H^{(N)}_3(t))\coloneqq\sqrt{N}\left(\chi(0)+H^{(N)}(\lfloor Nt\rfloor)-Z(t)\right)$ for $A=3$ and $N=100.000$.}
  \label{fig: H}
\end{figure}

\providecommand{\bysame}{\leavevmode\hbox to3em{\hrulefill}\thinspace}
\providecommand{\MR}{\relax\ifhmode\unskip\space\fi MR }
\providecommand{\MRhref}[2]{%
  \href{http://www.ams.org/mathscinet-getitem?mr=#1}{#2}
}
\providecommand{\href}[2]{#2}



\appendix

\section{Appendix}

\subsection{Functional Limit Theorems with non-linear time scale}\label{appendix: beta}

 From a stochastical point of view, the first steps of a generalized Pólya urn are of special interest because the randomness plays a significant role. In the later stages of the process, the market shares and thus the probability of winning in a certain step remain almost invariant, such that the sequence of winners $(X(n+1)-X(n))_n$ is almost independent and identically distributed for large $n$. Even in the Central Limit Theorem \ref{thm: CLT} the limiting process $M$ becomes virtually constant for large $t$. In order to particularly focus on the early stages of the process, we analyse the process $N^{1-\frac{\beta}{2}}\left(\chi^{(N)}(\lfloor N^\beta t\rfloor)\right)_{t>0}$ for large initial market size $N$ and $\beta\in(0,1)$. Recall the Doob decomposition (\ref{hmdef}) and the notations from Section \ref{sec:dynamics}.

\begin{theorem}
\label{thm: CLT2}
Suppose that the assumptions of Theorem \ref{thm: CLT} are fulfilled and denote by $(B_t )_{t\geq 0}$ a standard Brownian motion. Then for any $\beta\in(0, 1)$ we have weak convergence to a Brownian motion on $\mathbb{D}([0, \infty)$:
\begin{equation*}
    N^{1-\frac{\beta}{2}}\left(M_1^{(N)}(\lfloor N^\beta t\rfloor)\right)_{t\ge0}\xrightarrow{N\to\infty}\sqrt{p_1(\chi(0))(1-p_1(\chi(0)))}\, (B_t )_{t\geq 0}
\end{equation*}
\end{theorem}

\begin{proof}
We will only sketch the proof as it is quite analogous to the proof of Theorem \ref{thm: CLT}. We use the tightness given by Lemma \ref{lemma: CLT1} and assume that the sequence $ N^{1-\frac{\beta}{2}}\left(M_1^{(N)}(\lfloor N^\beta (\cdot)\rfloor)\right)_N$ converges to a process $\hat M_1$. Then we take a smooth test-function $f\colon\R\to\R$ with compact support and consider the martingales
\begin{equation*}
    f\left(N^{1-\frac{\beta}{2}}M_1^{(N)}(\lfloor N^\beta t\rfloor)\right)-f(0)-\sum_{k=0}^{\lfloor N^\beta t\rfloor-1}\E\left[f\left(N^{1-\frac{\beta}{2}}M_1^{(N)}(k+1)\right)-f\left(N^{1-\frac{\beta}{2}}M_1^{(N)}(k)\right)\big|\mathcal{F}_k^{(N)}\right].
\end{equation*}
Then we know that $f\left(N^{1-\frac{\beta}{2}}M_1^{(N)}(\lfloor N^\beta t\rfloor)\right)$ converges to $f\left(\hat M_1(t)\right)$ and via Lemma \ref{lemma: CLT2} we get:
\begin{align*}
    &\sum_{k=0}^{\lfloor N^\beta t\rfloor-1}\E\left[f\left(N^{1-\frac{\beta}{2}}M_1^{(N)}(k+1)\right)-f\left(N^{1-\frac{\beta}{2}}M_1^{(N)}(k)\right)\big|\mathcal{F}_k^{(N)}\right]\\
    &=\sum_{k=0}^{\lfloor N^\beta t\rfloor-1}\frac{N^{2-\beta}}{2(N+k+1)^2}f''\left(N^{1-\frac{\beta}{2}}M_1^{(N)}(k)\right)p_1(N+k, \chi^{(N)}(k))\left(1-p_1(N+k, \chi^{(N)}(k))\right)\\
    &=\sum_{k=0}^{\lfloor N^\beta t\rfloor-1}\frac{f''\left(N^{1-\frac{\beta}{2}}M_1^{(N)}\left(N^\beta\frac{k}{N^\beta}\right)\right)}{2N^\beta(1+\frac{k}{N}+\frac{1}{N})^2}p_1\left(N+k, Z^{(N)}\left(\frac{k}{N}\right)\right)\left(1-p_1\left(N+k, Z^{(N)}\left(\frac{k}{N}\right)\right)\right)\\
    &\xrightarrow{N\to\infty}\int_0^t\frac{f''(\hat M_1(s))}{2}p_1(Z(0))\left(1-p_1(Z(0)\right)ds\ ,
\end{align*}
where we have used $\beta <1$ and $k/N\to 0$ in the last step. This implies that $\hat M_1$ is a Markov process with generator $p_1(Z(0))\left(1-p_1(Z(0)\right)\frac{f''}{2}$. Hence, $\hat M_1$ is the desired Brownian motion.
\end{proof}

Note that Theorem \ref{thm: CLT2} is consistent with Theorem \ref{thm: CLT} for small $t$. This limiting Brownian motion can be understood as a consequence of Donsker's invariance principle, since the shares do barely change at the beginning of the process for large initial values. Again, a straight forward extension to higher dimensions is possible.

\begin{theorem}\label{thm: CLT2MultiDim}
Suppose that the assumptions of Theorem \ref{thm: CLT} are fulfilled and let  $\beta\in(0, 1)$. Then the sequence of processes $ N^{1-\frac{\beta}{2}}\left(M^{(N)}(\lfloor N^{\beta} t\rfloor)\right)_{t\ge0}$ converges for $N\to\infty$ to a time-homogeneous Markov process with generator
\begin{equation*}
    \hat Lf(x)\coloneqq\frac{1}{2}\sum_{i, j=1\atop i<j}^A p_i(Z(0)) p_j(Z(0))\left(\frac{\partial}{\partial x_i}-\frac{\partial}{\partial x_j}\right)^2f(x),\quad x\in\R^A
\end{equation*}
weakly on $\mathbb{D}([0, \infty), T\Delta_{A-1})$.
\end{theorem}

As in Subsection \ref{subsec: martingale}, the limit process can be interpreted as independent exchanges of mass between pairs of agents according to a Brownian motion. 

We already know from Theorem \ref{thm: dynamic} that $\chi^{(N)}(\lfloor N^\beta t\rfloor)$ converges to $\chi(0)$ for $N\to\infty$, when $\beta<1$. Moreover, Theorem \ref{thm: CLT2} states, that the process $\left(M^{(N)}(\lfloor N^\beta t\rfloor)\right)_{t\ge0}$ converges to zero at rate $N^{1-\frac{\beta}{2}}$. In addition, it follows from
\begin{align*}
&N^{1-\beta}H^{(N)}(\lfloor N^\beta t\rfloor)=N^{1-\beta}\sum_{k=0}^{\lfloor N^\beta t\rfloor-1}\frac{G(N+k, \chi^{(N)}(k))}{N+k+1}\\
&=\sum_{k=0}^{\lfloor N^\beta t\rfloor-1}\frac{1}{N^\beta}\cdot\frac{G(N+k, \chi^{(N)}(N^\beta\cdot\frac{k}{N^\beta}))}{1+\frac{k}{N}+\frac{1}{N}}\xrightarrow{N\to\infty}\int_0^t G(\chi(0))du =G(\chi(0))t,
\end{align*}
that $\left(H^{(N)}(\lfloor N^\beta t\rfloor)\right)_{t\ge0}$ converges to $\left(G(\chi(0))t\right)_{t\ge0}$ at rate $N^{1-\beta}$, which immediately implies the following law of large numbers.

\begin{corollary}\label{cor: LLNbeta}
    Under the assumptions of Theorem \ref{thm: CLT2} we have
    \begin{equation*}
    N^{1-\beta}\left(\chi^{(N)}(\lfloor N^\beta t\rfloor)-\chi(0)\right)_{t\ge0}\xrightarrow{N\to\infty}(G(\chi(0))t)_{t\ge0}\quad\text{weakly on }\mathbb{D}([0, \infty), \mathbb{R}^A)\ .
\end{equation*}
\end{corollary}

Combining these results for an analysis of the deviations of $\chi^{(N)}(\lfloor N^\beta t\rfloor)$ requires further distinction of $\beta$ as specified in the following functional CLT.

\begin{corollary}\label{cor: CLTbeta}
    Let $\beta\in(0, 1)$ and $\gamma>0$ as specified below. Suppose that 
\begin{equation*}
    \lim_{k\to\infty}k^\gamma\sup_{x\in\Delta_{A-1}}\|G(k, x)-G(x)\|=0.
\end{equation*}
Moreover, let $G$ be continuously differentiable. Then
\begin{equation*}
    N^{\gamma}\left(N^{1-\beta}\left(\chi^{(N)}(\lfloor N^\beta t\rfloor)-\chi(0)\right)-G(\chi(0))t\right)_{t\ge0}\xrightarrow{N\to\infty}(\hat Z(t))_{t\ge0}\quad\text{weakly on }\mathbb{D}([0, \infty), \mathbb{R}^A)\ ,
\end{equation*}
where the limiting process $\hat Z$ is defined as follows:
\begin{enumerate}
    \item For $\beta>\frac23$ set $\gamma=1-\beta$. Then:
    $$\hat Z(t)=\frac12 DG(\chi(0))G(\chi(0)) t^2.$$
    \item For $\beta=\frac23$ set $\gamma=\frac13$. Then
    $$\hat Z(t)=\frac12 DG(\chi(0))G(\chi(0)) t^2+\hat M(t),$$
    where $\hat M$ is the limiting process from Theorem \ref{thm: CLT2MultiDim}.
    \item For $\beta<\frac23$ set $\gamma=\frac{\beta}{2}$. Then $\hat Z=\hat M$.
\end{enumerate}
\end{corollary}

\begin{proof}
We only sketch the proof as it is widely analogous to the proof of Theorem \ref{thm: H}. Again, we use the decomposition (\ref{hmdef}) and rephrase as follows:
\begin{align*}
&N^{\gamma}\left(N^{1-\beta}\left(\chi^{(N)}(\lfloor N^\beta t\rfloor)-\chi(0)\right)-G(\chi(0))t\right)\\
&=N^{\gamma}\left(N^{1-\beta}H^{(N)}(\lfloor N^\beta t\rfloor)-G(\chi(0))t\right)+N^{1-\beta+\gamma}M^{(N)}(\lfloor N^\beta t\rfloor)\\
&=N^{\gamma}\int_0^tG(\chi^{(N)}(N^\beta u)-G(\chi(0)) du+o(1)+N^{1-\beta+\gamma}M^{(N)}(\lfloor N^\beta t\rfloor)\\
&=N^{\gamma}\int_0^tDG(\chi(0))\bullet\left(\chi^{(N)}(N^\beta u)-\chi(0)\right)du+o(1)+N^{1-\beta+\gamma}M^{(N)}(\lfloor N^\beta t\rfloor)\\
&=N^{\gamma+\beta-1}DG(\chi(0))\bullet\int_0^tN^{1-\beta}\left(\chi^{(N)}(N^\beta u)-\chi(0)\right)du+o(1)+N^{1-\beta+\gamma}M^{(N)}(\lfloor N^\beta t\rfloor)
\end{align*}
Finally, the claims follow via Theorem \ref{thm: CLT2MultiDim} and Corollary \ref{cor: LLNbeta}.
\end{proof}

\begin{figure}
  \centering
    \subfloat[][$F_1(k)=F_2(k)=F_3(k)=\sqrt{k}$,\\$\chi(0)=\left(\frac{1}{10}, \frac{1}{10}, \frac{4}{5}\right)$]{\includegraphics[width=0.5\linewidth]{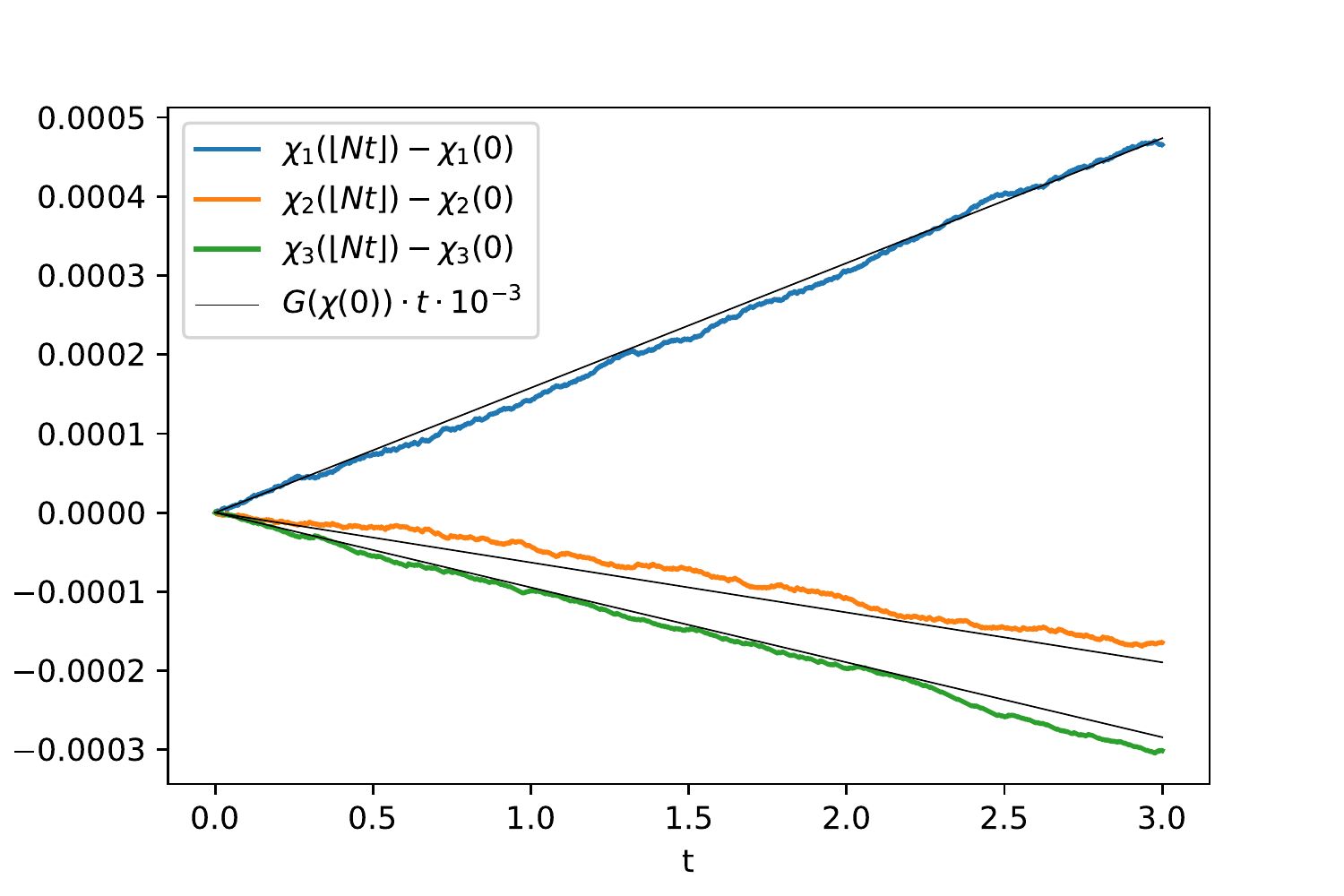}}
  \subfloat[][$F_1(k)=F_2(k)=F_3(k)=k^2$,\\$\chi(0)=\left(\frac{3}{10}, \frac{3}{10}, \frac{2}{5}\right)$]{\includegraphics[width=0.5\linewidth]{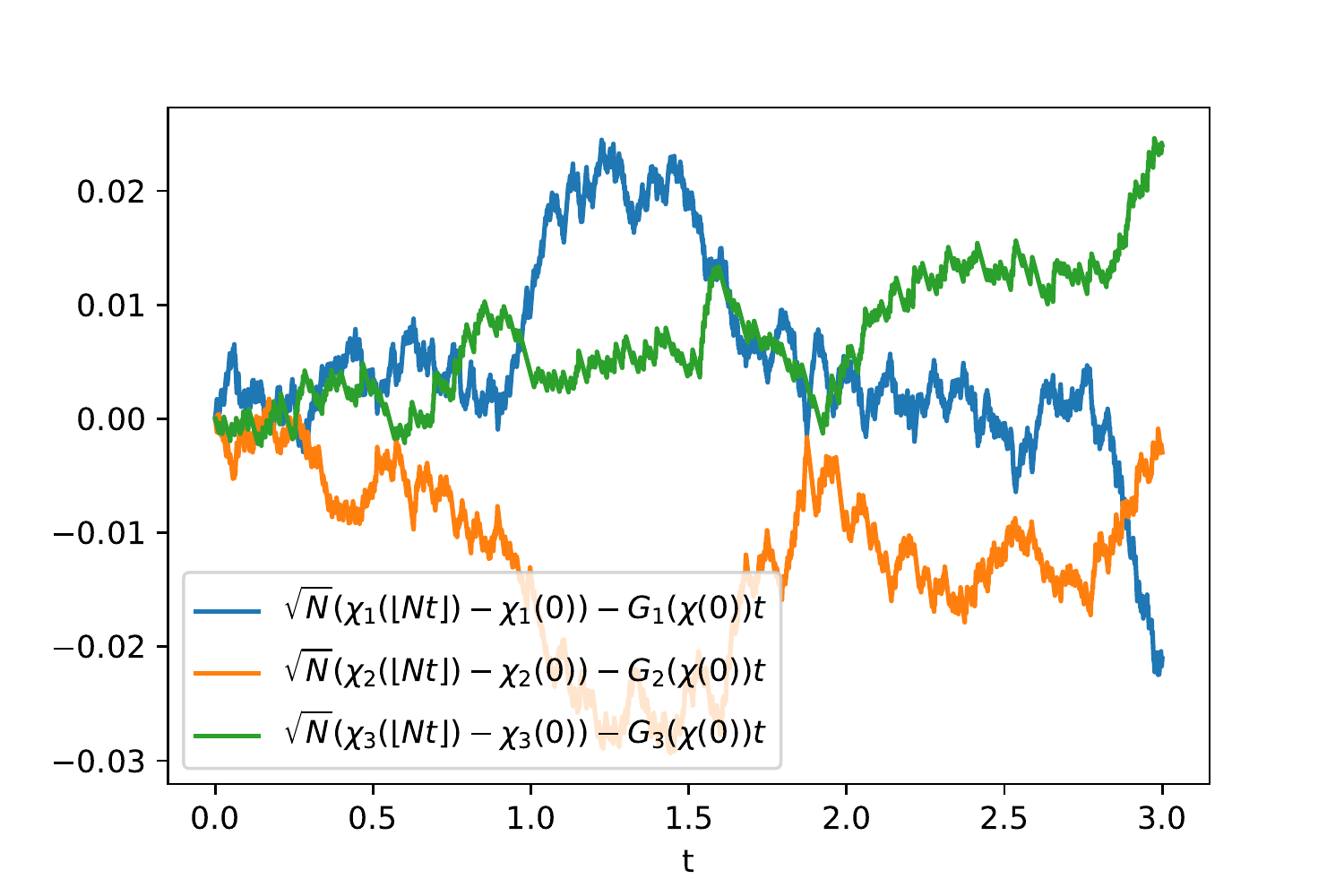}}
  \caption{Simulation of the processes $\chi^{(N)}(\lfloor N^\beta t\rfloor)-\chi(0)$ and $N^{1-\beta}\left(\chi^{(N)}(\lfloor N^\beta t\rfloor)-\chi(0)\right)-G(\chi(0))t$ for $\beta=\frac12$ and $N=10^6$. We took $A=3$, $F_1(k)=F_2(k)=F_3(k)=k^2$ and $\chi(0)=(0.5, 0.3, 0.2)$.}
  \label{fig: LLNbeta}
\end{figure}

The assumptions of Theorem $\ref{cor: CLTbeta}$ are satisfied e.g. for $F_i(k)=\alpha_i k^\beta$. To sum up, in the limit $N\to\infty$ the process $\chi^{(N)}(\lfloor N^\beta t\rfloor)$ stays at $\chi(0)$ for all time $t$. After scaling, Corollary \ref{cor: LLNbeta} reveals a linear drift into direction $G(\chi(0))$. The fluctuations around this linear drift can itself be described by a random SDE for $\beta\le\frac23$ and by a deterministic ODE for $\beta>\frac23$, since second order terms dominate the randomness for too large $\beta$. These findings are illustrated by Figure \ref{fig: LLNbeta}.

\subsection{Exponentially decreasing feedback}\label{sec: expDecreasingF}

Based on an example, this supplemental section discusses the long-time limits of a Pólya urn with exponentially decreasing feedback, since this case is not covered by our previous results.

\begin{example}\label{example: expfallend}
Let $A=2$ and $F_i(k)=\alpha_ie^{-\beta_ik}, \alpha_i, \beta_i>0, i=1, 2$. As explained in detail in Section \ref{sec:dynamics}, we can write
\begin{equation*}
\chi_1(n)=\chi_1(0)+H_1(n)+M_1(n)\quad\text{for }n\geq 0\ ,
\end{equation*}
where $(M_1(n))_{n\in\N_0}$ is an almost sure convergent martingale and 
\begin{equation*}
H_1(n)\coloneqq\sum_{k=0}^{n-1}\frac{G_1(N+k, \chi_1(k))}{N+k+1}
\end{equation*}
is predictable with $G_1(k, x)\coloneqq p_1(k, (x, 1-x))-x,\,x\in(0, 1)$ given by centered transition probabilities \eqref{eq:transpr}. In the case of exponentially decreasing feedback, we have the following convergence:
\begin{equation}
\label{eq: convG}
G_1(k, x)\xrightarrow{k\to\infty}G_1(x)\coloneqq\begin{cases}
1-x, & \text{if } x\beta_1<(1-x)\beta_2\\
\frac{\alpha_1}{\alpha_1+\alpha_2}-x, & \text{if } x\beta_1=(1-x)\beta_2\\
-x, &\text{otherwise} \end{cases}
\end{equation}
The convergence is locally uniform in $(0,1)$ apart from the point $x=x_0\coloneqq\frac{\beta_2}{\beta_1+\beta_2}$. Take $\epsilon>0$. For large enough $k$, $G_1(k, (\cdot))$ is sufficiently close to $G_1$ outside an $\epsilon$-neighborhood of $x_0$. If for a large $n$, $|\chi_1(n)-x_0|>\epsilon$, then the process $(\chi_1(n))_n$ enters the $\epsilon$-neighborhood of $x_0$ in finite time because of the convergence of the martingale. As the same holds for $\epsilon/2$ instead of $\epsilon$, we get that the process leaves this $\epsilon$-neighborhood only finitely often. This yields
\begin{equation*}
\chi_1(n)\xrightarrow{n\to\infty}x_0\quad\text{almost surely.}
\end{equation*}
Thus, the limit is not only independent of the initial market shares, but also of the fitness-parameters $\alpha_i$ (in contrast to polynomially decreasing feedback). Note that these findings are consistent with Corollary \ref{cor: main4}, i.e. (\ref{eq: limshare}) still holds. Because of the independence property in the exponential embedding in Section \ref{sec: model}, this can easily be extended to general A. For different (at least) exponentially decreasing feedback, we basically only need a convergence as in (\ref{eq: convG}) for an analogous result.
\end{example}

Remarkably, Example \ref{example: expfallend} reveals the following behavioural difference between exponentially decreasing and polynomial feedback. Suppose that there are agents $i, j$ such that
$$\lim_{k\to\infty}\frac{F_i(k)}{F_j(k)}=0.$$
Then agent $i$ is marginalized, i.e. $\lim_{n\to\infty}\chi_i(n)=0$, if $F_i$ satisfies (\ref{eq: sublin+}), in particular if $F_i(k)=\alpha_ik^{\beta_i}$ for $\beta_i<1$. On the other hand, for exponentially decreasing feedback like in Example \ref{example: expfallend}, we might still have $\lim_{n\to\infty}\chi_i(n)>0$.

\end{document}